\documentclass[reqno,a4paper,twoside]{amsart}

\usepackage[a4paper,margin=0.65in,footskip=0.25in]{geometry}

\usepackage{graphicx}
\usepackage{enumitem}

\setenumerate{label=\textnormal{(\arabic*)}}

\allowdisplaybreaks[4]

\usepackage{amsmath,amssymb,dsfont,verbatim,bm,mathrsfs,stmaryrd,mathabx}

\usepackage{mathtools}
\usepackage[raggedright]{titlesec}

\usepackage[utf8]{inputenc}
\usepackage[T1]{fontenc}

\usepackage[x11names,dvipsnames]{xcolor}
	\definecolor{egyptianblue}{rgb}{0.06, 0.2, 0.65}
	\definecolor{green(ncs)}{rgb}{0.0, 0.62, 0.42}

\usepackage{hyperref}
\hypersetup{
	colorlinks=true,
	citecolor=green(ncs),
	linkcolor= egyptianblue,
	filecolor=magenta,
	urlcolor=cyan,
	unicode=true,
	pdfpagemode=FullScreen,
}

\titleformat{\chapter}[display]
{\normalfont\huge\bfseries}{\chaptertitlename\\thechapter}{20pt}{\Huge}
\titleformat{\section}
{\normalfont\Large\bfseries\center}{\thesection}{1em}{}
\titleformat{\subsection}
{\normalfont\large\bfseries}{\thesubsection}{1em}{}
\titleformat{\subsubsection}
{\normalfont\normalsize\bfseries}{\thesubsubsection}{1em}{}
\titleformat{\paragraph}[runin]
{\normalfont\normalsize\bfseries}{\theparagraph}{1em}{}
\titleformat{\subparagraph}[runin]
{\normalfont\normalsize\bfseries}{\thesubparagraph}{1em}{}

\titlespacing*{\chapter} {0pt}{50pt}{40pt}
\titlespacing*{\section} {0pt}{3.5ex plus 1ex minus .2ex}{2.3ex plus .2ex}
\titlespacing*{\subsection} {0pt}{3.25ex plus 1ex minus .2ex}{1.5ex plus .2ex}
\titlespacing*{\subsubsection}{0pt}{3.25ex plus 1ex minus .2ex}{1.5ex plus .2ex}
\titlespacing*{\paragraph} {0pt}{3.25ex plus 1ex minus .2ex}{1em}
\titlespacing*{\subparagraph} {\parindent}{3.25ex plus 1ex minus .2ex}{1em}

\usepackage{float, subfig}
\usepackage{amsrefs}

\usepackage{titletoc}

\contentsmargin{2.55em}
\dottedcontents{section}[1.5em]{}{2em}{1pc}
\dottedcontents{subsection}[4.35em]{}{2.8em}{1pc}
\dottedcontents{subsubsection}[7.6em]{}{3.2em}{1pc}
\dottedcontents{paragraph}[10.3em]{}{3.2em}{1pc}

\usepackage[columns=3,rule=0pt]{idxlayout}
\setindexprenote{For convenience of the reader we list below almost all symbols used in this work, indicating the page in which each of them is introduced. We only except those that are extremely standard.}

\makeindex

\newtheorem{theorem}{Theorem}[section]

\newtheorem{proposition}[theorem]{Proposition}
\newtheorem{corollary}[theorem]{Corollary}

\theoremstyle{definition}
\newtheorem{definition}[theorem]{Definition}

\newtheorem{notation}[theorem]{Notation}
\newtheorem{example}[theorem]{Example}

\theoremstyle{remark}
\newtheorem{remark}[theorem]{Remark}

\usepackage{tikz}
\usetikzlibrary{matrix,shapes}
\usetikzlibrary{cd}
\tikzcdset{diagrams={nodes={inner sep=2pt}}}
\usetikzlibrary{arrows}
\tikzcdset{arrow style=tikz, diagrams={>={Latex[round,length=3pt,width=2pt,inset=1.75pt]}}}

\DeclareMathOperator{\End}{End}

\DeclareMathOperator{\Ide}{Id}
\DeclareMathOperator{\Ext}{Ext}
\DeclareMathOperator{\EXT}{EXT}

\DeclareMathOperator{\Soc}{Soc}

\DeclareMathOperator{\ima}{Im}

\DeclareMathOperator{\Ho}{H}

\DeclareMathOperator{\Hom}{Hom}

\DeclareMathOperator{\Tot}{Tot}

\DeclareMathOperator{\sh}{sh}

\DeclareMathOperator{\sg}{sg}

\DeclareMathOperator{\ho}{H}

\newcommand{\xcirc}{\hspace{0pt}}

\newcommand{\ov}{\overline}
\newcommand{\ot}{\otimes}
\newcommand{\wh}{\widehat}
\newcommand{\wt}{\widetilde}

\numberwithin{equation}{section}

\DeclareMathAlphabet{\mathpzc}{OT1}{pzc}{m}{it}

\usepackage{mathtools}
\newtagform{brackets}{[}{]}
\usetagform{brackets}

\let\origmaketitle\maketitle
\def\maketitle{
	\begingroup
	\def\uppercasenonmath##1{} 
	\let\MakeUppercase\relax 
	\origmaketitle
	\endgroup
}

\usepackage{fancyhdr}
\usepackage[foot]{amsaddr}

\usepackage{adjustbox}

\begin{document}
	
\title{\large Extensions of a family of Linear Cycle Sets whose underlying group is a product of two cyclic groups}

\author[Jorge A. Guccione]{Jorge A. Guccione$^{1}$}
\address{$^{1}$ Pontificia Universidad Cat\'olica del Per\'u, Secci\'on Matem\'aticas, PUCP, Avenida Universitaria~1801, San Miguel, Lima~32, Per\'u.}
\email{vander@dm.uba.ar}
	
\author[Jorge A. Guccione]{Juan J. Guccione$^{1}$}
\email{jjguccione@gmail.com}
	
\author[Christian Valqui]{Christian Valqui$^{1}$}
\email{cvalqui@pucp.edu.pe}

\thanks{This study was funded by the Vicerrectorado de Investigaci\'on (VRI) at the Pontificia Universidad Cat\'olica del Per\'u through grant CAP 2025 - PI1273}

	
\begin{abstract}
We use the linear cycle set cohomology in order to compute extensions of a non trivial linear cycle sets whose underlying additive group is a product of two cyclic groups.
\end{abstract}

\maketitle
	
\tableofcontents
	
\section*{Introduction}	

The study of set-theoretic solutions of the Yang--Baxter equation has developed into a rich interaction between algebra, combinatorics, and group theory. Following the systematic study of such solutions initiated by Etingof, Schedler, and Soloviev~\cite{ESS}, several algebraic structures associated with a solution have been investigated. In particular, Rump introduced braces in the study of involutive solutions and established fundamental connections between braces and the Yang--Baxter equation~\cites{R1,R2}. Subsequent work has explored the structure and classification of solutions through associated groups, skew polynomial rings, and matched pairs~\cites{GI1,GIVB,GIM,CJO}. The notion of retractability provided an important tool for analyzing the complexity of solutions and led to structural results for irretractable solutions~\cites{CJO,BCJO}. Other approaches have connected set-theoretic solutions with RC-calculus and Garside structures~\cites{De1,DDM}. These developments have contributed to a broad algebraic framework in which set-theoretic solutions can be studied through associated algebraic objects, among which braces and linear cycle sets play a particularly important role.

Linear cycle sets  are  closely related to left braces, and provide an algebraic framework in which the additive and multiplicative structures occurring in braces can be studied simultaneously. On one hand, extensions of braces were studied by Bachiller in~\cite{B}. On the other hand, extensions of linear cycle sets and their classification by cohomological methods have been developed in~\cite{GGV} and~\cite{LV}. In particular, when the kernel of the extension is a trivial linear cycle set, equivalence classes of extensions are described by a second cohomology group.

The purpose of this paper is to study extensions of a particular family of linear cycle sets whose underlying additive group is a product of two cyclic groups. More precisely, let $d_1,d_2$ be positive integers such that $d_1\mid d_2$, and let $\lambda\in\mathds{Z}_{d_1}$. We consider the linear cycle set
$$
H^{d_1d_2}_{\lambda}\coloneqq \bigl(\mathds{Z}_{d_1}\oplus\mathds{Z}_{d_2},\cdot\bigr),
$$
where
$$
(x,y)\cdot(u,v)\coloneqq (u-\lambda yv,v).
$$
We first determine when two members of this family are isomorphic and analyze the adjoint group associated with $H^{d_1d_2}_{\lambda}$. A relevant feature of this analysis is the condition
$$
d_1\mid \lambda\binom{d_2}{2},
$$
which determines whether the adjoint group can be identified with the underlying additive group $\mathds{Z}_{d_1}\oplus\mathds{Z}_{d_2}$.

A fundamental problem in the theory of linear cycle sets is the description and classification of extensions. Given a linear cycle set $H$ and an abelian group $I$ endowed with the trivial structure, extensions of $H$ by~$I$ can be described in terms of suitable actions $\blackdiamond\colon H \times I \to I$ and $\Yleft\colon I \times H \to I$, together with cohomological data arising from a filtered cochain complex. In this framework, it was proved that equivalence classes of extensions with fixed actions~$\blackdiamond$ and~$\Yleft$, correspond bijectively to the second cohomology group $H^2_{\blackdiamond,\Yleft}(H,I)$, reducing the classification problem to the determination of admissible actions and the computation of this group. Our main objective is to carry out this task for $H\coloneqq H^{d_1d_2}_{\lambda}$. For this we first characterize the actions
$$
\blackdiamond\colon H^{d_1d_2}_{\lambda}\times I\longrightarrow I \qquad\text{and}\qquad \Yleft\colon I\times H^{d_1d_2}_{\lambda}\longrightarrow I
$$
in terms of endomorphisms of $I$. The resulting descriptions depend on the arithmetic of $d_1$, $d_2$ and $\lambda$ (in particular on the corresponding $2$-adic valuations).

A second motivation for this work is methodological. In~\cite{GGV2}, explicit methods were developed for computing the cohomology of linear cycle sets whose adjoint group is finite abelian. These methods give concrete descriptions of the cochain representatives and of the maps whose kernels and images determine the second cohomology. The present paper provides a nontrivial family in which these methods can be applied in detail. Thus, besides obtaining cohomological information for the family $H^{d_1d_2}_{\lambda}$, we aim to illustrate how the general computational procedure of~\cite{GGV2} can be used in a specific family where the structure of the linear cycle set and its adjoint group can be described explicitly. More precisely, after recalling the necessary definitions and the cohomological description of extensions, we study the linear cycle sets $H^{d_1d_2}_{\lambda}$ and their adjoint groups. We then determine the possible actions on $I$ and use them to express the relevant second cohomology in terms of explicit homomorphisms
$$
T_1,\quad T_2,\quad T_3\quad\text{and}\quad S.
$$
These maps are obtained from the general construction of~\cite{GGV2}; the subsequent sections specialize their formulas to the cases under consideration. When
$$
d_1\mid\lambda\binom{d_2}{2},
$$
we obtain a general description under suitable hypotheses and study several important cases. When this divisibility condition does not hold, the same computational framework is applied to several families distinguished by the $2$-adic valuations of $d_1$ and $d_2$ and by the parity of $\lambda$.

The paper is organized as follows. Section~1 recalls the necessary facts about braces, linear cycle sets, extensions, and the cohomology of linear cycle sets. In Section~2 we introduce the family $H^{d_1d_2}_{\lambda}$ and study its basic properties and adjoint group. Section~3 determines the actions $\blackdiamond$ and $\Yleft$ on an abelian group $I$. Subsection~4.1 applies the computational machinery of~\cite{GGV2} to compute the relevant cohomology when $d_1\mid\lambda\binom{d_2}{2}$. Finally, Subsection~4.2 treats several examples in the complementary case
$d_1\nmid\lambda\binom{d_2}{2}$.

\section{Preliminaries}\label{section:Preliminaries}

In this paper, we work within the category of abelian groups. All maps are understood to be $\mathbb{Z}$-linear, the symbol~$\ot$ denotes the tensor product $\ot_{\mathbb{Z}}$, and $\Hom$ refers to $\Hom_{\mathbb{Z}}$. Moreover we shall employ the Iverson bracket that maps any statement $P$ to the function $[P]$, of the free variables of $P$, defined by
$$
[P]\coloneqq \begin{cases} 1 & \text{if $P$ is true,}\\ 0 & \text{if $P$ is false.} \end{cases}
$$

\subsection[Braces and linear cycle sets]{Braces and linear cycle sets}\label{Braces and linear cycle sets}
A {\em left brace} is an abelian group $(A,+)$ endowed with an additional group operation $(a,b)\mapsto a\times b$, such that
\begin{equation}\label{compatibilidad braza}
a\times (b+c)= a\times b+a\times c-a\quad\text{for all $a,b,c\in A$.}
\end{equation}
The two group structures necessarily share the same neutral element, denoted by $0$. Of course, the additive inverse of $a$ is denoted by $-a$, while its multiplicative inverse is denoted by $\cramped{a^{\times^{-1}}}$. In particular $\cramped{0^{\times^{-1}}} = 0$. Let $A$ and $B$ be left braces. A map $f\colon A\to B$ is an homomorphism of left braces if
$$
f(a+a') = f(a)+f(a')\quad\text{and}\quad f(a\times a') = f(a)\times f(a')\qquad\text{for all $a,a'\in A$.}
$$
A {\em linear cycle set} is an abelian group $(A,+)$ endowed with a binary operation $\cdot$, such that
\begin{equation}\label{compatibilidades cdot suma y suma cdot}
a\cdot (b+c)=a\cdot b + a\cdot c\quad\text{and}\quad (a+b)\cdot c = (a\cdot b)\cdot (a\cdot c) \qquad\text{for all $a,b,c\in A$,}
\end{equation}
and moreover, for each $b\in A$, the left translation $a\mapsto b\cdot a$, is bijective. It is well known that each linear cycle set is a cycle set. That is:
\begin{equation}
(a\cdot b)\cdot (a\cdot c)=(b\cdot a)\cdot (b\cdot c)\quad\text{for all $a,b,c\in A$}\label{condicion de cycle set}.
\end{equation}
Let $A$ and $B$ be linear cycle sets. A map $f\colon A\to B$ is an homomorphism of linear cycle sets if
$$
f(a+a') = f(a)+f(a')\quad\text{and}\quad f(a\cdot a')=f(a)\cdot f(a')\qquad\text{for all $a,a'\in A$.}
$$
The notions of left braces and linear cycle sets were introduced by Rump, who proved that they are equivalent via the relations
\begin{equation}
a\cdot b = \cramped{a^{\times^{-1}}}\times (a+b)\quad \text{and}\quad a\times b = a+{}^ab,\label{equivalencia linear cycle set braza}
\end{equation}
where the map $b\mapsto {}^ab$ is the inverse of the map $b\mapsto a\cdot b$.

\smallskip

Each additive abelian group $A$ is a linear cycle set via $a\cdot b\coloneqq b$. These ones are the so called {\em trivial linear cycle sets}. Note that by~\eqref{equivalencia linear cycle set braza}, a linear cycle set $A$ is trivial if and only if $a\times b = a+b$, for all $a,b\in A$.

\smallskip

Let $A$ be a linear cycle set. An {\em ideal} of $A$ is  a subgroup $I$ of $(A,+)$ such that $a\cdot y \in I$ and $y\cdot a - a\in I$, for all~$a\in A$ and $y\in I$. An ideal $I$ of $A$ is called a {\em central ideal} if $a\cdot y=y$ and $y\cdot a=a$, for all $a\in A$ and $y\in I$. The {\em socle} of $A$ is the ideal $\Soc(A)$, of all $y\in A$ such that $y\cdot a = a$, for all $a\in A$. The {\em center} of $A$ is the set $Z(A)$, of all~$y\in \Soc(A)$ such that $a\cdot y = y$, for all $a\in A$. Thus, the center of $A$ is a central ideal of $A$, and each central ideal of $A$ is included in $Z(A)$.

\begin{notation}\label{Yleft} Given an ideal $I$ of a linear cycle set $A$, for each $y\in I$ and $a\in A$, we set $y\Yleft a\coloneqq y\cdot a-a$.
\end{notation}

\subsection[Extensions of linear cycle sets]{Extensions of linear cycle sets}\label{Extensions of linear cycle sets}
Let $I$ and $H$ be additive abelian groups. Recall that an extension $(\iota, B, \pi)$, of $H$ by $I$, is a short exact sequence
\begin{equation}
    \begin{tikzpicture}
    \begin{scope}[yshift=0cm,xshift=0cm, baseline]
			\matrix(BPcomplex) [matrix of math nodes, row sep=0em, text height=1.5ex, text
			depth=0.25ex, column sep=2.5em, inner sep=0pt, minimum height=5mm, minimum width =6mm]
			{0 & I & B & H & 0,\\};
			\draw[-{latex}] (BPcomplex-1-1) -- node[above=1pt,font=\scriptsize] {} (BPcomplex-1-2);
			\draw[-{latex}] (BPcomplex-1-2) -- node[above=1pt,font=\scriptsize] {$\iota$} (BPcomplex-1-3);
			\draw[-{latex}] (BPcomplex-1-3) -- node[above=1pt,font=\scriptsize] {$\pi$} (BPcomplex-1-4);
			\draw[-{latex}] (BPcomplex-1-4) -- node[above=1pt,font=\scriptsize] {} (BPcomplex-1-5);
    \end{scope}\label{sucesion exacta corta de grupos abelianos}
    \end{tikzpicture}
\end{equation}
of additive abelian groups. Recall also that two such extensions, $(\iota, B, \pi)$ and $(\iota', B', \pi')$, are equivalent if there~exists a morphism $\phi\colon B\to B'$ such that $\pi'\xcirc \phi=\pi$ and $\phi\xcirc \iota=\iota'$. Then, $\phi$ is necessarily an isomorphism.

\smallskip

Let $(\iota, B, \pi)$ be an extensions of $H$ by $I$ and let $s\colon H\to B$ be a set theoretic section of $\pi$, with $s(0)=0$. Then, for each $b\in B$, there exist unique $y\in I$ and $h\in H$ such that $b=\iota(y)+s(h)$. Moreover, since~$\iota$ and $\pi$ are morphisms and $\iota$ is injective, there exists a unique map $\beta\colon H\times H\to I$ such that
\begin{equation}\label{construccion de beta}
\iota(y)+s(h) + \iota(y')+s(h')= \iota(y) + \iota(y')+ s(h) + s(h')= \iota(y+y' + \beta(h,h')) + s(h+h').
\end{equation}
A direct computation shows that
\begin{align}
& \beta(h,h')+\beta(h+h',h'') = \beta(h',h'')+\beta(h,h'+h''),\label{condicion de cociclo}\\
&\beta(h,0) = \beta(0,h) = 0\qquad\text{and}\qquad \beta(h,h') = \beta(h',h),\label{normalidad y cociclo abeliano}
\end{align}	
for all $h,h',h''\in H$. Conversely, given a map $\beta\colon H\times H\to I$ satisfying~\eqref{condicion de cociclo} and~\eqref{normalidad y cociclo abeliano}, we have an extension
\begin{equation}\label{sucesion exacta corta de grupos abelianos 2}
\begin{tikzpicture}
\begin{scope}[yshift=0cm,xshift=0cm, baseline]
			\matrix(BPcomplex) [matrix of math nodes, row sep=0em, text height=1.5ex, text
			depth=0.25ex, column sep=2.5em, inner sep=0pt, minimum height=5mm, minimum width =6mm]
			{0 & I & I\times_{\beta} H & H & 0,\\};
			\draw[->] (BPcomplex-1-1) -- node[above=1pt,font=\scriptsize] {} (BPcomplex-1-2);
			\draw[->] (BPcomplex-1-2) -- node[above=1pt,font=\scriptsize] {$\iota$} (BPcomplex-1-3);
			\draw[->] (BPcomplex-1-3) -- node[above=1pt,font=\scriptsize] {$\pi$} (BPcomplex-1-4);
			\draw[->] (BPcomplex-1-4) -- node[above=1pt,font=\scriptsize] {} (BPcomplex-1-5);
\end{scope}
\end{tikzpicture}
\end{equation}
of $H$ by $I$, in which  $I\times_{\beta} H$ is $I\times H$, endowed with the sum
\begin{equation*}
(y,h)+(y',h')\coloneqq (y+y'+\beta(h,h'),h+h'),
\end{equation*}
and the maps $\iota$ and $\pi$ are the canonical ones. For the sake of simplicity, here and subsequently, we will write $y$ instead of $(y,0)$ and we set $w_h\coloneqq (0,h)$. So, $(y,h)=(y,0)+(0,h)=y+w_h$. With this notation $\iota(y) = y$, $\pi(y+w_h) =h$ and the sum in $I\times_{\beta} H$ reads
\begin{equation}\label{formula para suma en I times H}
(y+w_h) + (y'+w_{h'}) = y+y'+\beta(h,h')+ w_{h+h'}.
\end{equation}
Each extension of abelian groups $(\iota,B,\pi)$, endowed with a set theoretic section $s$ of $\pi$ such that $s(0)=0$, is equivalent to the extension~\eqref{sucesion exacta corta de grupos abelianos 2} associated with the map $\beta$ obtained in~\eqref{construccion de beta}. Moreover, it is well known that two extensions $(\iota,I\times_{\beta} H,\pi)$ and $(\iota,I\times_{\beta'} H,\pi)$ are equivalent if and only if there exists a map $\varphi\colon H\to I$, such that
\begin{equation}\label{equivalencia a nivel de qrupos}
\varphi(0) = 0\qquad\text{and}\qquad \varphi(h)-\varphi(h+h')+\varphi(h')=\beta(h,h') - \beta'(h,h')\quad\text{for all $h,h'\in H$.}
\end{equation}
Concretely, the map $\phi\colon I\times_{\beta} H\to I\times_{\beta'} H$ performing the equivalence is given by $\phi(y+w_h)\coloneqq y+\varphi(h)+w_h$.

\begin{definition}\label{extensiones de linear cycle sets} Let $I$ and $H$ be linear cycle sets. An {\em extension} $(\iota,B, \pi)$, of $H$ by $I$, is a short exact sequence like~\eqref{sucesion exacta corta de grupos abelianos}, in which $B$ is a linear cycle set and both $\iota$ are $\pi$ are linear cycle sets morphisms. Two extensions $(\iota,B,\pi)$ and $(\iota',B',\pi')$, of $H$ by $I$, are {\em equivalent} if there exists a linear cycle set morphism  $\phi\colon B\to B'$ such that $\pi'\xcirc \phi=\pi$ and $\phi\xcirc \iota=\iota'$. As in the case of extension of groups, necessarily $\phi$ is an isomorphism. We let $\EXT(H;I)$ denote the set of equivalence classes of extensions $(\iota,B,\pi)$, of $H$ by $I$.
\end{definition}

Assume that~\eqref{sucesion exacta corta de grupos abelianos} is an extension of linear cycle sets and let $s$ be a section of $\pi$ with $s(0)=0$. For the sake of simplicity, in the rest of this section we identify $I$ with $\iota(I)$ and we write $y$ instead of $\iota(y)$. Let
$$
\blackdiamond\colon H\times I\to I\quad\text{and}\quad \Yleft \colon I\times H\to I
$$
be the maps defined by $h\blackdiamond y \coloneqq s(h)\cdot y$ and $y\Yleft h\coloneqq y\Yleft s(h) = y\cdot s(h)-s(h)$, respectively. From now on, we assume that $I$ is trivial. By~\cite{GGV}*{Proposition~3.3}, this implies that $\blackdiamond$ and $\Yleft$ do not depend on~$s$. Since $\pi\colon B\to H$ is a linear cycle set morphism, there is a map $f\colon H\times H\to I$, such that:
\begin{equation*}
(y+s(h))\cdot (y'+s(h'))= h\blackdiamond y' + f(h,h') + h\blackdiamond y \Yleft h\cdot h' + s(h\cdot h').
\end{equation*}

\begin{notation}\label{notacion ext} Fix maps $\blackdiamond\colon H\times I\to I$ and $\Yleft \colon I\times H\to I$. We let $\Ext_{\blackdiamond,\Yleft}(H;I)$ denote the set of equivalence classes of  extensions $(\iota,B,\pi)$, of $H$ by $I$, such that $s(h)\cdot y = h\blackdiamond y$ and $y\cdot s(h)-s(h) = y\Yleft h$, for every section $s$ of $\pi$ satisfying $s(0) = 0$. When $\Yleft$ is the zero map, we write $\Ext_{\blackdiamond}(H;I)$ instead of $\Ext_{\blackdiamond,0}(H;I)$.
\end{notation}

\subsection{Building extensions of linear cycle sets}\label{Building extensions of linear cycle sets}
Let $(\iota, I\times_{\beta} H, \pi)$ be the short exact sequence of abelian groups~\eqref{sucesion exacta corta de grupos abelianos 2}. We assume that $H$ is a linear cycle set and we fix maps $\blackdiamond\colon H\times I\to I$, $\Yleft \colon I\times H\to I$ and $f\colon H\times H\to I$. We let $I\times_{\beta,f}^{\blackdiamond,\Yleft} H$ denote $I\times_{\beta} H$ endowed with the binary operation $\cdot$ defined by
\begin{equation}\label{formula para cdot en I times H'}
(y+w_h)\cdot (y'+w_{h'})= h\blackdiamond y' + f(h,h') + h\blackdiamond y \Yleft h\cdot h' + w_{h\cdot h'}.
\end{equation}

\begin{proposition}\label{condiciones para extension} Let $\triangleleft\colon I \times H\to I$ be the map defined by $y\triangleleft h\coloneqq  h \blackdiamond (y - y\Yleft h)$. For each $y\in I$ and $h\in H$, set~$y^h\coloneqq y-y\Yleft h$. By \cite{GGV}*{Theorem~5.6 and Remark~5.12} we know that $(\iota, I\times_{\beta,f}^{\blackdiamond,\Yleft} H, \pi)$ is an extension of linear cycle sets (where $I$ has the trivial structure) if and only if the following conditions are fulfilled:
\begin{align}
& y \triangleleft h\times h' = (y \triangleleft h) \triangleleft h',\quad (y + y')\triangleleft h = y\triangleleft h + y'\triangleleft h,\quad y \triangleleft 0 = y,\label{eqq 1}\\
& h'\times h\blackdiamond y  = h\blackdiamond (h' \blackdiamond y),\quad h\blackdiamond (y + y') = h\blackdiamond y + h\blackdiamond y', \quad 0\blackdiamond y = y\label{eqq 2}\\
& y^{h+h'} + y = y^h + y^{h'}\quad \text{and}\quad 0^h = 0,\label{eqq 3}\\
& h\blackdiamond  \beta(h',h'') + f(h,h'+h'') = f(h,h')+f(h,h'')+\beta(h\cdot h',h\cdot h'')\label{eqq 4}
\shortintertext{and}
&\begin{aligned}\label{eqq5}
f(h+h',h'') + (h+h')\blackdiamond \beta(h,h')\Yleft (h+h')\cdot h'' & = (h\cdot h')\blackdiamond f(h,h'') + f(h\cdot h',h\cdot h'')\\
&+ (h\cdot h')\blackdiamond f(h,h') \Yleft (h\cdot h')\cdot(h\cdot h'').
\end{aligned}
\end{align}
for all $h,h'\in H$ and $y,y'\in I$.
\end{proposition}

\begin{remark}\label{otra def de y^h} Note that if~\eqref{eqq 2} is satisfied, then $y^h = h^{\times -1}\blackdiamond (y \triangleleft h)$.
\end{remark}

\begin{remark}\label{linealidad de Yleft} Since $y\Yleft h = y-y^h = y - h^{\times -1}\blackdiamond (y\triangleleft h)$, by~\eqref{eqq 1}, \eqref{eqq 2} and~\eqref{eqq 3} we obtain that $\Yleft$ is bilinear. Consequently, $0\Yleft h = y\Yleft 0 = 0$. Moreover, by~\eqref{eqq 3} we have $y^0 = y$.
\end{remark}

\begin{remark}\label{equivalecia entre extensiones producto, caso I trivial} Two extensions $(\iota,I\times_{\beta,f}^{\blackdiamond,\Yleft} H,\pi)$ and $(\iota,I\times_{\beta',f'}^{\blackdiamond,\Yleft} H,\pi)$, of $H$ by $I$, are equivalent if and only if there exists a map $\varphi\colon H\to I$ satisfying~\eqref{equivalencia a nivel de qrupos} and
\begin{equation}\label{tercera condicion caso I trivial}
\varphi(h\cdot h') + f(h,h') = h\blackdiamond \varphi(h') + f'(h,h') + h\blackdiamond \varphi(h)\Yleft h\cdot h'\quad\text{for all $h,h'\in H$.}
\end{equation}
\end{remark}

\begin{remark} Let $I$ be a trivial linear cycle set, and let $\iota\colon I\to B$ and $\pi\colon B\to H$ be morphisms of linear cycle sets. Assume that $(\iota,B,\pi)$ is an extension of linear cycle sets, and let $s$ be a set-theoretic section of $\pi$ such that $s(0)=0$. Let $\beta$ be the map obtained in~\eqref{construccion de beta}, and let $\blackdiamond$, $\Yleft$ and $f$ be the maps defined above Notation~\ref{notacion ext}. Consider the short exact sequence of abelian groups~\eqref{sucesion exacta corta de grupos abelianos 2}, where $I\times_{\beta} H$ is endowed with the operation $\cdot$ given in~\eqref{formula para cdot en I times H'}. The map
$$
\phi\colon I\times_{\beta} H\to B,
$$
defined by $\phi(y+w_h)\coloneqq \iota(y)+s(h)$, is a bijection compatible with both $+$ and~$\cdot$. Hence,~\eqref{sucesion exacta corta de grupos abelianos 2} is an extension of~lin\-ear cycle sets, and $\phi$ is an equivalence of extensions.
\end{remark}

\begin{remark}\label{condicion para I incluido en Soc} Since $I$ is trivial, $I\subseteq \Soc\bigl(I\times_{\beta,f}^{\blackdiamond,\Yleft} H\bigr)$ if and only if $y\Yleft h= y\cdot w_h-w_h = 0$, for all $y\in I$ and~$h\in H$. Since $y^h=y-y\Yleft h$, this is equivalent to $y^h=y$, for all $y\in I$ and $h\in H$. Moreover, in this case $y\triangleleft h=h\blackdiamond y$, and condition~\eqref{eqq5} becomes
\begin{equation}\label{condicion equivalente compatibilidad suma cdot con hmassimp3 caso simple}
f(h+h',h'')  = (h\cdot h')\blackdiamond f(h,h'') + f(h\cdot h',h\cdot h''),
\end{equation}
while condition~\eqref{tercera condicion caso I trivial} becomes
$$
\varphi(h\cdot h') + f(h,h') = h\blackdiamond \varphi(h') + f'(h,h').
$$
Consequently, conditions~\eqref{eqq 1}--\eqref{eqq5} reduce to~\eqref{eqq 2}, \eqref{eqq 4} and~\eqref{condicion equivalente compatibilidad suma cdot con hmassimp3 caso simple}.
\end{remark}

\subsection[Cohomology of linear cycle sets]{Cohomology of linear cycle sets}\label{Cohomology of linear cycle sets}

Let $H$ be a linear cycle set. For each $s\ge 1$, we let $\mathds{Z}[H^s]$ denote the free abelian group with basis $H^s$. Let $\sh(\mathds{Z}[H^s])$ be the subgroup of $\mathds{Z}[H^s]$, generated by the shuffles
\begin{equation*}
\sum_{\sigma\in \sh_{l,s-l}} \sg(\sigma)(h_{\sigma^{-1}(1)},\dots, h_{\sigma^{-1}(s)}),
\end{equation*}
taken for all $1\le l<s$. Here $\sh_{l,s-l}$ is the subset of all the permutations $\sigma$ of $s$ elements satisfying $\sigma(1)<\cdots<\sigma(l)$ and $\sigma(l+1)<\cdots<\sigma(s)$. For each $r\ge 0$ and $s\ge 1$, let $C^N_{rs}(H)\coloneqq \mathds{Z}[H^{r}]\ot \ov{M}(s)$, where $\ov{M}(s)\coloneqq \frac{\mathds{Z}[H^s]}{\sh(\mathds{Z}[H^s)]}$. Given $h_1,\dots,h_s\in H$, we let $[h_1,\dots,h_s]$ denote the class of $(h_1,\dots,h_s)$ in $\ov{M}(s)$. Let $I$ be an abelian group endowed with a map $\blackdiamond \colon H\times I \to I$ satisfying conditions~\eqref{eqq 2}. In this subsection we consider~$I$ as a linear cycle set via the trivial operation $a\cdot b\coloneqq b$. Let $\wh{C}_N^{rs}(H,I)$ be the subgroup of $\Hom(C^N_{rs}(H),I)$, con\-sisting of the morphisms $f\colon C^N_{rs}(H)\to I$, such that $f(h_1,\dots, h_r\ot [h_{r+1},\dots,h_{r+s}]) = 0$, if some $h_i=0$. Notice that $\wh{C}_N^{rs}(H,I)$ is isomorphic to the subgroup of $\Hom(\mathds{Z}[H^{r+s}], I)$, consisting of all the morphisms $f$ that vanish on all the $(r+s)$-tuples with some coordinate equal to zero, and such that
\begin{equation*}
\sum_{\sigma\in \sh_{l,s-l}} \sg(\sigma) f(h_1,\dots,h_r, h_{r+\sigma^{-1}(1)},\dots, h_{r+\sigma^{-1}(s)})=0,
\end{equation*}
for all $h_1,\dots,h_{r+s}\in H$ and $1\le l<s$. Consider the diagram $\bigl(\wh{C}_N^{**}(H,I),\partial_{\mathrm{h}}^{**}, \partial_{\mathrm{v}}^{**}\bigr)$, where
\begin{equation*}
\partial_{\mathrm{h}}^{r+1,s}\colon \wh{C}_N^{rs}(H,I)\to \wh{C}_N^{r+1,s}(H,I)\quad\text{and}\quad \partial_{\mathrm{v}}^{r,s+1}\colon \wh{C}_N^{rs}(H,I)\to \wh{C}_N^{r,s+1}(H,I)
\end{equation*}
are the maps defined by
\begin{align*}
& \partial_{\mathrm{h}}^{r+1,s}(f)(h_1,\dots, h_{r+1},h_{r+2},\dots,h_{r+s+1}) \coloneqq f(h_1\cdot h_2,\dots,h_1\cdot h_{r+1},h_1\cdot h_{r+2},\dots,h_1\cdot h_{r+s+1})\\
&\phantom{\partial_{\mathrm{h}}^{r+1,s}(f)(h_1,\dots,h_{r+1}} + \sum_{\jmath=1}^{r} (-1)^{\jmath} f(h_1,\dots,h_{\jmath-1},h_{\jmath}+h_{\jmath+1}, h_{\jmath+2},\dots,h_{r+1},h_{r+2},\dots, h_{r+s+1})\\
& \phantom{\partial_{\mathrm{h}}^{r+1,s}(f)(h_1,\dots,h_{r+1}} + (-1)^{r+1} \bigl((h_1+\cdots + h_r)\cdot h_{r+1}\bigr) \blackdiamond f(h_1,\dots,h_r,h_{r+2},\dots, h_{r+s+1})
\shortintertext{and}
& \partial_{\mathrm{v}}^{r,s+1}(f)(h_1,\dots,h_r,h_{r+1},\dots,h_{r+s+1}) \coloneqq (-1)^r f(h_1,\dots,h_r,h_{r+2},\dots,h_{r+s+1})\\
& \phantom{\partial_{\mathrm{v}}^{r,s+1}(f)(h_1,\dots,h_r} + \sum_{\jmath=r+1}^{r+s} (-1)^{\jmath}f(h_1,\dots,h_r,h_{r+1},\dots,h_{\jmath-1}, h_{\jmath}+h_{\jmath+1},h_{\jmath+2},\dots, h_{r+s+1})\\
& \phantom{\partial_{\mathrm{v}}^{r,s+1}(f)(h_1,\dots,h_r} + (-1)^{r+s+1} f(h_1,\dots,h_r,h_{r+1},\dots,h_{r+s}).
\end{align*}
Note that $\bigl(\wh{C}_N^{**}(H,I),\partial_{\mathrm{h}}^{**}, \partial_{\mathrm{v}}^{**}\bigr)$ is the diagram obtained from Figure~\ref{ddiagonales}, by deleting the arrows $D$.

\begin{theorem}[\cite{GGV}*{Theorem 6.1}]\label{principal} The diagram $\bigl(\wh{C}_N^{**}(H,I),\partial_{\mathrm{h}}^{**},\partial_{\mathrm{v}}^{**}\bigr)$ is a double cochain complex.
\end{theorem}

Let $\blackdiamond \colon H\times I \to I$ and $\Yleft \colon I\times H\to I$ be maps satisfying conditions~\eqref{eqq 1}--\eqref{eqq 3}. According to~\cite{GGV}*{Sec\-tion~7}, there exist morphisms
\begin{equation*}
D_{rs}^{r+s,1}\colon \wh{C}_N^{rs}(H,I) \to \wh{C}_N^{r+s,1}(H,I),
\end{equation*}
that define a degree~$1$ endomorphism $D$ of $\wh{C}_N^*(H,I)\coloneqq \Tot\bigl(\wh{C}_N^{**}(H,I)\bigl)$ such that $\bigl(\wh{C}_N^*(H,I),\partial+D\bigr)$ is a cochain complex, where $\partial\coloneqq \partial_{\mathrm{h}}+\partial_{\mathrm{v}}$. Next we recall the definition of $D$.

\begin{notation} Given elements $h_1,\dots,h_{r+s}$, we set $\mathrm{h}_{1,r+s}\coloneqq (h_1,\dots,h_{r+s})$.
\end{notation}

Then, for all $r\ge0$ and $s\ge 1$, we define  $D_{rs}^{r+s,1}\colon \wh{C}_N^{rs}(H,I)\longrightarrow \wh{C}_N^{r+s,1}(H,I)$, by
\begin{equation*}
D_{rs}^{r+s,1}(f)(\mathrm{h}_{1,r+s+1})=(-1)^{r+s} \bigl(\bigl((h_1+\cdots+h_r)\cdot (h_{r+1}+\cdots+h_{r+s})\bigr) \blackdiamond f(\mathrm{h}_{1,r+s})\bigr) \Yleft (h_1+\cdots+h_{r+s})\cdot h_{r+s+1}.
\end{equation*}
Finally, for $n\ge 1$, we define $D_n^{n+1}\colon \wh{C}_N^n(H,I)\longrightarrow \wh{C}_N^{n+1}(H,I)$ by
\begin{equation*}
D_n^{n+1}(f_{0n},\dots,f_{n-1,1})\coloneqq \sum_{k=1}^n D_{n-k,k}^{n1}(f_{n-k,k}).
\end{equation*}
Figure~\ref{ddiagonales} shows in black the diagram $\bigl(\wh{C}_N^{**}(H,I),\partial_{\mathrm{h}}^{**}, \partial_{\mathrm{v}}^{**}\bigr)$, and in green the maps $D_{rs}^{r+s,1}$. For the sake of legibility, in that diagram we write $\wh{C}_N^{rs}$ instead of $\wh{C}_N^{rs}(H,I)$.

\begin{figure}[H]
\begin{center}
\adjustbox{scale=1.4,center}{%
\begin{tikzcd}[row sep=2em, column sep=2 * 1.61803398874988em, scale=2]
\vdots \arrow[rrrrddd, green(ncs), shift left=0.2ex, end anchor={[yshift=0.4ex]}, end anchor={[xshift=-0.1ex]}, start anchor={[xshift=0.5ex]}, start anchor={[yshift=0.3ex]}, outer sep=-1.6pt, "D" pos=0.07] & \vdots \arrow[rrrrddd, green(ncs), shift left=0.2ex, end anchor={[yshift=1.4ex]north west}, end anchor={[xshift=-1.6ex]}, start anchor={[xshift=0.5ex]}, start anchor={[yshift=0.3ex]}, outer sep=-1.6pt, "D" pos=0.07] & \vdots \arrow[rrrdd, dash, green(ncs), shift left=0.3ex, end anchor={[yshift=-1.25ex]north west}, start anchor={[xshift=0.3ex]}, start anchor={[yshift=0.05ex]}, outer sep=-1.6pt, "D" pos=0.105] & \vdots \arrow[rrd, dash, green(ncs), shift left=0.2ex, end anchor={[yshift=-3.2ex]north west}, start anchor={[xshift=0.5ex]}, start anchor={[yshift=0.05ex]}, outer sep=-1.6pt, "D" pos=0.18]  & \vdots \arrow[rd, dash, green(ncs), shift left= 0.4ex, end anchor={[yshift=2.4ex]north west}, outer sep=-1.6pt, "D" pos=0.35] & \\
\wh{C}_N^{03} \arrow[r] \arrow[u] \arrow[rrrdd, green(ncs), shift left=0ex, end anchor={[yshift=0.6ex]}, start anchor={[xshift=-0.3ex]}, start anchor={[yshift=0.3ex]}, outer sep=-1.6pt, "D" pos=0.08] & \wh{C}_N^{13} \arrow[r, crossing over] \arrow[u, crossing over] \arrow[rrrdd, green(ncs), shift left=0ex, end anchor={[yshift=0.6ex]}, start anchor={[xshift=-0.3ex]}, start anchor={[yshift=0.3ex]}, outer sep=-1.6pt, "D" pos=0.08]& \wh{C}_N^{23} \arrow[r, crossing over] \arrow[u, crossing over] \arrow[rrrdd, green(ncs), shift left=0ex, end anchor={[yshift=1.45ex]north west}, start anchor={[xshift=-0.3ex]}, start anchor={[yshift=0.3ex]}, outer sep=-1.6pt, "D" pos=0.08] & \wh{C}_N^{33} \arrow[r, crossing over] \arrow[u, crossing over] \arrow[rrd, dash, green(ncs), shift left=0ex, end anchor={[yshift=-1.25ex]north west}, start anchor={[xshift=-0.3ex]}, start anchor={[yshift=0.05ex]}, outer sep=-1.8pt, "D" pos=0.14] & \wh{C}_N^{43} \arrow[r, crossing over] \arrow[u, crossing over] \arrow[r, dash, green(ncs), shift left=-0.5ex, end anchor={[yshift=-3.4ex]north west}, outer sep=-1.8pt, "D" pos=0.35] & \dots \\
\wh{C}_N^{02} \arrow[r] \arrow[u] \arrow[rrd, green(ncs), shift left=0ex, end anchor={[yshift=0.5ex]}, outer sep=-1.6pt, "D" pos=0.11] & \wh{C}_N^{12} \arrow[r, crossing over] \arrow[u, crossing over]  \arrow[rrd, green(ncs), shift left=0ex, end anchor={[yshift=0.5ex]}, outer sep=-1.6pt, "D" pos=0.12]& \wh{C}_N^{22} \arrow[r, crossing over] \arrow[u, crossing over]  \arrow[rrd, green(ncs), shift left=0ex, end anchor={[yshift=0.5ex]}, outer sep=-1.6pt, "D" pos=0.12]& \wh{C}_N^{32} \arrow[r, crossing over] \arrow[u, crossing over] \arrow[rrd, green(ncs), shift left=0ex, end anchor={[yshift=1.1ex]north west}, outer sep=-1.6pt, "D" pos=0.12] & \wh{C}_N^{42} \arrow[r, crossing over] \arrow[u, crossing over] \arrow[r, dash, green(ncs), shift left=-0.5ex, end anchor={[yshift=-2.5ex]north west}, outer sep=-1.6pt, "D" pos=0.36] & \dots \\
\wh{C}_N^{01} \arrow[r] \arrow[u] \arrow[r, green(ncs), shift left=0.5ex, outer sep=-1.4pt, "D" pos=0.41]  & \wh{C}_N^{11} \arrow[r] \arrow[u, crossing over] \arrow[r, green(ncs), shift left=0.5ex, outer sep=-1.4pt, "D" pos=0.43] & \wh{C}_N^{21} \arrow[r] \arrow[u, crossing over] \arrow[r, green(ncs), shift left=0.5ex, outer sep=-1.4pt, "D" pos=0.42] & \wh{C}_N^{31} \arrow[r] \arrow[u, crossing over] \arrow[r, green(ncs), shift left=0.5ex, outer sep=-1.4pt, "D" pos=0.42] & \wh{C}_N^{41} \arrow[r] \arrow[u, crossing over, outer sep=-1.6pt, "D" pos=0.11] \arrow[r, green(ncs), shift left=0.5ex, outer sep=-1.4pt, "D" pos=0.42] & \dots
\end{tikzcd}
}
\end{center}
\caption{$\bigl(\wh{C}_N^{**}(H,I),\partial_{\mathrm{h}},\partial_{\mathrm{v}}, D\bigr)$}\label{ddiagonales}
\end{figure}

\begin{proposition}[\cite{GGV}*{Proposition~7.9}]\label{significado de cociclos caso general} A pair $(\beta,f)\in \wh{C}_N^2(H,I)=\wh{C}_N^{02}(H,I)\oplus \wh{C}_N^{11}(H,I)$ is a cocycle of the~com\-plex $(\wh{C}_N^*(H,I),\partial+D)$ if and only if $(\beta,-f)$ satisfies conditions~\eqref{condicion de cociclo},~\eqref{normalidad y cociclo abeliano},~\eqref{eqq 4} and~\eqref{eqq5}.
\end{proposition}

\begin{corollary}[\cite{GGV}*{Corollary~7.10}]\label{equiv caso general} There is a bijective correspondence $\Ext_{\blackdiamond,\Yleft}(H,I) \longleftrightarrow \ho_{\blackdiamond,\Yleft}^2(H,I)$. Concretely, this correspondence sends the class of a $2$-cocycle $(\beta,-f)\in \wh{C}_N^2(H,I)=\wh{C}_N^{02}(H,I)\oplus \wh{C}_N^{11}(H,I)$ module coboundaries into the class of the extension $(\iota, I\times_{\beta,f}^{\blackdiamond,\Yleft} H,\pi)$.
\end{corollary}

When $\Yleft=0$, we will write ${\mathcal{C}}_{\blackdiamond}(H,I)$ instead of ${\mathcal{C}}_{\blackdiamond,\Yleft}(H,I)$ and $\ho_{\blackdiamond}^2(H,I)$ instead of $\ho_{\blackdiamond,\Yleft}^2(H,I)$.

\section[The linear cycle set \texorpdfstring{$H^{d_1d_2}_{\lambda}$}{H{d1d2}a}]{The linear cycle set \texorpdfstring{$\pmb{H^{d_1d_2}_{\lambda}}$}{Hd1d2a}}\label{section 2}
Let $\mathrm{Ab}_{d_1d_2}\coloneqq \mathds{Z}_{d_1}\oplus \mathds{Z}_{d_2}$, where $d_1\mid d_2$, let $\mathrm{U}_{d_1d_2}$ be the underlying set of $\mathrm{Ab}_{d_1d_2}$, and let $\lambda\in \mathds{Z}_{d_1}$. A direct~compu\-tation proves that $H^{d_1d_2}_{\lambda}\coloneqq (\mathrm{Ab}_{d_1d_2},\cdot)$, where $\cdot$ is the binary operation of $\mathrm{U}_{d_1d_2}$, defined by
\begin{equation}\label{def de punto}
(x,y)\cdot (u,v)\coloneqq (u-\lambda yv,v) = (u,v)- (\lambda yv,0)
\end{equation}
is a linear cycle set (here the juxtaposition denotes the multiplication in the ring $\mathds{Z}_{d_1}$ and, in the first coordi\-nate,~$y$ and $v$ denote the canonical image of these elements in $\mathds{Z}_{d_1}$).

\begin{proposition}\label{isomorfos} Let $\lambda_1, \lambda_2\in \mathds{Z}_{d_1}$. The linear cycle sets $H^{d_1d_2}_{\lambda_1}$ and $H^{d_1d_2}_{\lambda_2}$ are isomorphic if and only if there exists a unit $a \in \mathds{Z}_{d_1}$ such that $a\lambda_1=\lambda_2$.
\end{proposition}

\begin{proof} $\Leftarrow$)\enspace A simple computation proves that the map $\varphi_a\colon H^{d_1d_2}_{\lambda_1}\to H^{d_1d_2}_{\lambda_2}$, defined by $\varphi_a(x,y)\coloneqq (ax,y)$, is~an~iso\-morphism.

\smallskip
	
\noindent $\Rightarrow$)\enspace Let $\varphi\colon H^{d_1d_2}_{\lambda_1}\to H^{d_1d_2}_{\lambda_2}$ be an isomorphism of linear cycle sets and let
\begin{equation}
A=\begin{pmatrix}
a_{11}&a_{12} \\
a_{21}&a_{22}
\end{pmatrix}\in\begin{pmatrix} \mathds{Z}_{d_1}&\mathds{Z}_{d_1}\\
\frac{d_2}{d_1} \mathds{Z}_{d_2}&\mathds{Z}_{d_2}
\end{pmatrix}
\end{equation}
be such that $\varphi \begin{psmallmatrix}x \\ y \end{psmallmatrix}= A\begin{psmallmatrix}x \\ y \end{psmallmatrix}$, where we are writing the elements of  $H^{d_1d_2}_{\lambda_1}$ and $ H^{d_1d_2}_{\lambda_2}$ as ``column vectors''. Note that~$A$ is invertible, since $\varphi$ is an isomorphism. On one hand, we have
\begin{equation*}
A\begin{pmatrix}x\\y\end{pmatrix}\cdot A\begin{pmatrix} u\\v\end{pmatrix} =\begin{pmatrix} a_{11}x+a_{12}y\\ a_{21}x+a_{22}y \end{pmatrix}\cdot \begin{pmatrix} a_{11}u+a_{12}v\\ a_{21}u+a_{22}v \end{pmatrix} = \begin{pmatrix} a_{11}u+a_{12}v-\lambda_2(a_{21}x+a_{22}y)(a_{21}u+a_{22}v)\\ a_{21}u+a_{22}v\end{pmatrix},
\end{equation*}
and, on the other hand, we have
\begin{equation*}
A\left(\begin{pmatrix} x\\y\end{pmatrix}\cdot \begin{pmatrix}u\\v \end{pmatrix}\right) = A \begin{pmatrix} u-\lambda_1 yv\\v\end{pmatrix} =\begin{pmatrix} a_{11}u-\lambda_1 a_{11}yv+a_{12}v\\ a_{21}u-a_{21}\lambda_1 yv+a_{22}v\end{pmatrix}.
\end{equation*}
Since $A\begin{psmallmatrix}x\\y\end{psmallmatrix}\cdot A\begin{psmallmatrix}u\\v\end{psmallmatrix}=A\bigl(\begin{psmallmatrix}x\\y \end{psmallmatrix} \cdot \begin{psmallmatrix}u\\v\end{psmallmatrix}\bigr)$, we conclude that for all $x,u\in \mathds{Z}_{d_1}$ and  $y,v\in \mathds{Z}_{d_2}$,
\begin{equation*}
\lambda_1 a_{11}yv=\lambda_2(a_{21}x+a_{22}y)(a_{21}u+a_{22}v)\quad\text{and}\quad a_{21}\lambda_1 yv=0.
\end{equation*}
The second equality, which takes place in $\mathds{Z}_{d_2}$, holds if and only if $a_{21}\lambda_1=0$; and the first equality, which takes place in $\mathds{Z}_{d_1}$, with $x = 1$, $y = 0$ and $\begin{psmallmatrix}u\\v\end{psmallmatrix}= A^{-1} \begin{psmallmatrix} 0 \\ 1 \end{psmallmatrix}$, yields $a_{21}\lambda_2=0$. A direct computation using this, shows that the first equality holds if and only if
\begin{equation}\label{igualdad}
a_{21}\lambda_2=0\quad\text{and}\quad\lambda_1 a_{11}=\lambda_2 a_{22}^2.
\end{equation}
Let $B = \begin{psmallmatrix} b_{11}&b_{12}\\b_{21}&b_{22}\end{psmallmatrix} = A^{-1}$. By symmetry, $b_{21}\lambda_2=0$ and  $b_{21}\lambda_1=0$. Hence,
\begin{equation*}
\begin{pmatrix}\lambda_1\\0\end{pmatrix} = BA\begin{pmatrix}\lambda_1\\0\end{pmatrix} = B\begin{pmatrix}a_{11}\lambda_1\\0\end{pmatrix}=\begin{pmatrix}b_{11}a_{11}\lambda_1\\0\end{pmatrix},
\end{equation*}
which says that $b_{11}a_{11}\lambda_1=\lambda_1$. Therefore, by~\eqref{igualdad},
\begin{equation*}
\lambda_1=b_{11}a_{11}\lambda_1=\lambda_2 b_{11}a_{22}^2,
\end{equation*}	
which shows that~$\lambda_1$ is a multiple of~$\lambda_2$  in $\mathds{Z}_{d_1}$. By symmetry, also~$\lambda_2$ is a multiple of~$\lambda_1$. But it is well known that if $\lambda_1\mid \lambda_2$ and $\lambda_2\mid \lambda_1$, then $\lambda_2=a\lambda_1$ with $a\in \mathds{Z}_{d_1}^{\times}$. \end{proof}

\begin{remark}\label{lambda divide a d1} If $\lambda \neq 0$, we may assume that $0 < \lambda < d_1$ and that $\lambda \mid d_1$. Indeed, let $g \coloneqq \gcd(\lambda, d_1)$, with $0 < g < d_1$. Clearly, $g \mid \lambda$ in $\mathds{Z}_{d_1}$, and the existence of $b, c \in \mathds{Z}$ such that $g = b\lambda + c d_1$ shows that $\lambda \mid g$ in $\mathds{Z}_{d_1}$. Consequently, there exists a unit $a \in \mathds{Z}_{d_1}$ such that $a\lambda = g$. By Proposition~\ref{isomorfos}, this implies that the linear cycle sets $H^{d_1 d_2}_{\lambda}$ and~$H^{d_1 d_2}_{g}$ are isomorphic. Therefore, we may replace $\lambda$ by $g$.
\end{remark}
A direct computation using~\eqref{equivalencia linear cycle set braza} shows that the brace associated with $H^{d_1d_2}_{\lambda}$ is $(\mathrm{Ab}_{d_1d_2}, \times)$, where
\begin{equation}\label{producto}
(x,y)\times (x',y') = (x+x' + \lambda y y',y+y').
\end{equation}
This formula implies that
\begin{equation}\label{potencia}
(x,y)^{\times a}=\left(ax+\lambda \binom{a}{2}y^2,ay\right),\qquad\text{for each $(x,y)\in H^{d_1d_2}_{\lambda}$ and $a\in \mathds{Z}$.}	
\end{equation}
Combining~\eqref{producto} with~\eqref{potencia}, we obtain that
\begin{equation}\label{combinacion}
(a,a') = (1,0)^{\times a-\lambda \binom{a'}{2}}\times (0,1)^{\times a'},\qquad\text{for all $a,a'\in \mathds{Z}$.}
\end{equation}
For all $(x,y)\in H^{d_1d_2}_{\lambda}$, we let $|(x,y)|$ denote the order of $(x,y)$ in $(\mathrm{U}_{d_1d_2},\times)$.

\begin{remark}\label{caso d2 impar} Since $\binom{y+y'}{2}=\binom{y}{2}+\binom{y'}{2}+yy'$, the surjective map $\phi\colon \mathds{Z}\times \mathds{Z}\to (\mathrm{U}_{d_1d_2},\times)$, given by
\begin{equation*}
\phi(x,y)\coloneqq \left(x+\lambda\binom{y}{2},y\right),
\end{equation*}
is a morphism. Note that $\phi(d_1,0) = 0$, and that $\phi(0,d_2) = 0$ if and only if $d_1\mid \lambda \binom{d_2}{2}$. Hence, in this case, $\phi$ induces an isomor\-phism $\ov{\phi}\colon \mathrm{Ab}_{d_1d_2}\to (\mathrm{U}_{d_1d_2},\times)$, and so $(\mathrm{U}_{d_1d_2},\times)$ is generated by the elements $a_1\coloneqq (1,0)$ and $a_2\coloneqq (0,1)$, of orders $r_1 = d_1$ and $r_2 = d_2$, respectively.
\end{remark}

For a natural number $d$, we define $v_2(d)$ as the multiplicity of $2$ in the prime factorization of $d$. As usual, we also set $v_2(0)\coloneqq +\infty$.

\smallskip

Next, we examine the case $d_1\nmid \lambda \binom{d_2}{2}$, which occurs precisely when $v_2(d_1)=v_2(d_2)>0$ and $\lambda$ is odd.

\begin{remark}\label{caso v_2=1} Assume that $v_2(d_1)=v_2(d_2)=1$ and that $\lambda$ is odd. Then $(0,1)$ has order $2d_2$ in $(\mathrm{U}_{d_1d_2},\times)$. Indeed
$$
(0,1)^{\times d_2}=\left(\lambda\binom{d_2}{2},0\right)\ne 0\quad\text{and}\quad (0,1)^{\times 2d_2}=\left(\lambda\binom{2d_2}{2},0\right) = 0,
$$
since $\lambda\binom{d_2}{2}$ is odd, $d_1$ is even, and $d_1\mid \lambda d_2(d_2-1) = \lambda\binom{2d_2}{2}$. We have two subcases:

\begin{description}
\item[$d_1=2$] In this subcase $(\mathrm{U}_{d_1d_2},\times)$ is cyclic generated by $a_1\coloneqq (0,1)$, of order $r_1 = 2d_2$.

\item[$d_1>2$] In this subcase
$$
(\mathrm{U}_{d_1d_2},\times)\cong \mathds{Z}_{\frac{d_1}{2}}\oplus \mathds{Z}_{2d_2},
$$
with generators $a_1\coloneqq (2,0)$ and $a_2\coloneqq (0,1)$, of or\-ders $r_1 = d_1/2$ and $r_2 = 2d_2$, respectively.
\end{description}
\end{remark}

\begin{proposition}\label{caso v2(d1)=v2(d2) y alfa impar} Assume that $v_2(d_1)=v_2(d_2)\ge2$ and that $\lambda$ is odd. Let $q\coloneqq d_2/d_1$ and let $0\le t <d_1/2$ be such that
\begin{equation}\label{paco}
\lambda q(2q - 1) - 2t= 1 \text{ in } \mathds{Z}_{d_1} \quad\text{($t$ exists since $\lambda$ and $q$ are odd)}.
\end{equation}
Then $|(0,1)|=2d_2$, $|(2t,2q)| = d_1/2$ and $\langle(0,1)\rangle\cap \langle(2t,2q)\rangle=0$; which implies $(\mathrm{U}_{d_1d_2},\times)\simeq \mathds{Z}_{\frac{d_1}{2}}\oplus \mathds{Z}_{2d_2}$.
\end{proposition}

\begin{proof} Note that $(0,1)^{\times d_2}=\bigl(\lambda \binom{d_2}{2},0\bigr)\ne 0$ and that $(0,1)^{\times 2d_2}=0$. Hence, $|(0,1)| = 2d_2$, as we want. On the other hand, if $1\le n<d_1/2$, then $(2t,2q)^{\times n}\ne 0$, since $d_2\nmid 2qn$; while
\begin{equation*}
(2t,2q)^{\times d_1/2}=\left(d_1 t+4\lambda q^2 \binom{d_1/2}{2},d_1 q\right)=(0,0).
\end{equation*}
So, $|(2t,2q)| = d_1/2$ as desired. It remains to check that $\langle (2t,2q)\rangle \cap \langle (0,1)\rangle=\{0\}$. Suppose that $(0,1)^{\times k}=(2t,2q)^{\times r}$ with $0\le r<d_1/2$. Then
\begin{equation}\label{paco1}
\left(\lambda \binom{k}{2},k\right)= (0,1)^{\times k}=(2t,2q)^{\times r} = \left(2rt+4\lambda q^2 \binom{r}{2},2rq\right),
\end{equation}
which implies that $k=2rq$ in $\mathds{Z}_{d_2}$. Let $l\in \mathds{Z}$ be such that $k + d_2l = 2rq$. Since $d_2$ is even, $k$ is also. Write $d_2 = 2d_2'$ and $k = 2k'$. Computing we obtain
$$
\binom{2rq}{2} = rq(2rq-1) = (k'+d_2'l)(k+d_2l-1) = k'(k-1) + k'd_2l + d_2'lk + d_2'd_2l^2 - d_2'l = \binom{k}{2} + \frac{d_1}{2} q,
$$
for some $q\in \mathds{Z}$. Combining this with~\eqref{paco1}, we obtain that, in $\mathds{Z}_{d_1}$,
\begin{equation*}
2\lambda r^2q^2-\lambda rq = \lambda \binom{2rq}{2} =\lambda \binom{k}{2} + \frac{d_1}{2}\lambda q =2rt+4\lambda q^2 \binom{r}{2} + \frac{d_1}{2}\lambda q = 2rt+2\lambda r^2q^2-2\lambda r q^2 + \frac{d_1}{2}\lambda q,
\end{equation*}
or, equivalently,
$$
2\lambda r q^2- \lambda rq-2rt  = \frac{d_1}{2}\lambda q.
$$
By~\eqref{paco} this implies that  $r = r(\lambda q(2q-1)-2t)=\frac{d_1}{2}\lambda q\pmod{d_1}$. Since $0\le r<d_1/2$, this implies that $r=0$, which finishes the proof.
\end{proof}

\begin{remark}\label{caso final} Under the hypotheses of Proposition~\ref{caso v2(d1)=v2(d2) y alfa impar}, the group $(\mathrm{U}_{d_1d_2},\times)$ is isomorphic to $\mathds{Z}_{\frac{d_1}{2}}\oplus \mathds{Z}_{2d_2}$ and has generators $a_1\coloneqq (2t,2q)$ and $a_2\coloneqq (0,1)$, of orders $r_1=d_1/2$ and~$r_2=2d_2$, respectively.
\end{remark}

\section[The actions \texorpdfstring{$\Yleft$}{<} and \texorpdfstring{$\blackdiamond$}{*} for \texorpdfstring{$H^{d_1d_2}_{\lambda}$}{Hd1d2a}] {The actions \texorpdfstring{$\pmb\Yleft$}{<} and \texorpdfstring{$\blackdiamond$}{*} for \texorpdfstring{$\pmb{H^{d_1d_2}_{\lambda}}$}{Ha}}

Let $I$ be an abelian group. In this section we characterize the actions $\blackdiamond$ and $\Yleft$ of $H^{d_1d_2}_{\lambda}$ on $I$, subject to  conditions~\eqref{eqq 1}--\eqref{eqq 3}. The analysis is carried out through a sequence of propositions:

\begin{itemize}
\item[-] In the first proposition we treat the case $d_2$ odd, or $v_2(d_1)<v_2(d_2)$, or $v_2(d_1)=v_2(d_2)\ge 1$ with $\lambda$ even.

\item[-] In the second one, we examine the situation $\lambda=1$, $d_1=2$, and $v_2(d_2)=1$.

\item[-] In the third one, we consider $\lambda$ odd, $d_1>2$, and $v_2(d_1)=v_2(d_2)=1$.

\item[-] Finally, in the fourth proposition we address the case  $\lambda$ odd with $v_2(d_1)=v_2(d_2)\ge 2$.
\end{itemize}
Note that the conditions in the first item are satisfied if and only if $d_1\mid \lambda\binom{d_2}{2}$.

\begin{proposition}\label{caracterizacion primer caso} Assume that $d_2$ is odd, or $v_2(d_1)<v_2(d_2)$, or $v_2(d_1)=v_2(d_2)\ge 1$ and $\lambda$ is even. Let $A_1$, $A_2$,~$B_1$ and $B_2$ be endomorphisms of $I$ such that
\begin{align}
&A_1^{d_1}=A_2^{d_2}=\Ide, \quad [A_1,A_2]=0, \quad d_1B_1=d_2B_2=0,\label{Propiedades de A1 y A2, B1, B2 primer caso}\\
&(A_1 -A_1B_1)^{d_1}=\Ide,\quad (A_2 -A_2\xcirc B_2)^{d_2}=\Ide,\quad [A_1 -A_1\xcirc B_1,A_2 -A_2\xcirc B_2]=0,\label{primeras propiedades de A-AB primer caso}\\
& [B_1,A_1]=B_1\xcirc A_1\xcirc B_1,\quad [B_1,A_2]=B_1\xcirc A_2\xcirc B_2\quad\text{and}\quad [B_2,A_2]=B_2\xcirc A_2\xcirc B_2 + \lambda A_2\xcirc B_1.\label{condiciones para [1.14] primer caso}
\end{align}
Then, the  maps $\blackdiamond \colon  H^{d_1d_2}_{\lambda}\times I\to I$ and $\Yleft \colon  I \times H^{d_1d_2}_{\lambda}\to I$, defined by
\begin{equation*}
(1,0)^{\times a}\times (0,1)^{\times a'} \blackdiamond y=A_1^a\xcirc A_2^{a'} y\quad\text{and}\quad	y\Yleft (a,a')\coloneqq aB_1y+a'B_2y,
\end{equation*}
satisfy the conditions~\eqref{eqq 1}--\eqref{eqq 3}. Conversely, if $\blackdiamond \colon  H^{d_1d_2}_{\lambda}\times I\to I$ and $\Yleft \colon  I \times H^{d_1d_2}_{\lambda}\to I$ are maps satisfying these conditions, then the endomorphisms $A_1$, $A_2$, $B_1$ and $B_2$, of $I$, given by
\begin{equation*}
A_1 y \coloneqq (1,0)\blackdiamond y,\quad A_2 y \coloneqq (0,1)\blackdiamond y,\quad B_1 y  \coloneqq y \Yleft (1,0)\quad\text{and}\quad B_2 y  \coloneqq y \Yleft (0,1),
\end{equation*}
satisfy~\eqref{Propiedades de A1 y A2, B1, B2 primer caso}--\eqref{condiciones para [1.14] primer caso}.
\end{proposition}

\begin{proof} $\Rightarrow)$\enspace Since $d_1\mid \lambda\binom{d_2}{2}$, from Remark~\ref{caso d2 impar} it follows that $(\mathrm{U}_{d_1d_2},\times)$ is generated, as an abelian group, by~$(1,0)$ and $(0,1)$, subject to the relations $(1,0)^{\times d_1} = (0,1)^{\times d_2} = (0,0)$. Hence, since $A_1^{d_1}=A_2^{d_2}=\Ide$ and $A_1\xcirc A_2=A_2\xcirc A_1$, the map $\blackdiamond$ is a well defined $\mathds{Z}$-linear left group action of $(\mathrm{U}_{d_1d_2},\times)$ on~$I$ (in other words conditions~\eqref{eqq 2} are satisfied). Similarly, since conditions~\eqref{primeras propiedades de A-AB primer caso} are satisfied, there is a well defined $\mathds{Z}$-linear right group action $\triangleleft$, of $(\mathrm{U}_{d_1d_2},\times)$ on~$I$ (in other words $\triangleleft$ satisfy conditions~\eqref{eqq 1}), given by
\begin{equation}\label{formula de triangleleft}
y\triangleleft (1,0)^{\times a}\times (0,1)^{\times a'} \coloneqq (A_1 -A_1\xcirc B_1)^a\xcirc (A_2 -A_2\xcirc B_2)^{a'}y.
\end{equation}
Moreover, the map $\Yleft$ is well defined, because $d_1B_1=d_2B_2=0$. Recall, from Proposition~\ref{condiciones para extension}, that
\begin{equation*}
y^{(a,a')} \coloneqq y - y\Yleft (a,a')= (\Ide - aB_1 -a'B_2)y.
\end{equation*}
A direct computation proves that this operation satisfies~\eqref{eqq 3}. To finish the proof it remains to check that
\begin{equation*}
(a,a')\blackdiamond (y - y\Yleft (a,a')) =y\triangleleft (a,a')\qquad\text{for all $(a,a')\in \mathrm{U}_{d_1d_2}$ and $y\in I$.}
\end{equation*}
Concretely, by~\eqref{combinacion}, we must show that
\begin{equation}\label{equivalencia entre las dos definiciones}
A_1^{a-\lambda \binom{a'}{2}}A_2^{a'}(\Ide - aB_1 -a'B_2)= (A_1 -A_1B_1)^{a-\lambda \binom{a'}{2}}(A_2 -A_2B_2)^{a'} \quad\text{for all $(a,a')\in \mathrm{U}_{d_1d_2}$.}
\end{equation}
Clearly, the first condition in~\eqref{condiciones para [1.14] primer caso} can be rewritten as
\begin{equation}\label{otra forma de condicion para (2,0)}
B_1(A_1-A_1B_1)=A_1B_1.
\end{equation}
On the basis of this identity, we shall now prove that
\begin{equation}\label{para k}
(A_1-A_1B_1)^k= A_1^k-k A_1^kB_1\quad\text{for all $k\in \mathds{Z}$.}
\end{equation}
For $k=0$ and $k=1$, this is trivial. Assume that it is true for some $k\ge 1$. Then
\begin{align*}
(A_1-A_1B_1)^{k+1}&=\left(A_1^k-kA_1^k B_1\right)(A_1-A_1B_1)\\
&=A_1^{k+1}-A_1^{k+1}B_1-kA_1^k B_1(A_1-A_1B_1)\\
&= A_1^{k+1}-A_1^{k+1}B_1-kA_1^k A_1B_1 && \text{by~\eqref{otra forma de condicion para (2,0)}}\\
&= A_1^{k+1}-(k+1) A_1^{k+1}B_1,
\end{align*}
which establishes equality~\eqref{para k} for all $k\in \mathds{N}_0$, and hence for $k\in \mathds{Z}$, since
$$
(A_1-A_1B_1)^{d_1} = \Ide = A_1^{d_1} = A_1^{d_1}-d_1 A_1^{d_1}B_1.
$$
By~\eqref{para k}, equality~\eqref{equivalencia entre las dos definiciones} is equivalent to
\begin{equation*}
A_1^{a-\lambda \binom{a'}{2}}A_2^{a'}(\Ide-aB_1-a'B_2) = A_1^{a-\lambda\binom{a'}{2}}\left(\Ide - \left(a-\lambda\binom{a'}{2}\right)B_1\right)(A_2 -A_2B_2)^{a'}.
\end{equation*}
Cancelling $A_1^{a-\lambda \binom{a'}{2}}$, we obtain the equivalent condition
\begin{equation}\label{paquito}
A_2^{a'}(\Ide - aB_1 -a'B_2) = \left(\Ide - \left(a-\lambda\binom{a'}{2}\right)B_1\right)(A_2 -A_2B_2)^{a'}.
\end{equation}
On the other hand, the second condition in~\eqref{condiciones para [1.14] primer caso} states that $B_1(A_2-A_2B_2)=A_2B_1$, which immediately implies
\begin{equation}\label{otra forma de condicion para (1,1)}
B_1(A_2-A_2B_2)^k=A_2^kB_1 \quad\text{for all $k\in \mathds{N}_0$.}
\end{equation}
We claim that
\begin{equation}\label{condicion mas dificil}
(A_2 -A_2B_2)^k = A_2^k-\lambda \binom{k}{2}A_2^kB_1 -kA_2^kB_2\qquad\text{for all $k\in \mathds{N}_0$.}
\end{equation}
For $k=0$, this is trivially true. Assume it is true for some $k\ge 0$. Then,
\begin{align*}
(A_2-A_2B_2)^{k+1}&=\left(A_2^k-\lambda\binom{k}{2} A_2^k B_1-kA_2^k B_2\right)(A_2-A_2B_2)\\
&=A_2^{k+1}-A_2^{k+1}B_2-\lambda\binom{k}{2} A_2^k B_1(A_2-A_2B_2)-kA_2^k B_2(A_2-A_2B_2)\\
&=A_2^{k+1}-A_2^{k+1}B_2-\lambda\binom{k}{2} A_2^k A_2B_1-kA_2^k A_2(B_2+\lambda B_1) &&\text{by~\eqref{condiciones para [1.14] primer caso}}\\
&= A_2^{k+1}-\lambda\left(\binom{k}{2} +k\right) A_2^{k+1} B_1-(k+1) A_2^{k+1}B_2\\
&= A_2^{k+1}-\lambda\binom{k+1}{2} A_2^{k+1} B_1-(k+1) A_2^{k+1}B_2,
\end{align*}
as desired. Finally, in view of~\eqref{otra forma de condicion para (1,1)} and~\eqref{condicion mas dificil}, equality~\eqref{paquito} becomes
\begin{equation*}
A_2^{a'}-a A_2^{a'}\xcirc B_1-{a'}A_2^{a'}\xcirc B_2 = A_2^{a'}-\lambda\binom{a'}{2} A_2^{a'}\xcirc  B_1-a'A_2^{a'}\xcirc  B_2-\left(a-\lambda\binom{a'}{2}\right)A_2^{a'}\xcirc  B_1,
\end{equation*}
which is trivially true.

\smallskip
	
\noindent $\Leftarrow)$\enspace By equality~\eqref{potencia} we have $(1,0)^{\times d_1} = (d_1,0) = (0,0)$ and $(0,1)^{\times d_2} = \bigl(\lambda\binom{d_2}{2},d_2\bigr) = (0,0)$, because $d_1\mid \lambda\binom{d_2}{2}$. Since,~more\-over $(\mathrm{U}_{d_1d_2},\times)$ is an abelian group,
by~\eqref{eqq 2}, we have
\begin{equation*}
A_1^{d_1}=A_2^{d_2}= \Ide\quad\text{and}\quad A_1A_2=A_2A_1.
\end{equation*}
Furthermore, by the bilinearity of $\Yleft$ and the definitions of $B_1$ and $B_2$, we have
\begin{equation*}
d_1B_1 y= d_1 (y\Yleft (1,0))=  y\Yleft (d_1,0)=  0 \quad\text{and}\quad d_2B_2 y= d_2 (y\Yleft (0,1)) =  y\Yleft (0,d_2)= 0\qquad\text{for all $y\in I$.}
\end{equation*}	
Thus, the identities in~\eqref{Propiedades de A1 y A2, B1, B2 primer caso} hold. On the other hand, the identities in~\eqref{primeras propiedades de A-AB primer caso} follow directly from~\eqref{eqq 1}, together with the facts that
$$
(1,0)^{\times d_1} = (0,1)^{\times d_2} = (0,0)\quad\text{and}\quad (1,0)\times (0,1) = (0,1)\times (1,0)
$$
and that, by definition we have
$$
y\triangleleft (1,0) = (A_1 -A_1B_1)y\quad\text{and}\quad y\triangleleft (0,1) = (A_2 -A_2B_2)y.
$$
It remains to verify the three identities in~\eqref{condiciones para [1.14] primer caso}. Since these follow directly from equality~\eqref{equivalencia entre las dos definiciones} with $(a,a') = (2,0)$, $(a,a') = (1,1)$, and $(a,a') = (\lambda,2)$, the problem reduces to proving~\eqref{equivalencia entre las dos definiciones}. But this holds, because
\begin{align*}
A_1^{a-\lambda \binom{a'}{2}}A_2^{a'}(y - aB_1 y -a'B_2 y) & = (a,a') \blackdiamond (y - aB_1 y -a'B_2 y)\\
& = (a,a') \blackdiamond (y - y\Yleft (a,a'))\\
& = y \triangleleft (1,0)^{\times a-\lambda \binom{a'}{2}}\times (0,1)^{\times a'}\\
& = (A_1 -A_1B_1)^{a-\lambda \binom{a'}{2}}(A_2 -A_2B_2)^{a'} y,
\end{align*}
where the first equality holds by the first condition in~\eqref{eqq 2} and equality~\eqref{combinacion}; the second one, since $\Yleft$ is bilinear; the third one, by the definition of~$\triangleleft$ and equality~\eqref{combinacion}; and the last one, by the first equal\-ity in~\eqref{eqq 1}.
\end{proof}

\begin{proposition}\label{caracterizacion segundo caso} Assume that $\lambda=1$, $d_1=2$ and $v_2(d_2)=1$. Let $A$, $B_1$ and $B_2$ be endomorphisms of $I$ such that
\begin{align}
& A^{2d_2}=\Ide,\quad 2B_1=d_2B_2=0,\quad (A- AB_2)^{2d_2}=\Ide,\label{primeras condiciones A, B1 y B2 segundo caso}\\
& [B_2,A]= AB_1+ B_2AB_2\quad\text{and}\quad [B_1,A]=B_1AB_2.\label{segundas condiciones A, B1 y B2 segundo caso}
\end{align}	
Then, the  maps $\blackdiamond \colon  H^{d_1d_2}_{\lambda}\times I\to I$ and $\Yleft \colon  I \times H^{d_1d_2}_{\lambda}\to I$, defined by
\begin{equation*}
(0,1)^{\times{a}}\blackdiamond y\coloneqq A^a y \quad\text{and}\quad	y\Yleft (a,a')= (aB_1+ a'B_2)y,
\end{equation*}
satisfy the conditions~\eqref{eqq 1}--\eqref{eqq 3}. Conversely, if $\blackdiamond \colon  H^{d_1d_2}_{\lambda}\times I\to I$ and $\Yleft \colon  I \times H^{d_1d_2}_{\lambda}\to I$ are maps satisfying these conditions, then the endomorphisms $A$, $B_1$ and $B_2$, of $I$, given by	
\begin{equation*}
A y \coloneqq (0,1)\blackdiamond y, \quad B_1 y \coloneqq y\Yleft (1,0)\quad\text{and}\quad B_2 y \coloneqq y\Yleft (0,1),
\end{equation*}	
satisfy~\eqref{primeras condiciones A, B1 y B2 segundo caso} and~\eqref{segundas condiciones A, B1 y B2 segundo caso}.
\end{proposition}

\begin{proof} $\Rightarrow)$\enspace By Remark~\ref{caso v_2=1}, the group $(\mathrm{U}_{d_1d_2},\times)$ is generated by $(0,1)$, subject to the relation $(0,1)^{\times 2d_2} = (0,0)$.~Con\-sequently, since $A^{2d_2}=\Ide$, the map $\blackdiamond$ is a well defined $\mathds{Z}$-linear left group action of $(\mathrm{U}_{d_1d_2},\times)$ on $I$ (that is, con\-ditions~\eqref{eqq 2} are satisfied). Similarly, since $(A-AB_2)^{2d_2}=\Ide$, there is a well defined $\mathds{Z}$-linear right group action~$\triangleleft$, of $(\mathrm{U}_{d_1d_2},\times)$ on $I$ (in other words $\triangleleft$ satisfy conditions~\eqref{eqq 1}), given by
\begin{equation}\label{formula de triangleleft'}
y\triangleleft (0,1)^{\times a} \coloneqq (A-AB_2)^ay.
\end{equation}
Moreover, the map $\Yleft$ is well defined, because $2B_1=d_2B_2=0$. Recall, from Proposition~\ref{condiciones para extension}, that
\begin{equation*}
y^{(a,a')} \coloneqq y - y\Yleft (a,a')= (\Ide - aB_1 -a'B_2)y.
\end{equation*}
A direct computation proves that this operation satisfies~\eqref{eqq 3}. To finish the proof it remains to check that
\begin{equation*}
(a,a')\blackdiamond (y - y\Yleft (a,a')) =y\triangleleft (a,a')\qquad\text{for all $(a,a')\in \mathrm{U}_{d_1d_2}$ and $y\in I$.}
\end{equation*}
Since $(0,1)$ is a generator of $(\mathrm{U}_{d_1d_2},\times)$, of order $|(0,1)| = 2d_2$; given $(a,a')\in H_{\lambda}^{d_1d_2}$, there exists a unique $n(a,a')\in \mathds{N}$, such that $1\le n(a,a')\le 2d_2$ and $(0,1)^{\times n(a,a')}=(a,a')$. With this notation the previous conditions becomes
\begin{equation}\label{equivalencia entre las dos definiciones segundo caso}
A^{n(a,a')}(\Ide - aB_1-a'B_2)=(A-AB_2)^{n(a,a')} \quad\text{for all $(a,a')\in H_{\lambda}^{d_1d_2}$.}
\end{equation}
For $(a,a')=(0,1)$ this condition is trivial. Assume it is true for some $(a,a')\coloneqq (0,1)^{\times n}$ with $1\le n = n(a,a')<2d_2$. We want to prove that~\eqref{equivalencia entre las dos definiciones segundo caso} is true for $(0,1)^{\times n+1}$. Since, by equality~\eqref{producto},
\begin{equation}\label{paquito1}
(0,1)^{\times n+1} = (a+a',a'+1),
\end{equation}
we must prove that
\begin{equation*}
A^{n+1}(\Ide -(a+a')B_1 - (a'+1)B_2)=(A-AB_2)^{n+1}.
\end{equation*}
But since, on the one hand,
\begin{equation*}
A^{n+1}(\Ide -(a+a')B_1 - (a'+1)B_2)=A^{n+1}(\Ide -aB_1 - a'B_2) - A^{n+1}(a'B_1+B_2)
\end{equation*}
and, on the other hand,
\begin{equation*}
(A-AB_2)^{n+1}= A(A-AB_2)^n- AB_2(A-AB_2)^n,
\end{equation*}
by the inductive hypothesis, we are reduce to prove that
\begin{equation}\label{paso necesario}
A^n(a'B_1+B_2)= B_2(A-AB_2)^n.
\end{equation}
We prove this, for $1\le n\le 2d_2$, by induction on $n$. For $(a,a')=(0,1)$, this equality holds by the first condition in~\eqref{segundas condiciones A, B1 y B2 segundo caso}. Assume it is true for some $(a,a')\coloneqq (0,1)^{\times n}$ with $1\le n = n(a,a')<2d_2$. By~\eqref{paquito1}, we must prove that
\begin{equation*}
A^{n+1}((a'+1)B_1+B_2)= B_2(A-AB_2)^{n+1}.
\end{equation*}
But, since
\begin{equation*}
A^{n+1}((a'+1)B_1+B_2) =  A^{n+1}(a'B_1+B_2)+ A^{n+1}B_1
\end{equation*}
and
\begin{align*}
B_2(A-AB_2)^{n+1}& = (B_2A-B_2AB_2)(A-AB_2)^n\\
& = (AB_2+AB_1)(A-AB_2)^n &&\text{by~\eqref{segundas condiciones A, B1 y B2 segundo caso}}\\
&= AB_2(A-AB_2)^n + AB_1(A-AB_2)^n,
\end{align*}
by the inductive hypothesis, it suffices to prove that 	
\begin{equation*}
A^nB_1=B_1(A-AB_2)^n\quad\text{for all $n\in \mathds{N}$}.
\end{equation*}
Again, we proceed by induction on $n$. For $n=1$ this equality follows from the second condition in~\eqref{segundas condiciones A, B1 y B2 segundo caso}. Assume it is true for some $n\ge 1$. Then, by the inductive hypothesis and the second condition in~\eqref{segundas condiciones A, B1 y B2 segundo caso},
\begin{equation*}
A^{n+1}B_1= AB_1(A-AB_2)^n= B_1(A-AB_2)^{n+1},
\end{equation*}
as desired.
\smallskip

\noindent $\Leftarrow)$\enspace Since $(0,1)^{\times 2d_2} =  (0,0)$, by~\eqref{eqq 2}, we have $A^{2d_2}=\Ide$. Moreover, by the bilinearity of $\Yleft$ and the definitions of~$B_1$ and $B_2$,
\begin{equation*}
2B_1 y= 2 (y\Yleft (1,0)) =  y\Yleft (2,0) = 0 \quad\text{and}\quad d_2B_2 y = d_2 (y\Yleft (0,1)) =  y\Yleft (0,d_2) = 0\quad\text{for all $y\in I$.}
\end{equation*}
Furthermore, the identity $(A-AB_2)^{2d_2}=\Ide$ follows from the first condition in~\eqref{eqq 1}, together with the facts that $(0,1)^{\times 2d_2} = (0,0)$ and, by def\-i\-ni\-tion, $y\triangleleft (0,1) = (A -AB_2)y$. Thus, the identities in~\eqref{primeras condiciones A, B1 y B2 segundo caso} hold. It remains to verify the identities in~\eqref{segundas condiciones A, B1 y B2 segundo caso}. We claim that these follow from~\eqref{equivalencia entre las dos definiciones segundo caso}, which holds because
\begin{align*}
A^{n(a,a')}(y - aB_1 y -a'B_2 y) & = (a,a') \blackdiamond (y - aB_1 y -a'B_2 y) &&\text{by~\eqref{eqq 2}, since $(0,1)^{\times n(a,a')} =(a,a')$}\\
& = (a,a') \blackdiamond (y - y\Yleft (a,a')) &&\text{since $\Yleft$ is bilinear}\\
& = y \triangleleft (a,a') && \text{by the definition of $\triangleleft$}\\
& = (A-AB_2)^{n(a,a')}y. &&\text{by~\eqref{eqq 1}, since $(0,1)^{\times n(a,a')} =(a,a')$}
\end{align*}
In fact, applying~\eqref{equivalencia entre las dos definiciones segundo caso} with $(a,a')=(0,1)^{\times 2} =(1,2)$, and using the invertibility of $A$, we obtain
$$
A (\Ide - B_1 - 2B_2) = (\Ide-B_2)(A-AB_2) = A-AB_2-B_2A+B_2AB_2,
$$
which proves the first equality in~\eqref{segundas condiciones A, B1 y B2 segundo caso}. Then, by~\eqref{equivalencia entre las dos definiciones segundo caso} with $(a,a')=(0,1)^{\times 3} =(3,3)$, we have
\begin{align*}
A^3(\Ide - 3B_1 - 3B_2) = (A -AB_2)^3 & \Leftrightarrow A^3(\Ide - B_1 - 3B_2) = A^2 (\Ide - B_1 - 2B_2) (A-AB_2)\\
& \Leftrightarrow A - AB_1 - 3AB_2 = A-AB_2 - B_1A+B_1AB_2 - 2B_2A + 2B_2AB_2\\
& \Leftrightarrow B_1A - AB_1 = B_1AB_2  + 2AB_2- 2B_2A + 2B_2AB_2\\
& \Leftrightarrow B_1A - AB_1 = B_1AB_2  - 2AB_1\\
& \Leftrightarrow B_1A - AB_1 = B_1AB_2,
\end{align*}
where the first equivalence follows from $2B_1 = 0$ and the identity $A^2(\Ide - B_1 - 2B_2) = (A - AB_2)^2$; the second equivalence uses the invertibility of $A$; the fourth one relies on the fact that, by the first equality in~\eqref{segundas condiciones A, B1 y B2 segundo caso},
$$
2AB_2 - 2B_2A + 2B_2AB_2 = -2AB_1;
$$
and the last equivalence holds since $2B_1=0$.
\end{proof}

\begin{remark}\label{tres20} Let $A,B_1,B_2$ be such that $[B_1,A] = B_1AB_2$. An inductive argument using $AB_1=B_1(A-AB_2)$,~shows that $A^nB_1 = B_1(A-AB_2)^n$, for all $n\in \mathds{N}$.
\end{remark}

\begin{proposition}\label{caracterizacion tercer caso} Assume that $\lambda$ is odd, $v_2(d_1)=v_2(d_2)=1$ and $d_1>2$. Let $A_1,A_2,B_1,B_2 \in\End(I)$ be such that
\begin{align}
& A_1^{d_1/2} = A_2^{2d_2}=\Ide,\quad [A_1,A_2]=0,\quad d_1B_1=d_2B_2=0,\label{primeras condiciones A_1, A_2, B1 y B2 tercer caso}\\
&(A_1 - 2A_1B_1)^{d_1/2}=(A_2 -A_2B_2)^{2d_2}=\Ide,\quad [A_1 - 2A_1B_1,A_2 -A_2B_2]=0,\label{segundas condiciones A_1, A_2, B1 y B2 tercer caso}\\
& [B_2,A_2]=B_2A_2B_2+\lambda A_2B_1,\quad [B_1,A_2] = B_1A_2B_2\quad\text{and}\quad 2[B_1,A_1] = 4B_1A_1B_1.\label{terceras condiciones A_1, A_2, B1 y B2 tercer caso}
\end{align}	
Then, the maps $\blackdiamond\colon H^{d_1d_2}_{\lambda}\times I\to I$ and $\Yleft\colon I\times H^{d_1d_2}_{\lambda}\to I$, defined by
\begin{equation*}
(2,0)^{\times a}\times (0,1)^{\times a'}\blackdiamond y\coloneqq A_1^a A_2^{a'}y\quad\text{and}\quad y\Yleft (a,a')\coloneqq (aB_1+ a'B_2)y,
\end{equation*}
satisfy the conditions~\eqref{eqq 1}--\eqref{eqq 3}. Conversely, if $\blackdiamond \colon H^{d_1d_2}_{\lambda}\times I\to I$ and $\Yleft \colon I\times H^{d_1d_2}_{\lambda}\to I$ are maps satisfying these conditions, then the endomorphisms $A_1$, $A_2$, $B_1$ and $B_2$, of $I$, given by
\begin{equation*}
A_1 y \coloneqq (2,0)\blackdiamond y,\quad A_2 y \coloneqq (0,1)\blackdiamond y,\quad B_1 y\coloneqq y \Yleft (1,0)\quad\text{and}\quad B_2 y\coloneqq y \Yleft (0,1),
\end{equation*}
satisfy~\eqref{primeras condiciones A_1, A_2, B1 y B2 tercer caso}--\eqref{terceras condiciones A_1, A_2, B1 y B2 tercer caso}.
\end{proposition}

\begin{proof} $\Rightarrow)$\enspace By Remark~\ref{caso v_2=1}, the abelian group $(\mathrm{U}_{d_1d_2},\times)$ is generated by the elements $(2,0)$ and $(0,1)$, subject to the relations $(2,0)^{\times d_1/2} = (0,1)^{\times 2d_2} = (0,0)$. Thus, since $A_1^{d_1/2} = A_2^{2d_2} = \Ide$ and $[A_1, A_2] = 0$, the map~$\blackdiamond$ defines a well–posed $\mathds{Z}$-linear left group action of $(\mathrm{U}_{d_1d_2}, \times)$ on $I$. Therefore, conditions~\eqref{eqq 2} are satisfied.
Similarly, since conditions~\eqref{segundas condiciones A_1, A_2, B1 y B2 tercer caso} are satisfied, there exists a well–defined $\mathds{Z}$-linear right group action $\triangleleft$ of $(\mathrm{U}_{d_1d_2},\times)$ on $I$, given by
\begin{equation}\label{formula de triangleleft''}
y\triangleleft (2,0)^{\times a}\times (0,1)^{\times a'} \coloneqq (A_1 - 2A_1B_1)^a(A_2 -A_2B_2)^{a'} y.
\end{equation}
Thus, $\triangleleft$ satisfies conditions~\eqref{eqq 1}. Moreover, the map $\Yleft$ is well–defined, since $d_1B_1 = d_2B_2 = 0$. Recall that
\begin{equation*}
y^{(a,a')} \coloneqq y - y\Yleft (a,a')= (\Ide - aB_1 -a'B_2)y.
\end{equation*}
A direct computation proves that this operation satisfies~\eqref{eqq 3}. To finish the proof it remains to check that
\begin{equation}\label{pepe}
(a,a')\blackdiamond (y - y\Yleft (a,a')) =y\triangleleft (a,a')\qquad\text{for all $(a,a')\in \mathrm{U}_{d_1d_2}$ and $y\in I$.}
\end{equation}
As follows from the discussion at the beginning of this proof, for each $(a,a') \in \mathrm{U}_{d_1d_2}$ there exist unique non-negative integers $n_1(a,a')$ and $n_2(a,a')$ such that
$$
0 \leq n_1(a,a') < d_1/2,\quad 1 \leq n_2(a,a') \leq 2d_2\quad\text{and}\quad (2,0)^{\times n_1(a,a')} \times (0,1)^{\times n_2(a,a')} = (a,a').
$$
Then, condition~\eqref{pepe} becomes
\begin{equation}\label{equivalencia entre las dos definiciones3}
A_1^{n_1(a,a')}A_2^{n_2(a,a')}(\Ide - aB_1-a'B_2)= (A_1 - 2A_1 B_1)^{n_1(a,a')}(A_2 - A_2 B_2)^{n_2(a,a')}.
\end{equation}
First we prove this for $(a,a')\coloneqq (0,1)^{\times n}$ with $1\le n = n_2(a,a')\le 2d_2$, by induction on $n$. When $n=1$ this is trivially true. Assume it~is true for some $(a,a')\coloneqq (0,1)^{\times n}$ with $1\le n <2d_2$. Since, by equality~\eqref{producto},
\begin{equation*}
(0,1)^{\times n+1} = \left(a+\lambda a',a'+1\right),
\end{equation*}
we must verify that
\begin{equation*}
A_2^{n+1}\left(\Ide-(a+\lambda a')B_1-(a'+1)B_2\right) = (A_2-A_2B_2)^{n+1}.
\end{equation*}
But, since
\begin{equation*}
A_2^{n+1}\left(\Ide-(a+\lambda a')B_1-(a'+1)B_2\right) = A_2^{n+1}\left(\Ide-aB_1-a'B_2\right) - A_2^{n+1}(\lambda a'B_1+B_2)
\end{equation*}
and
\begin{equation*}
(A_2-A_2B_2)^{n+1} = A_2(A_2-A_2B_2)^n - A_2B_2(A_2-A_2B_2)^n,
\end{equation*}
by the inductive hypothesis and the invertibility of $A_2$, it suffices to prove that 	
\begin{equation*}
A_2^n(\lambda a'B_1+B_2) = B_2(A_2-A_2B_2)^n.
\end{equation*}
We proceed by induction on $n$. For $n=1$, we have $(a,a') = (0,1)$, and the equality follows from the first identity in~\eqref{terceras condiciones A_1, A_2, B1 y B2 tercer caso}. Assume it holds for some $(a,a') \coloneqq (0,1)^{\times n}$ with $1\leq n < 2d_2$. Then, on the one hand, we have
\begin{equation*}
A_2^{n+1}\bigl(\lambda (a'+1)B_1 + B_2\bigr) = A_2^{n+1}\bigl(\lambda a'B_1 + B_2\bigr) + \lambda A_2^{n+1}B_1 = A_2B_2(A_2 - A_2B_2)^n + \lambda A_2^{n+1}B_1,
\end{equation*}
where the last equality follows from the inductive hypothesis; while, on the other hand,
\begin{equation*}
B_2(A_2-A_2B_2)^{n+1} = (B_2A_2-B_2A_2B_2)(A_2-A_2B_2)^n = (A_2B_2+\lambda A_2B_1) (A_2-A_2B_2)^n,
\end{equation*}
where the last equality holds again by the first identity in~\eqref{terceras condiciones A_1, A_2, B1 y B2 tercer caso}. Hence, it suffices to prove that
\begin{equation}\label{caso con alpha}
A_2^nB_1 = B_1(A_2-A_2B_2)^n\quad\text{for all $n\in \mathds{N}$.}
\end{equation}
But this follows from  the second identity in~\eqref{terceras condiciones A_1, A_2, B1 y B2 tercer caso} together with Remark~\ref{tres20}. Consequently, equality~\eqref{equivalencia entre las dos definiciones3} holds, for $(a,a') = (0,1)^{\times n}$ with $1\le n = n(a,a')\le 2d_2$. In order to complete the proof, we must verify
 equality~\eqref{equivalencia entre las dos definiciones3} for~$(a,a') = (2,0)^{\times n_1} \times (0,1)^{\times n_2}$, with $0 \leq n_1 < d_1/2$ and $1 \leq n_2 \leq 2d_2$. We
 proceed by induction on $n_1$. We already know that the statement holds when $n_1=0$ and $1\le n_2\le 2d_2$. Assume it holds for some
 $(a,a') = (2,0)^{\times n_1}\times (0,1)^{\times n_2}$ with $0\le n_1<\frac{d_1}{2}-1$ and~$1\le n_2\le 2d_2$. Since, by equality~\eqref{producto},
$$
(2,0)^{\times n_1+1}\times (0,1)^{\times n_2}=\left(a+2,a'\right),
$$
we must verify that
\begin{equation}\label{ecua4}
A_1^{n_1+1}A_2^{n_2}\left(\Ide-(a+2)B_1-a'B_2\right) = (A_1-2A_1B_1)^{n_1+1}(A_2-A_2B_2)^{n_2}.
\end{equation}
But, since
\begin{equation*}
A_1^{n_1+1}A_2^{n_2}\left(\Ide-\left(a+2\right)B_1-a'B_2\right) = A_1^{n_1+1}A_2^{n_2}\left(\Ide-aB_1-a'B_2 - 2B_1\right)
\end{equation*}
and
\begin{equation*}
(A_1\!-\!2A_1B_1)^{n_1+1}(A_2\!-\!A_2B_2)^{n_2} = A_1(A_1\!-\!2A_1B_1)^{n_1}(A_2\!-\!A_2B_2)^{n_2} - 2A_1B_1(A_1\!-\!2A_1B_1)^{n_1}(A_2\!-\!A_2B_2)^{n_2},
\end{equation*}
to prove~\eqref{ecua4}, it suffices, by the inductive hypothesis, to show that
\begin{equation}\label{paso necesario2''}
2A_1^{n_1}A_2^{n_2}B_1 = 2B_1(A_1-2A_1B_1)^{n_1}(A_2-A_2B_2)^{n_2} \quad\text{for all $n_1\in \mathds{N}_0$ and $n_2\in \mathds{N}$.}
\end{equation}
When $n_1= 0$, this follows from~\eqref{caso con alpha}. Assume that~\eqref{paso necesario2''} holds for some $n_1\in \mathds{N}_0$. Then,
\begin{equation*}
2A_1^{n_1+1}A_2^{n_2}B_1 = 2A_1B_1(A_1\!-\!2A_1B_1)^{n_1}(A_2\!-\!A_2B_2)^{n_2} = 2B_1(A_1\!-\!2A_1B_1)^{n_1+1}(A_2\!-\!A_2B_2)^{n_2},
\end{equation*}
where the last equality follows from the final identity in~\eqref{terceras condiciones A_1, A_2, B1 y B2 tercer caso}.

\smallskip

\noindent $\Leftarrow)$\enspace Since $(2,0)^{\times d_1/2} = (0,1)^{\times 2d_2} = (0,0)$ and $(\mathrm{U}_{d_1d_2},\times)$ is an abelian group, it follows from~\eqref{eqq 2}, that
\begin{equation*}
A_1^{d_1/2}=A_2^{2d_2}= \Ide\quad\text{and}\quad A_1A_2=A_2A_1.
\end{equation*}
Furthermore, by the bilinearity of $\Yleft$ and the definitions of $B_1$ and $B_2$, we obtain
\begin{equation*}
d_1B_1 y= d_1 (y\Yleft (1,0))=  y\Yleft (d_1,0)=  0 \quad\text{and}\quad d_2B_2 y= d_2 (y\Yleft (0,1)) =  y\Yleft (0,d_2)= 0\quad\text{for all $y\in I$.}
\end{equation*}	
Thus, the identities in~\eqref{primeras condiciones A_1, A_2, B1 y B2 tercer caso} are satisfied. On the other hand, the identities in~\eqref{segundas condiciones A_1, A_2, B1 y B2 tercer caso} follow directly from~\eqref{eqq 1}, since, by definition, $y\triangleleft (2,0)\coloneqq (A_1 - 2A_1B_1)y$ and $y\triangleleft (0,1) \coloneqq(A_2 -A_2B_2)y$. It remains to verify the identities in~\eqref{terceras condiciones A_1, A_2, B1 y B2 tercer caso}. We claim that these identities follow from condition~\eqref{equivalencia entre las dos definiciones3}, which holds since
\begin{align*}
A_1^{n_1(a,a')}A_2^{n_2(a,a')}(y - aB_1 y -a'B_2 y) & = (a,a') \blackdiamond (y - aB_1 y -a'B_2 y) &&\text{by~\eqref{eqq 2}}\\
& = (a,a') \blackdiamond (y - y\Yleft (a,a')) &&\text{since $\Yleft$ is bilinear}\\
& = y \triangleleft (2,0)^{\times n_1(a,a')}\times (0,1)^{\times n_2(a,a')} && \text{by the definition of $\triangleleft$}\\
& = (A_1-2A_1B_1)^{n_1(a,a')}(A_2-A_2B_2)^{n_2(a,a')} y. &&\text{by~\eqref{eqq 1}}
\end{align*}
In fact, applying~\eqref{equivalencia entre las dos definiciones3} with $(a,a')=(0,1)^{\times 2}=(\lambda,2)$, and using the invertibility of $A_2$, we obtain
$$
A_2(\Ide - \lambda B_1 - 2B_2) = (\Ide-B_2)(A_2-A_2B_2) = A_2-A_2B_2 - B_2A_2 + B_2A_2B_2,
$$
which establishes the first equality in~\eqref{terceras condiciones A_1, A_2, B1 y B2 tercer caso}. Applying condition~\eqref{equivalencia entre las dos definiciones3} with $(a,a')=(0,1)^{\times 3}=(3\lambda,3)$, we obtain
$$
A_2^3(\Ide - 3\lambda B_1 - 3B_2) = (A_2-A_2B_2)^3 = A_2^2(\Ide - \lambda B_1 - 2B_2)(A_2-A_2B_2),
$$
where the last equality follows from $A_2^2(\Ide - \lambda B_1 - 2B_2) = (A_2-A_2B_2)^2$. Consequently,
$$
\lambda B_1A_2- 3\lambda A_2B_1 +2B_2A_2 - 2A_2B_2 = \lambda B_1A_2B_2 +2B_2A_2B_2.
$$
Combining this with the first equality in~\eqref{terceras condiciones A_1, A_2, B1 y B2 tercer caso}, we obtain that
\begin{equation}\label{con alfa'}
\lambda B_1A_2-\lambda A_2B_1 = \lambda B_1A_2B_2.
\end{equation}
Moreover, since $A_1$ is invertible, condition~\eqref{equivalencia entre las dos definiciones3} with $(a,a')=(2,0)\times (0,1)=(2,1)$, implies
$$
A_2(\Ide - 2 B_1 - B_2) = (\Ide-2B_1)(A_2-A_2B_2) = A_2-A_2B_2 - 2B_1A_2 + 2B_1A_2B_2,
$$
which in turns yields
\begin{equation}\label{con dos}
2 B_1A_2 - 2A_2B_1 = 2B_1A_2B_2.
\end{equation}
Since $\lambda$ is odd, combining~\eqref{con alfa'} and~\eqref{con dos}, we obtain the second equality in condition~\eqref{terceras condiciones A_1, A_2, B1 y B2 tercer caso}. Finally, since $A_1$ is invertible, by condition~\eqref{equivalencia entre las dos definiciones3} with $(a,a')=(2,0)^{\times 2}=(4,0)$, we have
$$
A_1(\Ide - 4 B_1) = (\Ide-2B_1)(A_1-2A_1B_1) = A_1-2A_1B_1 - 2B_1 A_1+4 B_1A_1B_1,
$$
which gives the last equality in condition~\eqref{terceras condiciones A_1, A_2, B1 y B2 tercer caso}.
\end{proof}

\begin{proposition}\label{caracterizacion cuarto caso} Assume that $\lambda$ is odd and $v_2(d_1)=v_2(d_2)\ge 2$. Let $A_1,A_2,B_1,B_2 \in \End(I)$ be such that
\begin{align}
& A_1^{d_1/2} = A_2^{2d_2}=\Ide,\quad [A_1,A_2]=0,\quad  d_1B_1=d_2B_2=0,\label{primeras condiciones A_1, A_2, B1 y B2 cuarto caso}\\
& (A_1-2tA_1B_1-2qA_1B_2)^{d_1/2} = (A_2-A_2B_2)^{2d_2}=\Ide,\quad [A_1-2tA_1B_1-2qA_1B_2,A_2-A_2B_2]=0,\label{segundas condiciones A_1, A_2, B1 y B2 cuarto caso}\\
& [B_2,A_2]=B_2A_2B_2+\lambda A_2B_1,\quad [B_1,A_2] = B_1A_2B_2, \quad [B_1,A_1] = 2 tB_1A_1B_1+2 qB_1A_1B_2,\label{terceras condiciones A_1, A_2, B1 y B2 cuarto caso}\\
&\text{and}\quad [B_2,A_1]= 2\lambda qA_1B_1 + 2tB_2A_1B_1+2qB_2A_1B_2,\label{cuartas condiciones A_1, A_2, B1 y B2 cuarto caso}
\end{align}	
where $q$ and $t$ are as in Proposition~\ref{caso v2(d1)=v2(d2) y alfa impar}. Then, the maps $\blackdiamond\colon  H^{d_1d_2}_{\lambda}\times I\to I$ and $\Yleft \colon I\times H^{d_1d_2}_{\lambda}\to I$, defined by
\begin{equation*}
(2t,2q)^{\times a}\times (0,1)^{\times a'}\blackdiamond y\coloneqq A_1^a A_2^{a'} y \quad\text{and}\quad y\Yleft (a,a')\coloneqq (aB_1+ a'B_2)y,
\end{equation*}
satisfy the conditions~\eqref{eqq 1}--\eqref{eqq 3}. Conversely, if $\blackdiamond \colon  H^{d_1d_2}_{\lambda}\times I\to I$ and $\Yleft \colon I\times H^{d_1d_2}_{\lambda}\to I$ are maps satisfying these conditions, then the endomorphisms $A_1$, $A_2$, $B_1$ and $B_2$, of $I$, given by
\begin{equation*}
A_1 y \coloneqq (2t,2q)\blackdiamond y,\quad A_2 y\coloneqq (0,1)\blackdiamond y,\quad B_1 y\coloneqq y \Yleft (1,0)\quad\text{and}\quad B_2 y\coloneqq y \Yleft (0,1),
\end{equation*}
satisfy~\eqref{primeras condiciones A_1, A_2, B1 y B2 cuarto caso}--\eqref{cuartas condiciones A_1, A_2, B1 y B2 cuarto caso}.
\end{proposition}

\begin{proof} $\Rightarrow)$\enspace By Proposition~\ref{caso v2(d1)=v2(d2) y alfa impar}, the abelian group $(\mathrm{U}_{d_1d_2},\times)$ is generated by the elements $(2t,2q)$ and $(0,1)$, subject to the relations $(2t,2q)^{\times d_1/2} = (0,1)^{\times 2d_2} = (0,0)$. Thus, since $A_1^{d_1/2} = A_2^{2d_2}=\Ide$ and $[A_1,A_2] = 0$, the map~$\blackdiamond$ defines a well-posed $\mathds{Z}$-linear left group action of $(\mathrm{U}_{d_1d_2},\times)$ on $I$. Therefore, conditions~\eqref{eqq 2} are satisfied. Similarly, since conditions~\eqref{segundas condiciones A_1, A_2, B1 y B2 cuarto caso} are satisfied, there exists a well-defined $\mathds{Z}$-linear right group action $\triangleleft$, of $(\mathrm{U}_{d_1d_2},\times)$ on $I$, given by
\begin{equation}\label{formula de triangleleft'''}
y\triangleleft (2t,2q)^{\times a}\times (0,1)^{\times a'} \coloneqq (A_1-2tA_1B_1-2qA_1B_2)^a (A_2 -A_2B_2)^{a'}y.
\end{equation}
Thus, $\triangleleft$ satisfy conditions~\eqref{eqq 1}. Moreover, the map $\Yleft$ is well-defined, since $d_1B_1=d_2B_2=0$. Recall that
\begin{equation*}
y^{(a,a')}\coloneqq y - y\Yleft (a,a')= (\Ide - aB_1 -a'B_2)y.
\end{equation*}
A direct computation proves that this operation satisfies~\eqref{eqq 3}. To finish the proof it remains to check that
\begin{equation}\label{a probar}
(a,a')\blackdiamond (y - y\Yleft (a,a')) =y\triangleleft (a,a')\qquad\text{for all $(a,a')\in \mathrm{U}_{d_1d_2}$ and $y\in I$.}
\end{equation}
As follows from the discussion at the beginning of this proof, for each $(a,a') \in \mathrm{U}_{d_1d_2}$ there exist unique non-negative integers $n_1(a,a')$ and $n_2(a,a')$ such that
$$
0\le n_1(a,a')<d_1/2,\quad 1\le n_2(a,a')\le 2d_2\quad\text{and}\quad (2t,2q)^{\times n_1(a,a')}\times (0,1)^{\times n_2(a,a')} =(a,a').
$$
Then, condition~\eqref{a probar} becomes
\begin{equation}\label{equivalencia entre las dos definiciones4}
A_1^{n_1(a,a')}A_2^{n_2(a,a')}(\Ide - aB_1-a'B_2)= (A_1 - 2tA_1 B_1 - 2q A_1B_2)^{n_1(a,a')}(A_2 - A_2 B_2)^{n_2(a,a')}.
\end{equation}
Arguing as in the proof of the preceding proposition, we verify that this equality holds for $(a,a') \coloneqq (0,1)^{\times n}$, with $1\leq n = n_2(a,a') \leq 2d_2$.

Consequently, equality~\eqref{equivalencia entre las dos definiciones4} holds, for $(a,a') = (0,1)^{\times n}$ with $1\le n = n(a,a')\le 2d_2$. In order to complete the proof, we must verify equality~\eqref{equivalencia entre las dos definiciones4} for $(a,a') = (2t,2q)^{\times n_1}\times (0,1)^{\times n_2}$ with $0\le n_1< d_1/2$ and $1\le n_2\le 2d_2$. We proceed by induction on $n_1$. We already know that the statement holds when $n_1=0$ and $1\le n_2\le 2d_2$. Assume it holds for some $(a,a') = (2t,2q)^{\times n_1}\times (0,1)^{\times n_2}$ with $0\le n_1 = n_1(a,a')< \frac{d_1}{2}-1$ and~$1\le n_2\le 2d_2$. Since, by equality~\eqref{producto},
\begin{equation}\label{exponencial de (2t,2s),(0,1)}
(2t,2q)^{\times n_1+1}\times (0,1)^{\times n_2}= \left(2t+a+2\lambda q a',2q+a'\right) = (a,a') + \left(2t+2\lambda q a',2q\right),
\end{equation}
we must check that
\begin{equation}\label{ecua5}
A_1^{n_1+1}A_2^{n_2}\left(\Ide-\left(2t+a+2\lambda q a'\right)B_1-(2q+a')B_2\right) = (A_1-2tA_1B_1-2qA_1B_2)^{n_1+1}(A_2-A_2B_2)^{n_2}.
\end{equation}
But since
\begin{align*}
A_1^{n_1+1}A_2^{n_2}\left(\Ide-\left(2t+a+2\lambda q a'\right)B_1-(2q+a')B_2\right) & = A_1^{n_1+1}A_2^{n_2}\left(\Ide- a B_1- a'B_2\right)\\
& - A_1^{n_1+1}A_2^{n_2}\left(\left(2t+2\lambda q a'\right)B_1+2qB_2\right)
\end{align*}
and
\begin{align*}
(A_1 - 2tA_1 B_1 - 2qA_1 B_2)^{n_1+1}& (A_2 - A_2 B_2)^{n_2} = A_1 (A_1 - 2tA_1 B_1 - 2qA_1 B_2)^{n_1} (A_2 - A_2 B_2)^{n_2}\\
& -  (2tA_1 B_1 + 2qA_1 B_2) (A_1 - 2tA_1 B_1 - 2qA_1 B_2)^{n_1} (A_2 - A_2 B_2)^{n_2},
\end{align*}
by the inductive hypothesis and the invertibility of $A_1$, it suffices to verify that, for all $n_1,n_2\in \mathds{N}_0$,
\begin{equation}\label{paso necesario2'''}
A_1^{n_1}A_2^{n_2}\left(\left(2t+2\lambda q a'\right)B_1+2qB_2\right) = (2t B_1 + 2qB_2) (A_1 - 2tA_1 B_1 - 2qA_1 B_2)^{n_1} (A_2 - A_2 B_2)^{n_2}.
\end{equation}
We proceed by induction on $n_1$. Since, by~\eqref{potencia}, we have $(0,1)^{\times n_2}\! = \bigl(\lambda\binom{n_2}{2},n_2\bigr)$, when $n_1\!=0$, equality~\eqref{paso necesario2'''} becomes
\begin{equation}\label{paso necesario2''''}
A_2^{n_2}\left(\left(2t + 2\lambda q n_2\right)B_1 + 2qB_2\right) = (2t B_1 + 2q B_2)(A_2 - A_2 B_2)^{n_2}.
\end{equation}
For $n_2 = 0$, this is trivial. Assume it is true for some $n_2\ge 0$. Since
\begin{equation*}
A_2^{n_2+1}\left(\left(2t + 2\lambda q(n_2+1) \right)B_1 + 2qB_2\right) = A_2^{n_2+1}\left(\left(2t + 2\lambda qn_2 \right)B_1 + 2qB_2\right) + 2\lambda q A_2^{n_2+1}B_1
\end{equation*}
and, by the first two conditions of~\eqref{terceras condiciones A_1, A_2, B1 y B2 cuarto caso},
$$
(2t B_1 + 2q B_2) (A_2 - A_2 B_2)^{n_2+1} = (2t A_2B_1 + 2q A_2B_2) (A_2 - A_2 B_2)^{n_2} + 2\lambda q A_2B_1 (A_2 - A_2 B_2)^{n_2},
$$
by the inductive hypothesis, to conclude the proof of~\eqref{paso necesario2''''}, we must check that
$$
2\lambda q A_2^{n_2}B_1 = 2\lambda q B_1 (A_2 - A_2 B_2)^{n_2}.
$$
But this holds by Remark~\ref{tres20}. We now return to the proof of~\eqref{paso necesario2'''}. Assume that it holds for some $n_1 \ge 0$.~By~equal\-ity~\eqref{exponencial de (2t,2s),(0,1)}, to prove the case $n_1+1$, we must show that
\begin{equation*}
A_1^{n_1+1}A_2^{n_2}\left(\left(2t+2\lambda q a'+ 4\lambda q^2\right)B_1+2qB_2\right) = (2t B_1 + 2qB_2) (A_1 - 2tA_1 B_1 - 2qA_1 B_2)^{n_1+1} (A_2 - A_2 B_2)^{n_2}.
\end{equation*}
Since,
\begin{equation*}
A_1^{n_1+1}A_2^{n_2}\left(\left(2t+2\lambda q a'+ 4\lambda q^2\right)B_1+2qB_2\right) = A_1^{n_1+1}A_2^{n_2}\left(\left(2t + 2\lambda q a' \right)B_1+ 2qB_2\right) + 4\lambda q^2 A_1^{n_1+1}A_2^{n_2} B_1
\end{equation*}
and, by~\eqref{cuartas condiciones A_1, A_2, B1 y B2 cuarto caso} and the third condition in~\eqref{terceras condiciones A_1, A_2, B1 y B2 cuarto caso},
$$
(2t B_1 + 2q B_2)(A_1 - 2tA_1 B_1 - 2qA_1 B_2) = A_1(2tB_1+2qB_2) + 4\lambda q^2 A_1B_1,
$$
by the inductive hypothesis, to conclude the proof of~\eqref{paso necesario2'''}, it suffices to verify that
\begin{equation*}
4\lambda q^2 A_1^{n_1}A_2^{n_2}B_1 = 4\lambda q^2 B_1(A_1 - 2tA_1 B_1 - 2qA_1 B_2)^{n_1} (A_2 - A_2 B_2)^{n_2}.
\end{equation*}
We prove this by induction on $n_1$. For $n_1=0$, the statement follows from the second identity in~\eqref{terceras condiciones A_1, A_2, B1 y B2 cuarto caso} together with Remark~\ref{tres20}. Assume it holds for some $n_1 \ge 0$. Then, by the inductive hypothesis and the last condition in~\eqref{terceras condiciones A_1, A_2, B1 y B2 cuarto caso},
\begin{align*}
4\lambda q^2 A_1^{n_1+1}A_2^{n_2}B_1 & = 4\lambda q^2 A_1B_1(A_1 - 2tA_1 B_1 - 2qA_1 B_2)^{n_1} (A_2 - A_2 B_2)^{n_2}\\
& = 4\lambda q^2 B_1(A_1 - 2tA_1 B_1 - 2qA_1 B_2)^{n_1+1} (A_2 - A_2 B_2)^{n_2},
\end{align*}
as desired.

\smallskip

\noindent $\Leftarrow)$\enspace Since $(2t,2q)^{\times d_1/2} = (0,1)^{\times 2d_2} = (0,0)$ and $(\mathrm{U}_{d_1d_2},\times)$ is an abelian group, by~\eqref{eqq 2}, necessarily
\begin{equation*}
A_1^{d_1/2}=A_2^{2d_2}= \Ide\quad\text{and}\quad A_1A_2=A_2A_1.
\end{equation*}
Furthermore, by the bilinearity of $\Yleft$ and the definitions of $B_1$ and $B_2$, we obtain
\begin{equation*}
d_1B_1 y= d_1 (y\Yleft (1,0)) = y\Yleft (d_1,0)=  0 \quad\text{and}\quad d_2B_2 y= d_2 (y\Yleft (0,1)) =  y\Yleft (0,d_2)= 0\quad\text{for all $y\in I$.}
\end{equation*}	
Thus, the identities in~\eqref{primeras condiciones A_1, A_2, B1 y B2 cuarto caso} hold. Moreover, the identities in~\eqref{segundas condiciones A_1, A_2, B1 y B2 cuarto caso} follow directly from~\eqref{eqq 1}, since by definition,
$$
y\triangleleft (2t,2q) = (A_1 - 2tA_1B_1-2qA_1B_2)y\quad\text{and}\quad y\triangleleft (0,1) = (A_2 -A_2B_2)y.
$$
It remains to check~\eqref{terceras condiciones A_1, A_2, B1 y B2 cuarto caso} and~\eqref{cuartas condiciones A_1, A_2, B1 y B2 cuarto caso}. We are going to derive them from~\eqref{equivalencia entre las dos definiciones4}, which holds since
\begin{align*}
A_1^{n_1(a,a')}A_2^{n_2(a,a')}(y - aB_1 y -a'B_2 y) & = (a,a') \blackdiamond (y - aB_1 y -a'B_2 y) &&\text{by~\eqref{eqq 2}}\\
& = (a,a') \blackdiamond (y - y\Yleft (a,a')) &&\text{since $\Yleft$ is bilinear}\\
& = y \triangleleft (2t,2q)^{\times n_1(a,a')}\times (0,1)^{\times n_2(a,a')} && \text{by the def. of $\triangleleft$}\\
& = (A_1 - 2tA_1B_1-2qA_1B_2)^{n_1(a,a')}(A_2-A_2B_2)^{n_2(a,a')} y. &&\text{by~\eqref{eqq 1}}
\end{align*}
Since $A_2$ is invertible, by condition~\eqref{equivalencia entre las dos definiciones4} with $(a,a')=(0,1)^{\times 2}=(\lambda,2)$, we obtain
\begin{equation}\label{che}
A_2(\Ide - \lambda B_1 - 2B_2) = (\Ide-B_2)(A_2-A_2B_2) = A_2-A_2B_2 - B_2A_2 + B_2A_2B_2,
\end{equation}
which establishes the first equality in~\eqref{terceras condiciones A_1, A_2, B1 y B2 cuarto caso}. Applying condition~\eqref{equivalencia entre las dos definiciones3}
with $(a,a')=(0,1)^{\times 3}=(3\lambda,3)$, we obtain
$$
A_2^3(\Ide - 3\lambda B_1 - 3B_2) = (A_2-A_2B_2)^3 = A_2^2(\Ide - \lambda B_1 - 2B_2)(A_2-A_2B_2),
$$
where the last equality follows from $A_2^2(\Ide - \lambda B_1 - 2B_2) = (A_2-A_2B_2)^2$. Consequently,
$$
\lambda B_1A_2- 3\lambda A_2B_1 +2B_2A_2 - 2A_2B_2 = \lambda B_1A_2B_2 +2B_2A_2B_2.
$$
Combining this with the first equality in~\eqref{terceras condiciones A_1, A_2, B1 y B2 cuarto caso}, we obtain
\begin{equation}\label{con alfa''}
\lambda B_1A_2-\lambda A_2B_1 = \lambda B_1A_2B_2.
\end{equation}
Moreover, by condition~\eqref{equivalencia entre las dos definiciones4} with $(a,a')=(2t,2q)\times (0,1)=(2t+2\lambda q,2q+1)$, we have
\begin{equation}\label{para usar en un rato}
A_1A_2(\Ide-(2t+2\lambda q)B_1-(2q+1)B_2) = (A_1-2tA_1B_1-2qA_1B_2)(A_2-A_2B_2),
\end{equation}
which implies
$$
B_2A_2-A_2B_2-\lambda A_2B_1=B_2A_2B_2
$$
since $A_1$ is invertible. Combining this with the first equality in~\eqref{terceras condiciones A_1, A_2, B1 y B2 cuarto caso}, we obtain
\begin{equation}\label{sin alfa''}
2tB_1A_2-2tA_2B_1 = 2tB_1A_2B_2.
\end{equation}
Since $d_1B_1 = 0$ and, by Proposition~\ref{caso v2(d1)=v2(d2) y alfa impar}, we have $\lambda q(2q - 1) - 2t \equiv 1 \pmod{d_1}$, it follows from~\eqref{con alfa''} and~\eqref{sin alfa''} that the second equality in~\eqref{terceras condiciones A_1, A_2, B1 y B2 cuarto caso} holds. Combining~\eqref{para usar en un rato} with the second equalities in~\eqref{primeras condiciones A_1, A_2, B1 y B2 cuarto caso} and~\eqref{segundas condiciones A_1, A_2, B1 y B2 cuarto caso}, we obtain that
\begin{equation*}
A_2A_1(\Ide-(2t+2\lambda q)B_1-(2q+1)B_2) = (A_2-A_2B_2)(A_1-2tA_1B_1-2qA_1B_2),
\end{equation*}
which implies~\eqref{cuartas condiciones A_1, A_2, B1 y B2 cuarto caso}, since $A_2$ is invertible. It remains to verify the last equality in~\eqref{terceras condiciones A_1, A_2, B1 y B2 cuarto caso}. Since $A_1$ is invertible, by condition~\eqref{equivalencia entre las dos definiciones4} with $(a,a')=(2t,2q)^{\times 2}=(4t+4\lambda q^2,4q)$, we obtain
$$
A_1(\Ide-(4t+4\lambda q^2)B_1-4qB_2) = (\Ide-2tB_1-2qB_2)(A_1-2tA_1B_1-2qA_1B_2),
$$
which implies that
$$
2t[B_1,A_1]+2q[B_2,A_1]-4\lambda q^2A_1B_1 = 4t^2B_1A_1B_1+4tqB_1A_1B_2+4qtB_2A_1B_1+4q^2B_2A_1B_2.
$$
Using this together with the equality in~\eqref{cuartas condiciones A_1, A_2, B1 y B2 cuarto caso}, we obtain
\begin{equation}\label{condicion 6 caso 4'}
2t[B_1,A_1] = 4t^2B_1A_1B_1+4tqB_1A_1B_2.
\end{equation}
Condition~\eqref{equivalencia entre las dos definiciones4} for $(a,a') = (2t,2q)\times(0,1)^{\times 2} = (2t+\lambda(4q+1),\,2q+2)$, combined with the second equalities in~\eqref{primeras condiciones A_1, A_2, B1 y B2 cuarto caso} and~\eqref{segundas condiciones A_1, A_2, B1 y B2 cuarto caso}, yields
\begin{equation*}
A_2^2A_1\bigl(\Ide - (2t+\lambda(4q+1))B_1-(2q+2)B_2\bigr) = (A_2-A_2B_2)^2(A_1-2tA_1B_1-2qA_1B_2),
\end{equation*}
which, together with equality~\eqref{che}, gives
\begin{equation*}
A_2^2A_1\bigl(\Ide - (2t+\lambda(4q+1))B_1-(2q+2)B_2\bigr) = A_2^2\bigl(\Ide - \lambda B_1 - 2B_2\bigr)(A_1-2tA_1B_1-2qA_1B_2).
\end{equation*}
Set $R\coloneqq A_1-2t A_1 B_1-2qA_1B_2$. Since $A_2$ is invertible, we obtain that
$$
R- \lambda (4q+1)A_1B_1-2A_1B_2 = R-\lambda B_1A_1+2t\lambda B_1A_1 B_1+2q\lambda B_1A_1B_2-2B_2A_1+4tB_2A_1B_1+4qB_2A_2B_2.
$$
Hence
$$
\lambda[B_1,A_1]+2[B_2,A_1] = 2t\lambda B_1A_1 B_1+2q\lambda B_1A_1B_2+4\lambda qA_1B_1+4tB_2A_1B_1+4qB_2A_2B_2,
$$
which combined with~\eqref{cuartas condiciones A_1, A_2, B1 y B2 cuarto caso}, yields
$$
\lambda[B_1,A_1] = 2t\lambda B_1A_1 B_1+2q\lambda B_1A_1B_2.
$$
Since $d_1B_1 = 0$ and, by Proposition~\ref{caso v2(d1)=v2(d2) y alfa impar}, we have $\lambda q(2q - 1) - 2t \equiv 1 \pmod{d_1}$, from this and~\eqref{condicion 6 caso 4'}, we get the last equality in~\eqref{terceras condiciones A_1, A_2, B1 y B2 cuarto caso}, as desired.
\end{proof}

\begin{proposition} Let $H_{\lambda}^{d_1d_2}$ be as defined at the beginning of this section, and let $I$ be an abelian group. The~following facts hold:

\begin{enumerate}[label=\arabic*)]

\item Assume that $d_2$ is odd, or that $v_2(d_1)<v_2(d_2)$, or that $v_2(d_1)=v_2(d_2)\geq 1$ with $\lambda$ even. Let $A_1$, $A_2$, $B_1$ and $B_2$ be commuting endomorphisms of $I$. Then the conditions~\eqref{Propiedades de A1 y A2, B1, B2 primer caso}--\eqref{condiciones para [1.14] primer caso} in Proposition~\ref{caracterizacion primer caso} are satisfied if and only if
    \begin{equation*}
    A_1^{d_1} = A_2^{d_2}=\Ide,\quad d_1B_1=d_2B_2=0,\quad B_1^2=B_1B_2=0\quad\text{and}\quad B_2^2=-\lambda B_1.
    \end{equation*}

\item Assume that $\lambda=1$, $d_1=2$ and $v_2(d_2)=1$. Let $A$, $B_1$ and $B_2$ be commuting endomorphisms of $I$. Then the conditions~\eqref{primeras condiciones A, B1 y B2 segundo caso} and~\eqref{segundas condiciones A, B1 y B2 segundo caso} in Proposition~\ref{caracterizacion segundo caso} are satisfied if and only if
    \begin{equation*}
    A^{2d_2} =\Ide,\quad 2B_1=d_2B_2=0,\quad B_1B_2=0\quad\text{and}\quad B_1 = - B_2^2.
    \end{equation*}

\item Assume that $\lambda$ is odd, $v_2(d_1)=v_2(d_2)=1$ and $d_1>2$. Let $A_1$, $A_2$, $B_1$ and $B_2$ be commuting~endomor\-phisms of $I$. Then the conditions~\eqref{primeras condiciones A_1, A_2, B1 y B2 tercer caso}--\eqref{terceras condiciones A_1, A_2, B1 y B2 tercer caso} in Proposition~\ref{caracterizacion tercer caso} are satisfied if and only if
    \begin{equation*}
    A_1^{d_1/2} = A_2^{2d_2} =\Ide,\quad d_1B_1=d_2B_2=0,\quad B_1^2=B_1B_2=0\quad\text{and}\quad B_2^2 = -\lambda B_1.
    \end{equation*}

\item Assume that $\lambda$ is odd and $v_2(d_1)=v_2(d_2)\geq 2$. Let $A_1$, $A_2$, $B_1$ and $B_2$ be commuting endomorphisms of $I$. Then the conditions~\eqref{primeras condiciones A_1, A_2, B1 y B2 cuarto caso}--\eqref{terceras condiciones A_1, A_2, B1 y B2 cuarto caso} in Proposition~\ref{caracterizacion cuarto caso} are satisfied if and only if
    \begin{equation}\label{a usar}
    A_1^{d_1/2} = A_2^{2d_2} =\Ide,\quad d_1B_1=d_2B_2=0,\quad B_1^2 = B_1B_2=0\quad\text{and}\quad B_2^2 = -\lambda B_1.
    \end{equation}

\end{enumerate}
\end{proposition}

\begin{proof} We prove item~(4) and we left the other ones, which are easier, to the reader. By the hypothesis, the conditions in~\eqref{primeras condiciones A_1, A_2, B1 y B2 cuarto caso}, become
$$
A_1^{d_1/2} = A_2^{2d_2} =\Ide\quad\text{and}\quad d_1B_1=d_2B_2=0,
$$
while the conditions in~\eqref{terceras condiciones A_1, A_2, B1 y B2 cuarto caso}, become
$$
2 tB_1^2 = B_1B_2=0\quad\text{and}\quad B_2^2 = -\lambda B_1.
$$
But then, $\lambda B_1^2 = - B_1B_2^2 = 0$. Since, by Pro\-position~\ref{caso v2(d1)=v2(d2) y alfa impar}, we have $\lambda q(2q - 1) - 2t\equiv 1 \pmod{d_1}$ and  $d_1B_1 = 0$, we obtain $B_1^2 = 0$. Thus conditions~\eqref{primeras condiciones A_1, A_2, B1 y B2 cuarto caso} and~\eqref{terceras condiciones A_1, A_2, B1 y B2 cuarto caso} imply~\eqref{a usar}. On the other hand, it is easy to see that~\eqref{a usar} implies~\eqref{primeras condiciones A_1, A_2, B1 y B2 cuarto caso}, \eqref{terceras condiciones A_1, A_2, B1 y B2 cuarto caso}, \eqref{cuartas condiciones A_1, A_2, B1 y B2 cuarto caso} and the last equality in~\eqref{segundas condiciones A_1, A_2, B1 y B2 cuarto caso}. So, in order to finish the proof it suffices to check that~\eqref{a usar} implies
\begin{equation*}
(A_1-2tA_1B_1-2qA_1B_2)^{d_1/2} = (A_2-A_2B_2)^{2d_2} = \Ide.
\end{equation*}
For this note first that $B_2^3 = -\lambda B_1B_2 = 0$, which combined with~\eqref{a usar}, imply
$$
(2tB_1+2qB_2)^2 = 4t^2B_1^2 + 8qtB_1B_2 + 4q^2B^2_2 = 4q^2B^2_2\quad\text{and}\quad (2tB_1+2qB_2)^3 = (2tB_1+2qB_2)4q^2B^2_2 = 0.
$$
Hence, we have
\begin{equation*}
(A_1-2tA_1B_1-2qA_1B_2)^{d_1/2} = A_1^{d_1/2}(\Ide-2tB_1-2qB_2)^{d_1/2} = \Ide - \frac{d_1}{2}(2tB_1+2qB_2) + \binom{d_1/2}{2} 4q^2B^2_2 = \Ide,
\end{equation*}
where the last equality holds, since
$$
d_1B_1 = 0,\quad d_1qB_2=d_2B_2=0\quad\text{and}\quad \frac{\frac{d_1}{2}\bigl(\frac{d_1}{2}-1\bigr)}{2}4q^2B_2 = (\frac{d_1}{2}-1\bigr)qd_2B_2 = 0.
$$
Final\-ly,
\begin{equation*}
(A_2-A_2B_2)^{2d_2}=A_2^{2d_2} (\Ide-B_2)^{2d_2} = \Ide - 2d_2B_2 + \binom{2d_2}{2} B_2^2 = \Ide,
\end{equation*}
because $d_2B_2=0$ and $\binom{2d_2}{2} B_2^2 = d_2(2d_2-1)B_2^2=0$.
\end{proof}

\section[Computation of the cohomology]{Computation of the cohomology}\label{seccion 4}

Let $H^{d_1d_2}_{\lambda}$ be as in Section~\ref{section 2} and let $I$ be an abelian group. Let $n$ and $s$ be as at beginning of \cite{GGV2}*{Subsection~2.2}. Hence, in our case, $n=2$, and by Re\-mark~\ref{caso d2 impar}, we know that $s\in\{1,2\}$. Recall that in \cite{GGV2}*{Remark~2.6 and~2.7}, for each $\gamma_1,\gamma_2\in I$ a map $\alpha_{\gamma_1,\gamma_2}\colon H^{d_1d_2}_{\lambda}\times H^{d_1d_2}_{\lambda}\to I$ is defined by
\begin{equation}\label{alpha_{gamma_1,gamma_2}}
\alpha_{\gamma_1,\gamma_2}((h_1,h_2),(l_1,l_2))\coloneqq \alpha_1(\gamma_1)(h_1,l_1) + \alpha_2(\gamma_2)(h_2,l_2),
\end{equation}
where
$$
\alpha_i(\gamma)(h,l)\coloneqq \begin{cases} \gamma & \text{if $0\le h,l<d_i$ and $h+l\ge d_i$,}\\ 0 & \text{if $0\le h,l<d_i$ and $h+l<d_i$.}\end{cases}
$$
Let $\blackdiamond\colon H^{d_1d_2}_{\lambda}\times I\to I$ and $\Yleft \colon I\times H^{d_1d_2}_{\lambda}\to I$ two maps such that Conditions~\eqref{eqq 1}--\eqref{eqq 3} are satisfied. Let
$$
T_1,T_2\colon I^2\oplus I^{2s}\to I^{2s},\quad T_3\colon I^2\oplus I^{2s}\to I^{2\binom{s}{2}}\quad\text{and}\quad S\colon I^2\to I^2\oplus I^{2s},
$$
be as in~\cite{GGV2}*{Subsubsection~2.2.3} (note that if $s=1$, then $\binom{s}{2}=0$ and $T_3$ is the zero map). Recall from the discussion above  ~\cite{GGV2}*{Corollary~2.35} that each family $\bigl((\gamma_i),(\mathfrak{f}_{ij})\bigr)\in \ker(T_1)\cap \ker(T_2)\cap \ker(T_3)$ determines a map
$$
f_{(\mathfrak{f}_{ij}),(\gamma_i)} \colon H^{d_1d_2}_{\lambda}\times H^{d_1d_2}_{\lambda}\to I
$$
This map satisfies the equalities $f_{(\mathfrak{f}_{ij}),(\gamma_i)}(a_i,e_j) = \mathfrak{f}_{ij}$ and can be computed using~\cite{GGV2}*{Formula~(2.32) and Proposi\-tions~2.12 and~2.37}. Moreover $\bigl(\alpha_{\gamma_1,\gamma_2},-f_{(\mathfrak{f}_{ij}),(\gamma_i)}\bigr)$ is a $2$-cocycle of $\bigl(\wh{C}_N^*(H^{d_1d_2}_{\lambda},I),\partial+D\bigr)$ and the map
\begin{equation*}
T\colon \ker(T_1)\cap \ker(T_2)\cap \ker(T_3)\longrightarrow \wh{C}^{02}_N(H^{d_1d_2}_{\lambda},I)\oplus \wh{C}^{11}_N(H^{d_1d_2}_{\lambda},I),
\end{equation*}
given by
\begin{equation}\label{formula0}
T((\gamma_i),(\mathfrak{f}_{ij})) \coloneqq \bigl(\alpha_{\gamma_1,\gamma_2},-f_{(\mathfrak{f}_{ij}),(\gamma_i)}\bigr),
\end{equation}
induces an isomorphim
\begin{equation}\label{formula00}
\ov{T}\colon \frac{\ker(T_1)\cap \ker(T_2)\cap \ker(T_3)}{\ima(S)}\longrightarrow \Ho^2_{\blackdiamond,\Yleft}(H^{d_1d_2}_{\lambda},I).
\end{equation}
Note that the domain of $\ov{T}$ is a subquotient of $I^6$. If we can show that it is isomorphic to a direct sum of subquotients of $I$, then the co\-homology classes, and thus the extensions, can be parameterized by the elements of this de\-com\-position, as is done for example in Section~4 of~\cite{GGV3}.

Let $H$ be a finite linear cycle set and let $I$ be an abelian group. Assume that, as an additive group,
$$
H=\mathds{Z}_{d_1}\oplus \dots \oplus \mathds{Z}_{d_n}.
$$
For each $i$, let $e_i \in H$ denote the element whose only non-zero coordinate is the $i$-th one, equal to $1$. Suppose further that $H$ is abelian under the operation $\times$. Then
\[
(H,\times) = \langle a_1 \rangle \oplus \cdots \oplus \langle a_s \rangle.
\]
Define $Y\coloneqq \{a_1,\dots,a_s\}$ and $Z\coloneqq \{e_1,\dots,e_n\}$.  Let $AS_Z$ be the free abelian semigroup genetared by $Z$. For each family $(\mathfrak{f}_{ij})_{1\le i\le s,1\le j\le s,}$, of elements of $I$, Corollary~2.19 in \cite{GGV2}, provides a map
$$
f_1\colon Y\times AS_Z\to I
$$
such that $f_1(a_i,e_j) = \mathfrak{f}_{ij}$. Throughout this paper we denote this map by $f_1^{\mathfrak{f}_{11},\mathfrak{f}_{12},\dots,\mathfrak{f}_{sn}}$.

\subsection[Computation of the cohomology when \texorpdfstring{$d_1\mid \lambda\binom{d_2}{2}$}{d1|lambda d2(d2-1)/2}]{Computation of the cohomology when \texorpdfstring{$\pmb{d_1\mid \lambda\binom{d_2}{2}}$}{d1|lambda d2(d2-1)/2}}\label{Seccion 4}
By Remark~\ref{lambda divide a d1}, we can suppose that $\lambda=0$ or~$\lambda \mid d_1$, and we do it. In this section we assume that $d_1\mid \lambda\binom{d_2}{2}$ and we com\-pute~$\ho_{\blackdiamond,\Yleft}^2(H^{d_1d_2}_{\lambda},I)$, under suitable hypotheses. By Re\-mark~\ref{caso d2 impar}, we know that $s=n=2$, $r_1=d_1$, $r_2=d_2$, and we can take $a_1=e_1\coloneqq (1,0)$ and $a_2=e_2\coloneqq (0,1)$. Therefore, by~\eqref{def de punto},
\begin{equation}\label{a_j cdot e_i}
a_i\cdot e_j = \begin{cases} e_j & \text{if $(i,j)\ne (2,2)$,}\\ (-\lambda,1) & \text{if $(i,j)=(2,2)$.}\end{cases}
\end{equation}
Hence, if we write $a_i=\lambda_{1i}e_1 + \lambda_{2i}e_2$ and  $a_i\cdot e_j=\Gamma^1_{ji}e_1 + \Gamma^2_{ji}e_2$ with $0\le \lambda_{ji}< d_i$ and $0\le \Gamma^k_{ji}< d_k$, then
$$
\lambda_{11} = \lambda_{22} = 1,\quad \lambda_{12} = \lambda_{21} = 0\quad\text{and}\quad \Gamma_{ji}^k = \begin{cases} d_1-\lambda & \text{if $(i,j,k) = (2,2,1)$ and $\lambda\ne 0$,}\\ \delta_{jk} & \text{otherwise.}\end{cases}
$$
A straightforward computation, using~\cite{GGV2}*{Subsubsection~2.2.3} and these facts, shows that
$$
T_1,T_2\colon I^6\to I^4,\quad T_3\colon I^6\to I^2\quad\text{and}\quad S\colon I^2\to I^6,
$$
are given by
\begin{align}
& T_1(\gamma_1,\gamma_2, \mathfrak{f}_{11}, \mathfrak{f}_{12}, \mathfrak{f}_{21}, \mathfrak{f}_{22})= \left(d_1 \mathfrak{f}_{11} + \mathcal{U}_{11}, d_2 \mathfrak{f}_{12} + \mathcal{U}_{12}, d_1 \mathfrak{f}_{21} + \mathcal{U}_{21}, d_2 \mathfrak{f}_{22} + \mathcal{U}_{22}\right),\label{T_1}\\
& T_2(\gamma_1,\gamma_2, \mathfrak{f}_{11}, \mathfrak{f}_{12}, \mathfrak{f}_{21}, \mathfrak{f}_{22})= \left(\mathcal{V}_{11} + \mathcal{W}_{11}, \mathcal{V}_{12} + \mathcal{W}_{12}, \mathcal{V}_{21} + \mathcal{W}_{21}, \mathcal{V}_{22} + \mathcal{W}_{22}\right),\label{T_2}\\
& T_3(\gamma_1,\gamma_2, \mathfrak{f}_{11}, \mathfrak{f}_{12}, \mathfrak{f}_{21}, \mathfrak{f}_{22})= (F_1(a_2,a_1,e_1)-F_1(a_1,a_2,e_1),F_1(a_2,a_1,e_2)- F_1(a_1,a_2,e_2))\label{T_3}
\shortintertext{and}
& S(t_1,t_2)=\left(d_1t_1,d_2t_2,\mathcal{X}_{11},\mathcal{X}_{12},\mathcal{X}_{21},\mathcal{X}_{22} + \mathcal{Y}_{22}\right)\label{S},
\end{align}
in which
\begin{itemize}
	
\item[-] $\mathcal{U}_{ij}\coloneqq \sum_{k=1}^{d_j-1}\alpha_{\gamma_1,\gamma_2}(a_i\cdot e_j,k (a_i\cdot e_j))-a_i\blackdiamond \gamma_j$,
	
\item[-] $\mathcal{V}_{ij}\coloneqq \sum_{\ell=0}^{d_i-1} a_i^{\times (d_i-1-\ell)}\blackdiamond f_1(a_i,a_i^{\times \ell}\cdot e_j)$,
	
\item[-] $\mathcal{W}_{ij}\coloneqq \sum_{\ell=1}^{d_i-1} a_i^{\times \ell}\blackdiamond f_1(a_i,\bar{a}_{i\ell})\Yleft e_j -\sum_{\ell=1}^{d_i-1} a_i^{\times (\ell+1)}\blackdiamond \alpha_{\gamma_1,\gamma_2}(a_i,\bar{a}_{i\ell})\Yleft e_j$, where $\bar{a}_{i\ell}\coloneqq a_i^{\times(\ell+1)}-a_i$,
	
\item[-] $\mathcal{X}_{ij}\coloneqq -t_j + a_i\blackdiamond t_j + e_j\blackdiamond t_i \Yleft e_j$,
	
\item[-] $\mathcal{Y}_{22}\coloneqq 0$ if $\lambda=0$ and $\mathcal{Y}_{22}\coloneqq (\lambda - d_1)t_1 + (d_1-\lambda)(e_2\blackdiamond t_2 \Yleft e_1)$ if $\lambda\ne 0$,

\item[-] $f_1\coloneqq f_1^{\mathfrak{f}_{11},\mathfrak{f}_{12},\mathfrak{f}_{21},\mathfrak{f}_{22}}$ and $F_1$ is as in~\cite{GGV2}*{Definition~2.25}.

\end{itemize}

\subsubsection[The case \texorpdfstring{$\lambda\mid d_1$}{lambda|d1} and \texorpdfstring{$\Yleft=0$}{-<=0}]{The case \texorpdfstring{$\pmb{\lambda\mid d_1}$}{lambda|d1} and \texorpdfstring{$\pmb{\Yleft=0}$}{-<=0}}
By Proposition~\ref{caracterizacion primer caso}, the actions $\blackdiamond$ and $\Yleft$ are determined by endomorphisms $A_1$, $A_2$, $B_1$ and $B_2$ of $I$ satisfying~condi\-tions~\eqref{Propiedades de A1 y A2, B1, B2 primer caso}--~\eqref{condiciones para [1.14] primer caso}. Next, we set $C_i\coloneqq \Ide - A_i$ and we assume that $B_1=B_2=0$. Under this hypothesis,~\eqref{Propiedades de A1 y A2, B1, B2 primer caso}--~\eqref{condiciones para [1.14] primer caso} reduce to
$$
A_1^{d_1}=A_2^{d_2}=\Ide\quad \text{and} \quad [A_1,A_2]=0.
$$

\begin{proposition}\label{caso lambda no nulo  pero Yleft cero} Let $x\coloneqq (\gamma_1,\gamma_2,\mathfrak{f}_{11},\mathfrak{f}_{12},\mathfrak{f}_{21}, \mathfrak{f}_{22})$ and let $T_1$, $T_2$, $T_3$ and $S$ be as in~\eqref{T_1}--\eqref{S}. The following facts hold:
\begin{align*}
& T_1(x)=\Bigl(d_1 \mathfrak{f}_{11} + C_1\gamma_1,d_2 \mathfrak{f}_{12} + C_1\gamma_2, d_1 \mathfrak{f}_{21} + C_2\gamma_1, d_2 \mathfrak{f}_{22}+(d-1) \frac{d_2}{d}\gamma_1 + C_2\gamma_2\Bigr), \\
& T_2(x)=\Biggl(\sum_{\ell=0}^{d_1-1} A_1^{\ell}\mathfrak{f}_{11},\sum_{\ell=0}^{d_1-1} A_1^{\ell}\mathfrak{f}_{12},\sum_{\ell=0}^{d_2-1} A_2^{\ell}\mathfrak{f}_{21}, \sum_{\ell=0}^{d_2-1} A_2^{\ell} \mathfrak{f}_{22} + \sum_{\ell=0}^{d_2-1} \mathrm{r}_{d_1}(\lambda \ell) A_2^{\ell-1} \mathfrak{f}_{21}+ \sum_{\substack{\ell=0\\ d \nmid \ell}}^{d_2-1} A_2^{d_2-1-\ell} \gamma_1 \Biggr), \\
&  T_3(x)= (C_2\mathfrak{f}_{11} - C_1 \mathfrak{f}_{21}, (d_1-\lambda) \mathfrak{f}_{11} + C_2\mathfrak{f}_{12} - C_1 \mathfrak{f}_{22})
\shortintertext{and}
& S(t_1,t_2)=(d_1 t_1, d_2 t_2,-C_1 t_1,- C_1t_2, -C_2 t_1,-(d_1-\lambda) t_1- C_2 t_2),
\end{align*}
where $d\coloneqq d_1/\lambda$.
\end{proposition}

\begin{proof} We first compute $T_1$. By~\eqref{a_j cdot e_i} and the very definition of $\alpha_{\gamma_1,\gamma_2}$ (see \cite{GGV2}*{Remarks~2.6 and~2.7}), we have
$$
\mathcal{U}_{11}  = C_1\gamma_1,\quad \mathcal{U}_{12}  = C_1\gamma_2  \quad\text{and}\quad \mathcal{U}_{21}  = C_2\gamma_1.
$$
Moreover, using again~\eqref{a_j cdot e_i} and the definition of $\alpha_{\gamma_1,\gamma_2}$, we obtain that
\begin{align*}
\sum_{k=1}^{d_2-1}\alpha_{\gamma_1,\gamma_2}(a_2\cdot e_2,k (a_2\cdot e_2)) &=\sum_{k=1}^{d_2-1}\alpha_{\gamma_1,\gamma_2}((-\lambda,1),(-k \lambda,k))\\
&=\sum_{k=1}^{d_2-1}\alpha_1(\gamma_1)(-\lambda,-k \lambda)+\sum_{k=1}^{d_2-1}\alpha_2(\gamma_2)(1,k).
\end{align*}
We claim that
\begin{equation}\label{suma de alfas}
\sum_{k=1}^{d_2-1}\alpha_2(\gamma_2)(1,k)=\gamma_2\quad\text{and}\quad \sum_{k=1}^{d_2-1}\alpha_1(\gamma_1)(-\lambda,-k \lambda)=(d-1)\frac{d_2}{d}\gamma_1.
\end{equation}
The first assertion is clear. To prove the second one, we note that $-k\lambda$ takes the values $0,\lambda,\dots, \bigl(\frac{d_1}{\lambda}-1\bigr)\lambda$ modulo~$d_1$, and so
\begin{equation}\label{formula4}
\alpha_1(\gamma_1)(-\lambda,-k \lambda)=\alpha_1(\gamma_1)(d_1-\lambda,-k\lambda)=\begin{cases} 0 & \text{if $d_1\mid k\lambda$,}\\ \gamma_1 & \text{otherwise}.\end{cases}
\end{equation}
Using~\eqref{suma de alfas} and the definition of $\mathcal{U}_{22}$, we arrive immediately at the desired formula for $T_1(x)$. Next, we compute $T_2$. Since $\mathcal{W}_{ij}$ for all $i,j$, for this it suffices to calculate
$$
\mathcal{V}_{ij} = \sum_{\ell=0}^{d_i-1} a_i^{\times (d_i-1-\ell)}\blackdiamond f_1(a_i,a_i^{\times \ell}\cdot e_j)\quad \text{for all $i,j$.}
$$
By~\eqref{potencia}, we have
\begin{equation}\label{formula1}
a_1^{\times \ell} = \ell a_1\qquad\text{and}\qquad a_2^{\times \ell} = \lambda \binom{\ell}{2} a_1 + \ell a_2.
\end{equation}
Hence, by~\eqref{def de punto},
\begin{equation}\label{formula2}
a_i^{\times \ell}\cdot e_j = \begin{cases} e_j & \text{if $(i,j)\ne (2,2)$,} \\ (-\lambda\ell,1) & \text{if $(i,j)=(2,2)$.}\end{cases}
\end{equation}
Employing this fact and, noting that $f_1(a_i,e_j) = \mathfrak{f}_{ij}$ by definition, we deduce
\begin{align}
\sum_{\ell=0}^{d_1-1} a_1^{\times (d_1-1-\ell)}\blackdiamond f_1(a_1,a_1^{\times \ell}\cdot e_1) & = \sum_{\ell=0}^{d_1-1} A_1^{\ell}\mathfrak{f}_{11},\label{primera para T2} \\
\sum_{\ell=0}^{d_1-1} a_1^{\times (d_1-1-\ell)}\blackdiamond f_1(a_1,a_1^{\times \ell}\cdot e_2) & = \sum_{\ell=0}^{d_1-1} A_1^{\ell}\mathfrak{f}_{12},\label{segunda para T2} \\
\sum_{\ell=0}^{d_2-1} a_2^{\times (d_2-1-\ell)}\blackdiamond f_1(a_2,a_2^{\times \ell}\cdot e_1) & = \sum_{\ell=0}^{d_2-1} A_2^{\ell}\mathfrak{f}_{21}.\label{tercera para T2}
\end{align}
Moreover, by \cite{GGV2}*{Identities~(2.17) and~(2.18)} and identity~\eqref{formula2}, we have
$$
f_1(a_2,a_2^{\times \ell}\cdot e_2)=f_1(a_2,-\lambda \ell e_1+e_2)= r\mathfrak{f}_{21}+ \mathfrak{f}_{22} + \sum_{k=1}^{r-1} \alpha_{\gamma_1,\gamma_2}(e_1,ke_1) + \alpha_{\gamma_1,\gamma_2}(-\lambda\ell e_1,-\lambda e_1+e_2),
$$
where we write $r\coloneqq \mathrm{r}_{d_1}(-\lambda \ell)$. Consequently, since, by~\eqref{formula4},
$$
\alpha_{\gamma_1,\gamma_2}(e_1,ke_1) = 0\quad\text{and}\quad \alpha_{\gamma_1,\gamma_2}(-\lambda \ell e_1,-\lambda e_1+e_2) = \alpha_1(\gamma_1)(-\lambda \ell,-\lambda) = \begin{cases} 0 & \text{if $d \mid \ell$,}\\ \gamma_1  & \text{otherwise,}\end{cases}
$$
we obtain
\begin{equation}\label{ultima para T2}
\sum_{\ell=0}^{d_2-1} a_2^{\times (d_2-1-\ell)}\blackdiamond f_1(a_2,a_2^{\times \ell}\cdot e_2)  = \sum_{\ell=0}^{d_2-1} A_2^{\ell} \mathfrak{f}_{22}
+
\sum_{\ell=0}^{d_2-1} \mathrm{r}_{d_1}(-\lambda \ell) A_2^{d_2-1-\ell} \mathfrak{f}_{21} + \sum_{\substack{\ell=0\\ d \nmid \ell}}^{d_2-1} A_2^{d_2-1-\ell} \gamma_1.
\end{equation}
The formula for $T_2(x)$ follows immediately from~\eqref{primera para T2}--\eqref{ultima para T2}. Next, we compute the $kth$ coordinate $T_3(x)_k$ of $T_3(x)$. By \cite{GGV2}*{Remark~2.36}, we have
$$
T_3(x)_k= f_1(a_1,a_2\cdot e_k) + a_1\blackdiamond f_1(a_2,e_k)- f_1(a_2,a_1\cdot e_k) - a_2\blackdiamond f_1(a_1,e_k).
$$
From this we obtain
$$
T_3(x)_k=\begin{cases} \mathfrak{f}_{11}+A_1 \mathfrak{f}_{21}-\mathfrak{f}_{21}-A_2 \mathfrak{f}_{11} & \text{when $k=1$}\\ f_1(a_1,a_2\cdot e_2) + A_1 \mathfrak{f}_{22}- \mathfrak{f}_{22} - A_2\mathfrak{f}_{12} & \text{when $k=2$.}\end{cases}
$$
Moreover, by \cite{GGV2}*{identities~(2.17) and~(2.18)}
$$
f_1(a_1,a_2\cdot e_2)=f_1(a_1,-\lambda e_1+e_2)=f_1(a_1,-\lambda e_1)+ \mathfrak{f}_{12}+\alpha_{\gamma_1,\gamma_2}(-\lambda e_1,e_2)
= (d_1-\lambda) \mathfrak{f}_{11}+\mathfrak{f}_{12}.
$$
The formula for $T_3(x)$ follows immediately from these facts. Finally, the formula for $S$ is straightforward.
\end{proof}

\begin{proposition}\label{lambda mid d_1, Yleft=0 and that A_1=A_2=Ide} Assume that $A_1=A_2=\Ide$. Then
$$
\ho_{\blackdiamond}^2(H^{d_1d_2}_{\lambda},I)\simeq I_{\lambda}\oplus I_{d_1}\oplus I_{d_1}\oplus \frac{I}{\lambda I}\oplus I_{\frac{d_2}{d}} \oplus \frac{I}{d_2 I},
$$
where $d\coloneqq d_1/\lambda$.
\end{proposition}

\begin{proof} Let $\sigma\colon I^6 \to I^6$ be the map given by
$$
\sigma(\mathfrak{f}_{11},\mathfrak{f}_{12},\mathfrak{f}_{21},\mathfrak{f}_{22},
\gamma_1,\gamma_2)=(\gamma_1,\gamma_2,\mathfrak{f}_{11},\mathfrak{f}_{12},\mathfrak{f}_{21},\mathfrak{f}_{22}).
$$
Set $\widetilde{T}_i\coloneqq T_i\xcirc \sigma$ and $\widetilde{S}\coloneqq \sigma^{-1} S$, where the $T_i$'s and $S$ are as in the previous proposition. By~\eqref{formula00}, we have
$$
\ho_{\blackdiamond}^2(H^{d_1d_2}_{\lambda},I) \simeq \frac{\ker T_1\cap \ker T_2\cap \ker T_3}{\ima S}\simeq \frac{\ker \wt{T}_1\cap \ker \wt{T}_2\cap \ker \wt{T}_3}{\ima \wt{S}},
$$
where the second isomorphism is induced by $\sigma^{-1}$. Using Proposition~\ref{caso lambda no nulo  pero Yleft cero} and taking into account that
$$
\sum_{\ell=0}^{d_2-1} A_2^{\ell-1} \mathrm{r}_{d_1}(\lambda \ell) =\frac{d_2}{d} \sum_{\ell=0}^{d} \mathrm{r}_{d_1}(\lambda \ell)
=\frac{d_2}{d}\binom{d}{2}\lambda=\frac{d_2(d-1)\lambda}{2},
$$
we obtain
\begin{align*}
& \wt{T}_1(\mathfrak{f}_{11},\mathfrak{f}_{12},\mathfrak{f}_{21},\mathfrak{f}_{22},\gamma_1,\gamma_2) = \left(d_1\mathfrak{f}_{11},d_2\mathfrak{f}_{12},d_1\mathfrak{f}_{21},d_2\mathfrak{f}_{22}+ (d-1)\frac{d_2}{d}\gamma_1\right), \\
& \wt{T}_2(\mathfrak{f}_{11},\mathfrak{f}_{12},\mathfrak{f}_{21},\mathfrak{f}_{22},\gamma_1,\gamma_2) = \left(d_1\mathfrak{f}_{11},d_1\mathfrak{f}_{12}, d_2\mathfrak{f}_{21},d_2\mathfrak{f}_{22}+\frac{d_2(d-1)\lambda}{2}\mathfrak{f}_{21}+(d-1)\frac{d_2}{d}\gamma_1\right), \\
& \wt{T}_3(\mathfrak{f}_{11},\mathfrak{f}_{12},\mathfrak{f}_{21},\mathfrak{f}_{22},\gamma_1,\gamma_2) = (0,(d_1-\lambda)\mathfrak{f}_{11})
\shortintertext{and}
& \wt{S}(t_1,t_2)  = (0,0,0,-(d_1-\lambda)t_1,d_1 t_1,d_2 t_2).
\end{align*}
Let $F_{45}\colon I^2\to I$ and $G_{45}\colon I\to I^2$ be the maps given by
$$
F_{45}(\mathfrak{f}_{22},\gamma_1)\coloneqq d_2 \mathfrak{f}_{22}+ (d-1)\frac{d_2}{d} \gamma_1\quad\text{and}\quad G_{45}(t)\coloneqq (-(d_1-\lambda)t,d_1 t).
$$	
Let $I_r\coloneqq \ker(r \Ide_I) = \{y\in I:ry=0\}$. A direct computation using that~$d_1\mid \frac{d_2}{d}\binom{d}{2}\lambda$, proves that
$$
\frac{\ker \wt{T}_1\cap \ker \wt{T}_2\cap \ker \wt{T}_3}{\ima \wt{S}} = I_{\lambda}\oplus I_{d_1}\oplus I_{d_1}\oplus\frac{\ker(F_{45})}{\ima(G_{45})} \oplus \frac{I}{d_2 I}.	
$$
We define the map $\Phi\colon I\oplus I\to I\oplus I$, by $\Phi(y_1,y_2)\coloneqq (y_1-y_2-d y_1, y_2+d y_1)$. Note that $\Phi$ is invertible, with inverse given by $\Phi^{-1}(z_1,z_2)=(z_1+z_2,z_2-d(z_1+z_2))$. We set $\wt{G}_{45}\coloneqq \Phi^{-1}\xcirc G_{45}$ and $\wt{F}_{45}\coloneqq F_{45}\xcirc \Phi$. Clearly $\Phi$ induces an isomorphism $\ov{\Phi}\colon \frac{\ker \wt{F}_{45}}{\ima \wt{G}_{45}}\to \frac{\ker F_{45}}{\ima G_{45}}$. A direct computation shows that
\begin{equation*}
\wt{G}_{45}(t) = (\lambda t,0) \quad \text{and}\quad \wt{F}_{45}(y_1,y_2) =  -\frac{d_2}{d} y_2,
\end{equation*}
and so we have $\frac{\ker(F_{45})}{\ima(G_{45})} \simeq \frac{I}{\lambda I}\oplus I_{\frac{d_2}{d}}$, which gives
$$
\ho_{\blackdiamond}^2(H^{d_1d_2}_{\lambda},I)\simeq I_{\lambda}\oplus I_{d_1}\oplus I_{d_1}\oplus \frac{I}{\lambda I}\oplus I_{\frac{d_2}{d}} \oplus \frac{I}{d_2 I},
$$
as desired.
\end{proof}

\begin{remark}\label{remark1} Assume that we are under the hypothesis of Proposition~\ref{lambda mid d_1, Yleft=0 and that A_1=A_2=Ide}. The map $\Psi\colon I^6\to I^6$, given by
$$
\Psi(y_1,y_2,y_3,y_4,y_5,y_6)\coloneqq (dy_4+y_5,y_6,y_1,y_2,y_3,y_4-dy_4-y_5),
$$
induces an isomorphism
$$
\ov{\Psi}\colon I_{\lambda}\oplus I_{d_1}\oplus I_{d_1}\oplus \frac{I}{\lambda I}\oplus I_{\frac{d_2}{d}} \oplus \frac{I}{d_2 I}\longrightarrow \frac{\ker T_1\cap \ker T_2\cap \ker T_3}{\ima S}.
$$
Let $Z\subseteq I_{\lambda}\oplus I_{d_1}\oplus I_{d_1}\oplus I\oplus I_{\frac{d_2}{d}} \oplus I$ be such that the natural map $Z\to I_{\lambda}\oplus I_{d_1}\oplus I_{d_1}\oplus \frac{I}{\lambda I}\oplus I_{\frac{d_2}{d}} \oplus \frac{I}{d_2 I}$ is a bijection and let~$T$ be as in~\eqref{formula0}, so that $$
T(\gamma_1,\gamma_2,\mathfrak{f}_{11},\mathfrak{f}_{12},\mathfrak{f}_{21},\mathfrak{f}_{22}) = (\alpha_{\gamma_1,\gamma_2}, f_{\mathfrak{f}_{11}, \mathfrak{f}_{12},\mathfrak{f}_{21},\mathfrak{f}_{22},\gamma_1,\gamma_2}).
$$
Clearly $T(\Psi(Z))$ is a complete family of representatives of~$2$-cocycles modulo coboundaries in $\wh{\mathcal{C}}_N^*(H^{d_1d_2}_{\lambda},I)$. Hence, computing $T$ becomes essential. Since $\alpha_{\gamma_1,\gamma_2}$ is already known from~\eqref{alpha_{gamma_1,gamma_2}}, the determination of $T$ reduces to establishing the following proposition.
\end{remark}

\begin{proposition}\label{calculo de f} Assume that we are under the hypothesis of Proposition~\ref{lambda mid d_1, Yleft=0 and that A_1=A_2=Ide}, and let $T_1$, $T_2$ and $T_3$ be as in~Pro\-position~\ref{caso lambda no nulo  pero Yleft cero}. For
$(\gamma_1,\gamma_2,\mathfrak{f}_{11}, \mathfrak{f}_{12},\mathfrak{f}_{21},\mathfrak{f}_{22})\in \ker T_1\cap \ker T_2\cap \ker T_3$,
let $f\coloneqq f_{\mathfrak{f}_{11},\mathfrak{f}_{12},\mathfrak{f}_{21},\mathfrak{f}_{22},\gamma_1, \gamma_2}$. We have
$$
f(a_1^{\times \ell_1}\times a_2^{\times \ell_2},(u,v)) = \ell_1 u \mathfrak{f}_{11}+ \ell_1 v\mathfrak{f}_{12} + \sum_{j=0}^{\ell_2-1} \mathrm{r}_{d_1}(u\!-\!\lambda jv) \mathfrak{f}_{21} + \ell_2 v \mathfrak{f}_{22} + \left(v-\left\lceil \frac{v}{d} \right\rceil\right)\ell_2 \gamma_1 + C^{\ell_2}_{uv} \gamma_1,
$$
where $0\le\ell_1,u<d_1$, $0\le\ell_2,v<d_2$ and $C^{\ell_2}_{uv} = \#\{0\le j < \ell_2: \mathrm{r}_{d_1}(u-\lambda jv)+\mathrm{r}_{d_1}(-\lambda_v)\ge d_1\}$.
\end{proposition}

\begin{proof} By~\eqref{def de punto} and~\eqref{formula1}, we obtain
\begin{equation}\label{formula3}
a_i^{\times j}\cdot (u,v) = \begin{cases} (u,v) & \text{if $i=1$,}\\ (u-\lambda jv,v) & \text{if $i=2$.}\end{cases}
\end{equation}
Hence, by~\cite{GGV2}*{Proposition~2.37},
$$
f(a_1^{\times \ell_1}\times a_2^{\times \ell_2},(u,v)) = f(a_1^{\times \ell_1},(u,v)) + f(a_2^{\times \ell_2},(u,v)).
$$
Consequently, by~\cite{GGV2}*{Proposition~2.12},
\begin{equation}\label{pepitito1}
f(a_1^{\times \ell_1}\times a_2^{\times \ell_2},(u,v)) = \ell_1 f(a_1,(u,v)) + \sum_{j=0}^{\ell_2-1} f(a_2,(u-\lambda jv,v)).
\end{equation}
Moreover, by~\eqref{formula3} and~\cite{GGV2}*{Identity~(2.32)},
\begin{equation*}
f(a_1,(u,v)) = u\mathfrak{f}_{11}+v\mathfrak{f}_{12}+\sum_{h=1}^{u-1} \alpha_{\gamma_1,\gamma_2}(e_1,he_1)+\sum_{h=1}^{v-1} \alpha_{\gamma_1,\gamma_2}(e_2,he_2) + \alpha_{\gamma_1,\gamma_2}(ue_1,ve_2) = u \mathfrak{f}_{11}+ v\mathfrak{f}_{12}
\end{equation*}
and
\begin{align*}
f(a_2,(u-\lambda jv,v)) & = \mathrm{r}_{d_1}(u-\lambda jv)\mathfrak{f}_{21}+v \mathfrak{f}_{22} + \sum_{h=1}^{\mathrm{r}_{d_1}(u-\lambda jv)-1} \alpha_{\gamma_1,\gamma_2}(e_1,he_1)\\
& + \sum_{h=1}^{v-1} \alpha_{\gamma_1,\gamma_2}(-\lambda e_1+e_2,-h\lambda e_1+he_2) + \alpha_{\gamma_1,\gamma_2}((u-\lambda jv)e_1,-\lambda v e_1+ve_2)\\
& = \mathrm{r}_{d_1}(u-\lambda jv)\mathfrak{f}_{21}+v\mathfrak{f}_{22} + \sum_{h=1}^{v-1} \alpha_1(\gamma_1)(-\lambda,-h\lambda) + \alpha_1(\gamma_1)(u-\lambda jv,-\lambda v).
\end{align*}
By combining these facts with~\eqref{pepitito1} and taking into account that, by~\eqref{formula4},
$$
\sum_{h=1}^{v-1} \alpha_1(\gamma_1)(-\lambda,-h\lambda) = \left(v-\left\lceil \frac{v}{d} \right\rceil\right)\gamma_1,
$$
we obtain the formula in the statement.
\end{proof}

In the following proposition we study a case where $H^{d_1d_2}_{\lambda}$ and $I$ are p-groups.

\begin{proposition}\label{caso particular 1} Let $p$ be an odd prime, and let $d_1\coloneqq p^{\eta_1}$, $d_2\coloneqq p^{\eta_2}$ and $I\coloneqq \mathds{Z}_{p^r}$ with $0<r\le \eta_1 \le \eta_2$. If $\lambda= p^a$ with $0\le a<\eta_1$, then there exist $k_1,k_2\in \mathds{Z}_{p^r}$ such that
\begin{align*}
& \text{$A_i$ is the multiplication by~$1 +p k_i$,}\\
& T_1(\gamma_1,\gamma_2,\mathfrak{f}_{11}, \mathfrak{f}_{12}, \mathfrak{f}_{21}, \mathfrak{f}_{22}) = \Bigl(-pk_1 \gamma_1,-pk_1 \gamma_2,-pk_2 \gamma_1, -\frac{d_2}{d}\gamma_1-pk_2\gamma_2\Bigr), \\
& T_2(\gamma_1,\gamma_2,\mathfrak{f}_{11}, \mathfrak{f}_{12}, \mathfrak{f}_{21},\mathfrak{f}_{22}) = \Biggl(0,0,0, -\sum_{\ell=0}^{\frac{d_2}{d}-1} (1+pk_2)^{d_2-1+\ell d}  \gamma_1 \biggr), \\
&  T_3(\gamma_1,\gamma_2,\mathfrak{f}_{11}, \mathfrak{f}_{12}, \mathfrak{f}_{21},\mathfrak{f}_{22})= \bigl(-pk_2 \mathfrak{f}_{11} + pk_1 \mathfrak{f}_{21}, -p^a \mathfrak{f}_{11} - pk_2 \mathfrak{f}_{12}+pk_1 \mathfrak{f}_{22}\bigr)
\shortintertext{and}
& S(t_1,t_2)=\bigl(0,0,pk_1 t_1,pk_1 t_2,pk_2 t_1,p^a t_1+pk_2 t_2\bigr),
\end{align*}
where $d\coloneqq d_1/\lambda=p^{\eta_1-a}$ and $T_1$, $T_2$ and $T_3$ are as in~Pro\-position~\ref{caso lambda no nulo  pero Yleft cero}.
\end{proposition}

\begin{proof} Let $a_i\in \mathds{Z}_{p^r}$ be such that $A_i$ is the multiplication by $a_i$. Since $A_i^{p^{\eta_i}}=\Ide_I$,
$$
1=a_i^{p^{\eta_i}}\equiv a_i \mod p.
$$
Hence, there exist $k_1,k_2\in \mathds{Z}_{p^r}$ such that $a_i=1+pk_i$. The identities for $T_1$, $T_3$ and $S$ follow immediately from Propo\-sition~\ref{caso lambda no nulo  pero Yleft cero}. Next we prove the formula for $T_2$. We claim that
\begin{equation}\label{suma de potencias de A}
\sum_{\ell=0}^{d_i-1} a_i^{\ell}=0.
\end{equation}
In order to prove this we first note that
$$
\sum_{\ell=0}^{d_i-1} (1+ pk_i)^{\ell}=\sum_{\ell=0}^{d_i-1} \sum_{j=0}^{\ell}\binom{\ell}{j}p^j k_i^j =\sum_{j=0}^{d_i-1} \sum_{\ell=j}^{d_i-1} \binom{\ell}{j}p^j k_i^j= \sum_{j=0}^{d_i-1}  \binom{d_i }{j+1}p^j k_i^j.
$$
Write $j+1 = p^tv$ with $p\nmid v$ and $t\ge 0$. It is well known that $p^{\eta_i-t}\mid \binom{d_i }{j+1}$. Hence, the exponent $w$, of $p$ in $\binom{d_i}{j+1} p^j$, is at least $\eta_i-t+j\ge \eta_i$, since $j\ge t$. Consequently
$$
\sum_{\ell=0}^{d_i-1} a_i^{\ell}=\sum_{\ell=0}^{d_i-1} (1+ pk_i)^{\ell}=0,
$$
because $r\le \eta_i$. From~\eqref{suma de potencias de A} and Proposition~\ref{caso lambda no nulo  pero Yleft cero}, it follows that
\begin{equation}\label{T2}
T_2((\gamma_i),(\mathfrak{f}_{ij}))=\left(0,0,0,\sum_{\ell=0}^{d_2-1} \mathrm{r}_{d_1}(\lambda \ell) a_2^{\ell-1} \mathfrak{f}_{21}+\sum_{\substack{\ell=0\\ d \nmid \ell}}^{d_2-1} a_2^{d_2-1-\ell} \gamma_1 \right).
\end{equation}
We claim that
$$
\sum_{\ell=0}^{d_2-1} \mathrm{r}_{d_1}(\lambda \ell) a_2^{\ell-1} \mathfrak{f}_{21} = 0.
$$
To prove this, first observe that
$$
\sum_{\ell=0}^{d_2-1} \ell (1+ pk_2)^{\ell}=\sum_{\ell=0}^{d_2-1} \sum_{j=0}^{\ell}\ell\binom{\ell}{j}p^j k_2^j=\sum_{j=0}^{d_2-1} \sum_{\ell=j}^{d_2-1} \ell\binom{\ell}{j}p^j k_2^j= \sum_{j=0}^{d_2-1}  \left( (d_2-1)\binom{d_2}{j+1}-\binom{d_2}{j+2} \right)p^j k_2^j,
$$
where the last equality hold by~\cite{Spivey}*{Identity~73}. From our earlier divisibility argument, we know that $p^r\mid \binom{d_2}{j+1} p^j$ and, since $p$ is odd, a similar reasoning shows that $p^r\mid \binom{d_2}{j+2} p^j$. Consequently,
$$
\sum_{\ell=0}^{d_2-1} a_2^{\ell-1} \mathrm{r}_{d_1}(\lambda \ell) \mathfrak{f}_{21}= \lambda a^{-1}\sum_{\ell=0}^{d_2-1} \ell a_2^{\ell}  \mathfrak{f}_{21} = \lambda a^{-1} \sum_{\ell=0}^{d_2-1} \ell (1+ pk_2)^{\ell} \mathfrak{f}_{21}=0,
$$
as claimed. The identity for $T_2$ now follows from this, equality~\eqref{T2} and the fact that
$$
\sum_{\substack{\ell=0\\ d \nmid \ell}}^{d_2-1} a_2^{d_2-1-\ell} \gamma_1=\sum_{\ell=0}^{d_2-1} a_2^{d_2-1-\ell} \gamma_1- \sum_{\substack{\ell=0\\ d \mid \ell}}^{d_2-1} a_2^{d_2-1-\ell} \gamma_1= \sum_{\ell=0}^{d_2-1} a_2^{\ell} \gamma_1 -\sum_{j=0}^{\frac{d_2}{d}-1} a_2^{d_2-1-jd}  \gamma_1 = - \sum_{\ell=0}^{\frac{d_2}{d}-1} a_2^{d_2-1+\ell d}  \gamma_1,
$$
where the last equality follows from~\eqref{suma de potencias de A}.
\end{proof}

\begin{corollary}\label{primer corolario} Assume that we are under the hypothesis of Proposition~\ref{caso particular 1}.  Let $k_1,k_2\in \mathds{Z}_{p^r}$ be such that $A_i$ is the multiplication by~$1 +p k_i$. Write $k_1 = p^{u_1}v_1$ and~$k_2 = p^{u_2}v_2$ with $0\le u_1,u_2< r$ and $v_1,v_2$ units of $\mathds{Z}_{p^r}$. The following assertions hold:

\begin{enumerate}[label=\arabic*)]

\item If $a\le u_1+1\le u_2+1$, then
\begin{align*}
\ho_{\blackdiamond}^2(H^{d_1d_2}_{\lambda},I) & \simeq p^{r-\epsilon_1} \mathds{Z}_{p^r} \oplus p^{r-u_1-1} \mathds{Z}_{p^r}\oplus p^{r-a} \mathds{Z}_{p^r} \oplus \frac{\mathds{Z}_{p^r}}{p^{u_1+1} \mathds{Z}_{p^r}}\oplus p^{r-u_1-1} \mathds{Z}_{p^r}\oplus \frac{\mathds{Z}_{p^r}}{p^a \mathds{Z}_{p^r}}\\
&\simeq\mathds{Z}_{p^{\epsilon_1}}\oplus\mathds{Z}_{p^{u_1+1}}\oplus \mathds{Z}_{p^a}\oplus\mathds{Z}_{p^{u_1+1}}\oplus\mathds{Z}_{p^{u_1+1}}\oplus \mathds{Z}_{p^a},
\end{align*}
where $\epsilon_1\coloneqq \min(\eta_2-\eta_1+a,u_1+1)$.

\item If $a\le u_2+1\le u_1+1$ and $u_2-a +1\le  u_1-u_2$, then
\begin{align*}
\ho_{\blackdiamond}^2(H^{d_1d_2}_{\lambda},I) & \simeq p^{r-\epsilon_1} \mathds{Z}_{p^r} \oplus p^{r-u_2-1} \mathds{Z}_{p^r}\oplus p^{r-a} \mathds{Z}_{p^r} \oplus p^{r-\epsilon_2} \mathds{Z}_{p^r}\oplus \frac{\mathds{Z}_{p^r}}{p^{\epsilon_2} \mathds{Z}_{p^r}}\oplus \frac{\mathds{Z}_{p^r}}{p^a \mathds{Z}_{p^r}}\\
& \simeq \mathds{Z}_{p^{\epsilon_1}}\oplus\mathds{Z}_{p^{u_2+1}}\oplus\mathds{Z}_{p^a}\oplus \mathds{Z}_{p^{\epsilon_2}}\oplus \mathds{Z}_{p^{\epsilon_2}} \oplus \mathds{Z}_{p^a},
\end{align*}
where $\epsilon_1\coloneqq \min(\eta_2-\eta_1+a,u_2+1)$ and $\epsilon_2\coloneqq 2u_2+2-a$.

\item If $a\le u_2+1\le u_1+1$ and $u_1-u_2\le u_2-a+1$, then
\begin{align*}
\ho_{\blackdiamond}^2(H^{d_1d_2}_{\lambda},I) & \simeq p^{r-\epsilon_1} \mathds{Z}_{p^r} \oplus p^{r-u_2-1} \mathds{Z}_{p^r}\oplus p^{r-a}\mathds{Z}_{p^r} \oplus \frac{\mathds{Z}_{p^r}}{p^{u_1+1} \mathds{Z}_{p^r}}\oplus p^{r-u_1-1} \mathds{Z}_{p^r}\oplus \frac{\mathds{Z}_{p^r}}{p^a \mathds{Z}_{p^r}}\\
&\simeq\mathds{Z}_{p^{\epsilon_1}}\oplus\mathds{Z}_{p^{u_2+1}}\oplus\mathds{Z}_{p^a}\oplus\mathds{Z}_{p^{u_1+1}}\oplus\mathds{Z}_{p^{u_1+1}}\oplus \mathds{Z}_{p^a},
\end{align*}
where $\epsilon_1\coloneqq \min(\eta_2-\eta_1+a,u_2+1)$.

\item If $u_1\le u_2$ and $u_1<a$, then
\begin{align*}
\ho_{\blackdiamond}^2(H^{d_1d_2}_{\lambda},I) & \simeq p^{r-\epsilon_1} \mathds{Z}_{p^r} \oplus p^{r-u_1-1} \mathds{Z}_{p^r}\oplus \frac{\mathds{Z}_{p^r}}{p^{u_1+1}\mathds{Z}_{p^r}} \oplus \frac{\mathds{Z}_{p^r}}{p^{u_1+1} \mathds{Z}_{p^r}}\oplus p^{r-u_1-1} \mathds{Z}_{p^r}\oplus p^{r-u_1-1} \mathds{Z}_{p^r}\\
& \simeq \mathds{Z}_{p^{\epsilon_1}}\oplus \mathds{Z}_{p^{u_1+1}}\oplus \mathds{Z}_{p^{u_1+1}}\oplus \mathds{Z}_{p^{u_1+1}}\oplus \mathds{Z}_{p^{u_1+1}} \oplus \mathds{Z}_{p^{u_1+1}},
\end{align*}
where $\epsilon_1\coloneqq \min(\eta_2-\eta_1+a,u_1+1)$.

\item If $u_2\le u_1$ and $u_2<a$, then
\begin{align*}
\ho_{\blackdiamond}^2(H^{d_1d_2}_{\lambda},I) & \simeq p^{r-\epsilon_1} \mathds{Z}_{p^r} \oplus p^{r-u_2-1} \mathds{Z}_{p^r}\oplus p^{r-u_2-1} \mathds{Z}_{p^r}\oplus p^{r-u_2-1} \mathds{Z}_{p^r}\oplus\frac{\mathds{Z}_{p^r}}{p^{u_2+1} \mathds{Z}_{p^r}}\oplus \frac{\mathds{Z}_{p^r}}{p^{u_2+1} \mathds{Z}_{p^r}}\\
& \simeq \mathds{Z}_{p^{\epsilon_1}} \oplus \mathds{Z}_{p^{u_2+1}}\oplus \mathds{Z}_{p^{u_2+1}} \oplus \mathds{Z}_{p^{u_2+1}} \oplus \mathds{Z}_{p^{u_2+1}} \oplus \mathds{Z}_{p^{u_2+1}},
\end{align*}
where $\epsilon_1\coloneqq \min(\eta_2-\eta_1+a,u_2+1)$.
\end{enumerate}
\end{corollary}

\begin{proof} 1)\enspace Clearly $h_1\coloneqq p^{u_1+1-a}v_1$ and $h_2\coloneqq p^{u_2-u_1}v_2v_1^{-1}$ satisfy $pk_1=h_1 p^a$ and $k_2=h_2 k_1$. Define $\phi\colon I^6\to I^6$ by
$$
\phi(\gamma_1,\gamma_2,\mathfrak{f}_{11},\mathfrak{f}_{12},\mathfrak{f}_{21},\mathfrak{f}_{22}) \coloneqq (\gamma_1,\gamma_2, \mathfrak{f}_{11}+h_1 \mathfrak{f}_{22}, \mathfrak{f}_{12},\mathfrak{f}_{21}+h_2 \mathfrak{f}_{11}+h_1 h_2 \mathfrak{f}_{22}, \mathfrak{f}_{22}+ h_2 \mathfrak{f}_{12}).
$$
Note that $\phi$ is invertible and that
$$
\phi^{-1}(\gamma_1,\gamma_2,x_1,\mathfrak{f}_{12},x_2,x_3) \coloneqq (\gamma_1,\gamma_2,x_1-h_1 x_3+h_1 h_2 \mathfrak{f}_{12}, \mathfrak{f}_{12}, x_2-h_2 x_1, x_3-h_2 \mathfrak{f}_{12}).
$$
Set $\wt{T}_i\coloneqq T_i\xcirc \phi$ and $\wt{S}\coloneqq \phi^{-1}\xcirc S$, where the $T_i$'s and $S$ are as in the previous proposition. By~\eqref{formula00}, we have
\begin{equation}\label{homo0}
\ho_{\blackdiamond}^2(H^{d_1d_2}_{\lambda},I) \simeq \frac{\ker T_1\cap\ker T_2\cap\ker T_3 }{\ima S} \simeq \frac{\ker \wt{T}_1\cap\ker \wt{T}_2\cap\ker \wt{T}_3 }{\ima \wt{S}},
\end{equation}
where the last isomorphism is induced by $\phi^{-1}$. A direct computation shows that
\begin{align*}
&\wt{T}_1=T_1,\quad \wt{T}_2=T_2,\quad \wt{T}_3(\gamma_1,\gamma_2,\mathfrak{f}_{11},\mathfrak{f}_{12},\mathfrak{f}_{21},\mathfrak{f}_{22})   =(pk_1 \mathfrak{f}_{21},-p^a \mathfrak{f}_{11})
\shortintertext{and}
&\wt{S}(t_1,t_2)= (0,0,0,pk_1 t_2,0,p^a t_1).
\end{align*}
Combining this with~\eqref{homo0}, we obtain
\begin{equation}\label{homo1}
\ho_{\blackdiamond}^2(H^{d_1d_2}_{\lambda},I) \simeq (I_{pk_1}\cap I_{\frac{d_2}{d}}) \oplus I_{p k_1} \oplus I_{p^a} \oplus \frac{I}{pk_1 I} \oplus I_{pk_1} \oplus \frac{I}{p^a I},
\end{equation}
where $d=d_1/\lambda = p^{\eta_1-a}$. The result follows from this fact using that $I=\mathds{Z}_{p^r}$.

\smallskip

\noindent 2)\enspace Clearly $h_1\coloneqq p^{u_1-u_2}v_1v_2^{-1}$, $h_2\coloneqq p^{u_2-a+1}v_2$ and $h_3\coloneqq p^{u_1-2u_2+a-1}v_1v_2^{-2}$ satisfy $pk_2=h_2 p^a$, $k_1=h_1k_2$ and $h_1=h_2 h_3$. Define $\phi\colon I^6\to I^6$ by and $\psi\colon I^2\to I^2$, by
\begin{align*}
& \phi(\gamma_1,\gamma_2,\mathfrak{f}_{11},\mathfrak{f}_{12},\mathfrak{f}_{21},\mathfrak{f}_{22}) \coloneqq (\gamma_1,\gamma_2, \mathfrak{f}_{11}-h_2 \mathfrak{f}_{12}+h_2 h_3 \mathfrak{f}_{21}+h_1 h_2 \mathfrak{f}_{22}, \mathfrak{f}_{12}-h_3 \mathfrak{f}_{21},\mathfrak{f}_{21}+h_2 \mathfrak{f}_{22},\mathfrak{f}_{22})
\shortintertext{and}
&\psi(t_1,t_2)\coloneqq (t_1-h_2 t_2,t_2).
\end{align*}
Note that $\phi$ and $\psi$ are invertible and that
$$
\phi^{-1}(\gamma_1,\gamma_2,x_1,x_2,x_3,\mathfrak{f}_{22}) = (\gamma_1,\gamma_2,x_1+h_2 x_2 - h_1 h_2 \mathfrak{f}_{22}, x_2+h_3 x_3-h_1 \mathfrak{f}_{22}, x_3-h_2 \mathfrak{f}_{22}, \mathfrak{f}_{22}).
$$
Set $\wt{T}_i\coloneqq T_i\xcirc \phi$ and $\wt{S}\coloneqq \phi^{-1}\xcirc S\xcirc \psi$, where the $T_i$'s and $S$ are as in Proposition~\ref{caso particular 1}. By~\eqref{formula00}, we have
\begin{equation}\label{homo2}
\ho_{\blackdiamond}^2(H^{d_1d_2}_{\lambda},I) \simeq \frac{\ker T_1\cap\ker T_2\cap\ker T_3 }{\ima S} \simeq \frac{\ker \wt{T}_1\cap\ker \wt{T}_2\cap\ker \wt{T}_3 }{\ima \wt{S}},
\end{equation}
where the last isomorphism is induced by $\phi^{-1}$. A direct computation shows that
\begin{align*}
&\wt{T}_1=T_1,\quad \wt{T}_2=T_2,\quad \wt{T}_3(\gamma_1,\gamma_2,\mathfrak{f}_{11},\mathfrak{f}_{12},\mathfrak{f}_{21},\mathfrak{f}_{22})   =(-pk_2 \mathfrak{f}_{11}-p k_2 h_2 \mathfrak{f}_{12},-p^a \mathfrak{f}_{11})
\shortintertext{and}
&\wt{S}(t_1,t_2)=(0,0,0,0,-pk_2h_2t_2,p^at_1).
\end{align*}
Combining this with~\eqref{homo2}, we obtain
\begin{equation}\label{homo3}
\ho_{\blackdiamond}^2(H^{d_1d_2}_{\lambda},I)\simeq (I_{pk_2}\cap I_{\frac{d_2}{d}}) \oplus I_{p k_2} \oplus I_{p^a} \oplus I_{pk_2 h_2} \oplus \frac{I}{pk_2 h_2 I} \oplus \frac{I}{p^a I},
\end{equation}
where $d\coloneqq d_1/\lambda = p^{\eta_1-a}$. The result follows from this fact using that $I=\mathds{Z}_{p^r}$ and $2u_2+2-a\le u_1+1\le r$.

\smallskip

\noindent 3)\enspace Clearly $h_1\coloneqq p^{u_1-u_2}v_1v_2^{-1}$, $h_2\coloneqq p^{u_2-a+1}v_2$ and $h_3\coloneqq p^{2u_2-u_1-a+1}v_1^{-1}v_2^2$ satisfy $pk_2=h_2 p^a$, $k_1=h_1k_2$ and $h_2=h_1 h_3$. Define $\phi\colon I^6\to I^6$ and $\psi\colon I^2\to I^2$, by
\begin{align*}
&\phi(\gamma_1,\gamma_2,\mathfrak{f}_{11},\mathfrak{f}_{12},\mathfrak{f}_{21},\mathfrak{f}_{22}) \coloneqq (\gamma_1,\gamma_2, \mathfrak{f}_{11}+h_1 h_2 \mathfrak{f}_{22}-h_2 \mathfrak{f}_{12}, \mathfrak{f}_{12}, \mathfrak{f}_{21}+h_2 \mathfrak{f}_{22}-h_3 \mathfrak{f}_{12},\mathfrak{f}_{22})
\shortintertext{and}
&\psi(t_1,t_2)\coloneqq (t_1-h_2 t_2,t_2).
\end{align*}
Note that $\phi$ and $\psi$ are invertible and that
$$
\phi^{-1}(\gamma_1,\gamma_2,x_1,\mathfrak{f}_{12},x_2,\mathfrak{f}_{22}) = (\gamma_1,\gamma_2,x_1- h_1 h_2 \mathfrak{f}_{22}+h_2 \mathfrak{f}_{12}, \mathfrak{f}_{12}, x_2-h_2 \mathfrak{f}_{22}+h_3 \mathfrak{f}_{12}, \mathfrak{f}_{22}).
$$
Set $\wt{T}_i\coloneqq T_i\xcirc \phi$ and $\wt{S}\coloneqq \phi^{-1}\xcirc S\xcirc \psi$, where the $T_i$'s and $S$ are as in Proposition~\ref{caso particular 1}. By~\eqref{formula00}, we have
\begin{equation}\label{homo4}
\ho_{\blackdiamond}^2(H^{d_1d_2}_{\lambda},I) \simeq \frac{\ker T_1\cap\ker T_2\cap\ker T_3 }{\ima S} \simeq \frac{\ker \wt{T}_1\cap\ker \wt{T}_2\cap\ker \wt{T}_3 }{\ima \wt{S}},
\end{equation}
where the last isomorphism is induced by $\phi^{-1}$. A direct computation shows that
\begin{align*}
& \wt{T}_1=T_1,\quad \wt{T}_2=T_2,\quad \wt{T}_3(\gamma_1,\gamma_2,\mathfrak{f}_{11},\mathfrak{f}_{12},\mathfrak{f}_{21},\mathfrak{f}_{22})   =(-pk_2 \mathfrak{f}_{11}+p k_1  \mathfrak{f}_{21},-p^a \mathfrak{f}_{11})
\shortintertext{and}
&\wt{S}(t_1,t_2)=(0,0,0,pk_1 t_2,0, p^a t_1).
\end{align*}
Combining this with~\eqref{homo4}, we obtain
\begin{equation}\label{homo5}
\ho_{\blackdiamond}^2(H^{d_1d_2}_{\lambda},I)\simeq (I_{pk_2}\cap I_{\frac{d_2}{d}})\oplus I_{p k_2}\oplus I_{p^a}\oplus \frac{I}{pk_1 I} \oplus I_{pk_1}\oplus \frac{I}{p^a I},
\end{equation}
where $d\coloneqq d_1/\lambda = p^{\eta_1-a}$. The result follows from this fact using that $I=\mathds{Z}_{p^r}$.

\smallskip

\noindent 4)\enspace Clearly $h_1\coloneqq p^{a-u_1-1}v_1^{-1}$ and $h_2\coloneqq p^{u_2-u_1}v_1^{-1}v_2$ satisfy $p^a=pk_1 h_1$ and $k_2=h_2k_1$. Define $\phi\colon I^6\to I^6$ by
$$
\phi(\gamma_1,\gamma_2,\mathfrak{f}_{11},\mathfrak{f}_{12},\mathfrak{f}_{21},\mathfrak{f}_{22}) \coloneqq (\gamma_1,\gamma_2, \mathfrak{f}_{11}, \mathfrak{f}_{12},\mathfrak{f}_{21}+h_2 \mathfrak{f}_{11}, \mathfrak{f}_{22}+h_1 \mathfrak{f}_{11}+h_2 \mathfrak{f}_{12}).
$$
Note that $\phi$ is invertible and that
$$
\phi^{-1}(\gamma_1,\gamma_2,\mathfrak{f}_{11},\mathfrak{f}_{12},x_1,x_2) = (\gamma_1,\gamma_2, \mathfrak{f}_{11}, \mathfrak{f}_{12},x_1-h_2 \mathfrak{f}_{11}, x_2-h_1 \mathfrak{f}_{11}-h_2 \mathfrak{f}_{12}).
$$
Set $\wt{T}_i\coloneqq T_i\xcirc \phi$ and $\wt{S}\coloneqq \phi^{-1}\xcirc S$, where the $T_i$'s and $S$ are as in Proposition~\ref{caso particular 1}. By~\eqref{formula00}, we have
\begin{equation}\label{homo6}
\ho_{\blackdiamond}^2(H^{d_1d_2}_{\lambda},I) \simeq \frac{\ker T_1\cap\ker T_2\cap\ker T_3 }{\ima S} \simeq \frac{\ker \wt{T}_1\cap\ker \wt{T}_2\cap\ker \wt{T}_3 }{\ima \wt{S}},
\end{equation}
where the last isomorphism is induced by $\phi^{-1}$. A direct computation shows that
\begin{align*}
&\wt{T}_1=T_1,\quad \wt{T}_2=T_2,\quad \wt{T}_3(\gamma_1,\gamma_2,\mathfrak{f}_{11},\mathfrak{f}_{12},\mathfrak{f}_{21},\mathfrak{f}_{22}) = (pk_1 \mathfrak{f}_{21},pk_1 \mathfrak{f}_{22})
\shortintertext{and}
&\wt{S}(t_1,t_2)=(0,0,p k_1 t_1,pk_1t_2,0,0).
\end{align*}
Combining this with~\eqref{homo6}, we obtain
\begin{equation}\label{homo7}
\ho_{\blackdiamond}^2(H^{d_1d_2}_{\lambda},I)\simeq (I_{pk_1}\cap I_{\frac{d_2}{d}})\oplus I_{p k_1}\oplus\frac{I}{pk_1 I}\oplus \frac{I}{pk_1 I}\oplus I_{pk_1}\oplus I_{pk_1}.
\end{equation}
The result follows from this fact using that $I=\mathds{Z}_{p^r}$.

\smallskip

\noindent 5)\enspace Clearly $h_1\coloneqq p^{u_1-u_2}v_1v_2^{-1}$ and $h_2\coloneqq p^{a-u_2-1}v_2^{-1}$ satisfy $k_1=h_1k_2$ and $p^a=pk_2 h_2$. Define $\phi\colon I^6\to I^6$ by
$$
\phi(\gamma_1,\gamma_2,\mathfrak{f}_{11},\mathfrak{f}_{12},\mathfrak{f}_{21},\mathfrak{f}_{22}) \coloneqq (\gamma_1,\gamma_2, \mathfrak{f}_{11}+h_1 \mathfrak{f}_{21}, \mathfrak{f}_{12}+h_1 \mathfrak{f}_{22},\mathfrak{f}_{21},\mathfrak{f}_{22}+h_2 \mathfrak{f}_{21}).
$$
Note that $\phi$ is invertible and that
$$
\phi^{-1}(\gamma_1,\gamma_2,x_1,x_2,\mathfrak{f}_{21},x_3) = (\gamma_1,\gamma_2,x_1-h_1 \mathfrak{f}_{21},x_2-h_1 x_3+h_1 h_2 \mathfrak{f}_{21}, \mathfrak{f}_{21},x_3-h_2 \mathfrak{f}_{21}).
$$
Set $\wt{T}_i\coloneqq T_i\xcirc \phi$ and $\wt{S}\coloneqq \phi^{-1}\xcirc S$, where the $T_i$'s and $S$ are as in Proposition~\ref{caso particular 1}. By~\eqref{formula00}, we have
\begin{equation}\label{homo8}
\ho_{\blackdiamond}^2(H^{d_1d_2}_{\lambda},I) \simeq \frac{\ker T_1\cap\ker T_2\cap\ker T_3 }{\ima S} \simeq \frac{\ker \wt{T}_1\cap\ker \wt{T}_2\cap\ker \wt{T}_3 }{\ima \wt{S}},
\end{equation}
where the last isomorphism is induced by $\phi^{-1}$. A direct computation shows that
\begin{align*}
&\wt{T}_1=T_1,\quad \wt{T}_2=T_2,\quad \wt{T}_3(\gamma_1,\gamma_2,\mathfrak{f}_{11},\mathfrak{f}_{12},\mathfrak{f}_{21},\mathfrak{f}_{22})   =(-pk_2 \mathfrak{f}_{11},-p^a \mathfrak{f}_{11}-pk_2 \mathfrak{f}_{12})
\shortintertext{and}
&\wt{S}(t_1,t_2)=(0,0,0,0,pk_2t_1,pk_2 t_2).
\end{align*}
Combining this with~\eqref{homo8}, we obtain
\begin{equation}\label{homo9}
\ho_{\blackdiamond}^2(H^{d_1d_2}_{\lambda},I)\simeq (I_{pk_2}\cap I_{\frac{d_2}{d}})\oplus I_{p k_2}\oplus I_{pk_2}\oplus I_{pk_2}\oplus\frac{I}{pk_2 I}\oplus\frac{I}{pk_2 I}.
\end{equation}
The result follows from this fact using that $I=\mathds{Z}_{p^r}$.
\end{proof}

\begin{remark}\label{remark2} Assume that we are under the hypothesis of item~1) of Corollary~\ref{primer corolario}. The map $\phi\colon I^6\to I^6$, induces an~iso\-mor\-phism
$$
\ov{\phi}\colon (I_{pk_1}\cap I_{\frac{d_2}{d}}) \oplus I_{p k_1} \oplus I_{p^a} \oplus \frac{I}{pk_1 I} \oplus I_{pk_1} \oplus \frac{I}{p^a I} \longrightarrow \frac{\ker T_1\cap \ker T_2\cap \ker T_3}{\ima S},
$$
where $T_1$, $T_2$, $T_3$ and $S$ are as in Proposition~\ref{caso particular 1}. Let $Z\subseteq (I_{pk_1}\cap I_{\frac{d_2}{d}}) \oplus I_{p k_1} \oplus I_{p^a} \oplus I \oplus I_{pk_1} \oplus I$ be such that the canonical map~$Z\to (I_{pk_1}\cap I_{\frac{d_2}{d}}) \oplus I_{p k_1} \oplus I_{p^a} \oplus \frac{I}{pk_1 I} \oplus I_{pk_1} \oplus \frac{I}{p^a I}$ is a bi\-jec\-tion and let $T$ be as in~\eqref{formula0}. Notice that there is an obvious choice for $Z$ and that
$$
T(\gamma_1,\gamma_2,\mathfrak{f}_{11},\mathfrak{f}_{12},\mathfrak{f}_{21},\mathfrak{f}_{22}) = (\alpha_{\gamma_1,\gamma_2}, f_{\mathfrak{f}_{11},\mathfrak{f}_{12},\mathfrak{f}_{21},\mathfrak{f}_{22},\gamma_1, \gamma_2}).
$$
Clearly $T(\phi(Z))$ is a complete family of representatives of~$2$-cocycles modulo coboundaries in $\wh{\mathcal{C}}_N^*(H^{d_1d_2}_{\lambda},I)$. Similar observations are pertinent for the items~2), 3), 4) and~5). We leave the task of formulating them to the reader. From the preceding discussion, it becomes clear that, in all these cases, computing $T$ is indispensable for constructing the extensions. Since $\alpha_{\gamma_1,\gamma_2}$ is already known from~\eqref{alpha_{gamma_1,gamma_2}}, the determination of $T$ reduces to establishing the following proposition.
\end{remark}

\begin{proposition}\label{1calculo de f} Assume that we are under the hypothesis of Proposition~\ref{caso particular 1}, and let $T_1$, $T_2$ and $T_3$ be~as~in~Pro\-position~\ref{caso lambda no nulo  pero Yleft cero}. For $(\gamma_1,\gamma_2,\mathfrak{f}_{11}, \mathfrak{f}_{12},\mathfrak{f}_{21}, \mathfrak{f}_{22})\in \ker T_1\cap \ker T_2\cap \ker T_3$, let $f\coloneqq f_{\mathfrak{f}_{11},\mathfrak{f}_{12},\mathfrak{f}_{21},\mathfrak{f}_{22}, \gamma_1, \gamma_2}$. We have
\begin{align*}
f(a_1^{\times \ell_1}\times a_2^{\times \ell_2},(u,v)) & = (1+pk_2)^{\ell_2} \sum_{j=0}^{\ell_1-1} (1+pk_1)^j (u \mathfrak{f}_{11}+ v\mathfrak{f}_{12}) + \sum_{j=0}^{\ell_2-1} (1+pk_2)^{\ell_2-j-1}(u\!-\!\lambda jv) \mathfrak{f}_{21}\\
& + v \sum_{j=0}^{\ell_2-1} (1+pk_2)^j \mathfrak{f}_{22} + \left(v-\left\lceil \frac{v}{p^{\eta_1-a}} \right\rceil\right) \sum_{j=0}^{\ell_2-1} (1+pk_2)^j \gamma_1 + C^{\ell_2}_{uv} \gamma_1,
\end{align*}
where $0\le\ell_1,u<p^{\eta_1}$, $0\le\ell_2,v<p^{\eta_2}$ and $C^{\ell_2}_{uv} = \#\{0\le j < \ell_2: \mathrm{r}_{p^{\eta_1}}(u-\lambda jv)+\mathrm{r}_{p^{\eta_1}}(-\lambda_v)\ge p^{\eta_1}\}$.
\end{proposition}

\begin{proof} Arguing as in the proof of Proposition~\ref{calculo de f}, we obtain
$$
f(a_1^{\times \ell_1}\times a_2^{\times \ell_2},(u,v)) = a_2^{\times \ell_2}\blackdiamond f(a_1^{\times \ell_1},(u,v)) + f(a_2^{\times \ell_2},(u,v)).
$$
Moreover, by~\cite{GGV2}*{Proposition~2.12} and~\eqref{formula3},
\begin{align*}
& f(a_1^{\times \ell_1},(u,v)) = \sum_{j=0}^{\ell_1-1} a_1^{\times (\ell_1-j-1)} \blackdiamond f(a_1,(u,v)) = \sum_{j=0}^{\ell_1-1} (1+pk_1)^{\ell_1-j-1} f(a_1,(u,v))
\shortintertext{and}
& f(a_2^{\times \ell_2},(u,v)) = \sum_{j=0}^{\ell_2-1} a_2^{\times (\ell_2-j-1)} \blackdiamond f(a_1,a_2^{\times j}\cdot (u,v)) = \sum_{j=0}^{\ell_2-1} (1+pk_2)^{\ell_2-j-1} f(a_1,(u-\lambda jv,v)).
\end{align*}
Furthermore, by~\eqref{formula3} and~\cite{GGV2}*{Identity~(2.32)},
$$
f(a_1,(u,v)) = u\mathfrak{f}_{11}+v\mathfrak{f}_{12}+\sum_{h=1}^{u-1} \alpha_{\gamma_1,\gamma_2}(e_1,he_1)+\sum_{h=1}^{v-1} \alpha_{\gamma_1,\gamma_2}(e_2,he_2) + \alpha_{\gamma_1,\gamma_2}(ue_1,ve_2)= u \mathfrak{f}_{11}+ v\mathfrak{f}_{12}
$$
and
\begin{align*}
f(a_2,(u-\lambda jv,v)) & = \mathrm{r}_{p^{\eta_1}}(u-\lambda jv)\mathfrak{f}_{21}+v\mathfrak{f}_{22} + \sum_{h=1}^{\mathrm{r}_{p^{\eta_1}}(u-\lambda jv)-1} \alpha_{\gamma_1,\gamma_2}(e_1,he_1)\\
& + \sum_{h=1}^{v-1} \alpha_{\gamma_1,\gamma_2}(-\lambda e_1+e_2,-h\lambda e_1+he_2) + \alpha_{\gamma_1,\gamma_2}((u-\lambda jv)e_1,-\lambda ve_1+ve_2)\\
& = (u-\lambda jv)\mathfrak{f}_{21}+v\mathfrak{f}_{22} + \sum_{h=1}^{v-1} \alpha_1(\gamma_1)(-\lambda,-h\lambda) + \alpha_1(\gamma_1)(u-\lambda jv,-\lambda v),
\end{align*}
where in the last equality we have used that $\mathrm{r}_{p^{\eta_1}}(u-\lambda jv)\mathfrak{f}_{21} = (u-\lambda jv)\mathfrak{f}_{21}$, because $r\le \eta_1$. Using these facts and taking into account that, by~\eqref{formula4},
$$
\sum_{h=1}^{v-1} \alpha_1(\gamma_1)(-\lambda,-h\lambda) = \left(v-\left\lceil \frac{v}{p^{\eta_1-a}} \right\rceil\right)\gamma_1,
$$
we obtain the formula in the statement.
\end{proof}

\begin{remark} For the computation of $f$ can be useful note that
$$
\sum_{j=0}^{\ell_i-1} (1+pk_i)^j = \sum_{j=0}^{\ell_i-1} \sum_{\ell=0}^j\binom{j}{\ell} p^{\ell}k^{\ell} = \sum_{\ell=0}^{\ell_i-1} \sum_{j=\ell}^{\ell_i-1} \binom{j}{\ell} p^{\ell}k^{\ell} = \sum_{\ell=0}^{\ell_i-1} \binom{\ell_i}{\ell+1} p^{\ell}k^{\ell},
$$
for $i\in\{1,2\}$.
\end{remark}

\subsubsection[The case \texorpdfstring{$\lambda=0$}{lambda=0}]{The case \texorpdfstring{$\pmb{\lambda=0}$}{lambda=0}}
In this subsubsection we assume that $I$ is an arbitrary abelian group and that $\lambda=0$, but not that $\Yleft= 0$. Recall from the beginning of Section~\ref{Seccion 4}, that $s=n=2$ and that we can take $a_1=e_1\coloneqq (1,0)$ and $a_2=e_2\coloneqq (0,1)$. Moreover, by~\eqref{def de punto}, the linear cycle $H$ is trivial. We take~$A_1$, $A_2$, $B_1$ and $B_2$ as in Proposition~\ref{caracterizacion primer caso}.

\begin{proposition}\label{prop1} Let $T_1$, $T_2$, $T_3$ and $S$ be as in~\eqref{T_1}--\eqref{S}. We have:
\begin{align*}
& T_1(\gamma_1,\gamma_2,\mathfrak{f}_{11},\mathfrak{f}_{12},\mathfrak{f}_{21},\mathfrak{f}_{22}) = \left(d_1 \mathfrak{f}_{11}+C_1\gamma_1,d_2 \mathfrak{f}_{12}+C_1\gamma_2,d_1 \mathfrak{f}_{21}+C_2\gamma_1,d_2 \mathfrak{f}_{22}+C_2\gamma_2 \right), \\
& T_2(\gamma_1,\gamma_2,\mathfrak{f}_{11},\mathfrak{f}_{12},\mathfrak{f}_{21},\mathfrak{f}_{22}) = \left(N_1\mathfrak{f}_{11}+ M_{11},N_1\mathfrak{f}_{12}+ M_{12}, N_2\mathfrak{f}_{21}+ M_{21},N_2\mathfrak{f}_{22}+ M_{22}\right),\\
& T_3(\gamma_1,\gamma_2,\mathfrak{f}_{11},\mathfrak{f}_{12},\mathfrak{f}_{21},\mathfrak{f}_{22}) = \left(C_2 \mathfrak{f}_{11}-C_1 \mathfrak{f}_{21} + B_1 A_1\mathfrak{f}_{21}-B_1 A_2\mathfrak{f}_{12}, C_2 \mathfrak{f}_{12}-C_1 \mathfrak{f}_{22} + B_2 A_1\mathfrak{f}_{21}-B_2 A_2\mathfrak{f}_{12}\right)
\shortintertext{and}
& S(t_1,t_2)  = \left(d_1 t_1, d_2t_2, -C_1t_1+B_1A_1t_1, -C_1t_2+B_2A_2t_1, -C_2t_1+B_1A_1t_2, -C_2t_2+B_2A_2t_2 \right),
\end{align*}
where
$$
C_i\coloneqq \Ide - A_i, \quad N_i\coloneqq \sum_{\ell=0}^{d_i-1} A_i^{\ell}\quad\text{and}\quad M_{ij}\coloneqq B_j \sum_{\ell=1}^{d_i-1} \ell  A_i^{\ell} \mathfrak{f}_{ii} - B_j \gamma_i.
$$
\end{proposition}

\begin{proof} Let $F_1(a_i,a_j,e_k)$, $\mathcal{U}_{ij}$, $\mathcal{V}_{ij}$, $\mathcal{X}_{ij}$ and $\mathcal{Y}_{22}$ as at the beginning of Subsection~\ref{Seccion 4}. By definition $\mathcal{Y}_{22}=0$~and~a~di\-rect computation using that $\cdot$ is trivial, $\times = +$ and \cite{GGV2}*{Definition~2.25}, shows that
$$
\mathcal{U}_{ij}=C_i\gamma_j,\quad \mathcal{V}_{ij}=N_{ij},\quad \mathcal{X}_{ij}= -C_it_j + B_jA_jt_i\quad\text{and}\quad F_1(a_i,a_j,e_k)=\mathfrak{f}_{jk} + A_j\mathfrak{f}_{ik} + B_kA_j \mathfrak{f}_{ij},
$$	
where in $F_1$ we assume that $i\ne j$. Moreover, using that $\cdot$ is trivial, $\times = +$, $\bar{a}_{i\ell}=\ell a_i$ and \cite{GGV2}*{Identities~(2.17) and~(2.18)},~we~ob\-tain that
$$
f_1(a_i,\bar{a}_{i\ell})=\ell \mathfrak{f}_{ii}\quad\text{and}\quad \sum_{\ell=1}^{d_i-1} a_i^{\times(\ell+1)}\blackdiamond \alpha_{\gamma_1,\gamma_2}(a_i,\bar{a}_{i\ell})=\gamma_i,
$$	
which implies that $\mathcal{W}_{ij}=B_j \sum_{\ell=1}^{d_i-1} \ell  A_i^{\ell} \mathfrak{f}_{ii} - B_j \gamma_i= M_{ij}$. The result follows
immediately from these facts.
\end{proof}

\noindent If we additionally assume that $A_1=A_2=\Ide$, then
\begin{align*}
&T_1(\gamma_1,\gamma_2,\mathfrak{f}_{11},\mathfrak{f}_{12},\mathfrak{f}_{21},\mathfrak{f}_{22}) =(d_1 \mathfrak{f}_{11},d_2 \mathfrak{f}_{12},d_1 \mathfrak{f}_{21},d_2 \mathfrak{f}_{22})  \\
& T_2(\gamma_1,\gamma_2,\mathfrak{f}_{11},\mathfrak{f}_{12},\mathfrak{f}_{21},\mathfrak{f}_{22}) =\biggl(d_1 \mathfrak{f}_{11}+ \binom{d_1}{2} B_1 \mathfrak{f}_{11} - B_1 \gamma_1, d_1 \mathfrak{f}_{12}+ \binom{d_1}{2} B_2 \mathfrak{f}_{11} - B_2 \gamma_1,\\
& \phantom{T_2(\gamma_1,\gamma_2,\mathfrak{f}_{11},\mathfrak{f}_{12},\mathfrak{f}_{21},\mathfrak{f}_{22}) =\biggl(} d_2 \mathfrak{f}_{21}+ \binom{d_2}{2} B_1 \mathfrak{f}_{22} - B_1 \gamma_2, d_2 \mathfrak{f}_{22}+ \binom{d_2}{2} B_2 \mathfrak{f}_{22} - B_2 \gamma_2\biggr),\\
&T_3(\gamma_1,\gamma_2,\mathfrak{f}_{11},\mathfrak{f}_{12},\mathfrak{f}_{21},\mathfrak{f}_{22})   =(B_1 \mathfrak{f}_{21}-B_1 \mathfrak{f}_{12},B_2 \mathfrak{f}_{21}-B_2 \mathfrak{f}_{12})
\shortintertext{and}
&S(t_1,t_2)= (d_1t_1,d_2t_2, B_1t_1, B_2t_1, B_1t_2, B_2t_2).
\end{align*}

\smallskip

Since $d_1B_1=d_2B_2=0$ (see Proposition~\ref{caracterizacion primer caso}), the condition $\binom{d_i}{2} B_j=0$ for $i,j=1,2$, required in the following proposition, is satisfied in many cases, for example when $d_2$ is odd and~$d_1 B_2=0$.

\begin{proposition}\label{prop2} If $A_i=\Ide$ and $\binom{d_i}{2} B_j=0$ for $i,j=1,2$, then
\begin{align*}
&T_1(\gamma_1,\gamma_2,\mathfrak{f}_{11},\mathfrak{f}_{12},\mathfrak{f}_{21},\mathfrak{f}_{22}) =(d_1 \mathfrak{f}_{11},d_2 \mathfrak{f}_{12},d_1 \mathfrak{f}_{21},d_2 \mathfrak{f}_{22}),\\
&T_2(\gamma_1,\gamma_2,\mathfrak{f}_{11},\mathfrak{f}_{12},\mathfrak{f}_{21},\mathfrak{f}_{22}) =(d_1 \mathfrak{f}_{11} - B_1 \gamma_1, d_1 \mathfrak{f}_{12}- B_2 \gamma_1, d_2 \mathfrak{f}_{21} - B_1 \gamma_2, d_2 \mathfrak{f}_{22} - B_2 \gamma_2),\\
&T_3(\gamma_1,\gamma_2,\mathfrak{f}_{11},\mathfrak{f}_{12},\mathfrak{f}_{21},\mathfrak{f}_{22}) =(B_1 \mathfrak{f}_{21}-B_1 \mathfrak{f}_{12},B_2 \mathfrak{f}_{21}-B_2 \mathfrak{f}_{12})
\shortintertext{and}
&S(t_1,t_2)= (d_1 t_1,d_2 t_2, B_1t_1, B_2t_1, B_1t_2, B_2t_2).
\end{align*}
\end{proposition}

\begin{proof} This follows immediately from the previous discussion.
\end{proof}

\begin{proposition}\label{d2 divide a Bi} Assume $A_1=A_2=\Ide$ and $\binom{d_1}{2} B_2=\binom{d_2}{2} B_2=0$. If $B_1 = \frac{d_2}{d_1} B_2$ and there exists $\wt{B}_2\in \End(I)$ such that $B_2=d_2 \wt{B}_2$, then
$$
\ho_{\blackdiamond,\Yleft}^2(H^{d_1d_2}_0,I) \simeq \frac{\wt{K}_{12}}{d_1 I} \oplus \frac{K_{12}}{d_2 I} \oplus I_{d_1} \oplus I_{d_1} \oplus (K_{12} \cap I_{d_1}) \oplus I_{d_2}.
$$
where $K_{12}\coloneqq \ker B_1\cap \ker B_2$ and $\wt{K}_{12}\coloneqq \ker B_1\cap \ker \frac{d_2}{d_1}B_2$.
\end{proposition}

\begin{proof} Set $\wt{B}_1\coloneqq \frac{d_2}{d_1} \wt{B}_2$ and define $\phi \colon I^6\to I^6$ by
$$
\phi(\gamma_1,\gamma_2,\mathfrak{f}_{11},\mathfrak{f}_{12},\mathfrak{f}_{21},\mathfrak{f}_{22}) \coloneqq \left(\gamma_1,\gamma_2, \mathfrak{f}_{11}+\frac{d_2}{d_1}\wt{B}_1 \gamma_1, \mathfrak{f}_{12}+\wt{B}_1 \gamma_1, \mathfrak{f}_{21}+\wt{B}_1\gamma_2, \mathfrak{f}_{22}+\wt{B}_2 \gamma_2\right).
$$
It is clear that $\phi$ is invertible and that
$$
\phi^{-1}(\gamma_1,\gamma_2,\mathfrak{f}_{11},\mathfrak{f}_{12},\mathfrak{f}_{21},\mathfrak{f}_{22}) = \left(\gamma_1,\gamma_2, \mathfrak{f}_{11}-\frac{d_2}{d_1}\wt{B}_1 \gamma_1,\mathfrak{f}_{12}-\wt{B}_1 \gamma_1,\mathfrak{f}_{21}-\wt{B}_1\gamma_2, \mathfrak{f}_{22}-\wt{B}_2 \gamma_2\right).
$$
Set $\wt{T}_i\coloneqq T_i\xcirc \phi$ and $\wt{S}\coloneqq \phi^{-1}\xcirc S$, where the $T_i$'s and $S$ are as in Proposition~\ref{prop2}. By~\eqref{formula00}, we have
\begin{equation}\label{homo10}
\ho_{\blackdiamond}^2(H^{d_1d_2}_0,I) \simeq \frac{\ker T_1\cap\ker T_2\cap\ker T_3 }{\ima S} \simeq \frac{\ker \wt{T}_1\cap\ker \wt{T}_2\cap\ker \wt{T}_3 }{\ima \wt{S}},
\end{equation}
where the last isomorphism is induced by $\phi^{-1}$. A direct computation shows that
\begin{align*}
&\wt{T}_1(\gamma_1,\gamma_2,\mathfrak{f}_{11},\mathfrak{f}_{12},\mathfrak{f}_{21},\mathfrak{f}_{22}) = \left(d_1 \mathfrak{f}_{11}+B_1\gamma_1, d_2 \mathfrak{f}_{12}+B_1\gamma_1,d_1\mathfrak{f}_{21}+B_2\gamma_2,d_2\mathfrak{f}_{22}+ B_2\gamma_2\right),\\
&\wt{T}_2(\gamma_1,\gamma_2,\mathfrak{f}_{11},\mathfrak{f}_{12},\mathfrak{f}_{21},\mathfrak{f}_{22}) = (d_1 \mathfrak{f}_{11},d_1 \mathfrak{f}_{12}, d_2\mathfrak{f}_{21}, d_2 \mathfrak{f}_{22}),\\
&\wt{T}_3(\gamma_1,\gamma_2,\mathfrak{f}_{11},\mathfrak{f}_{12},\mathfrak{f}_{21},\mathfrak{f}_{22}) = \bigl(B_1\mathfrak{f}_{21}-B_1\mathfrak{f}_{12}+B_1\wt{B}_1\gamma_2-B_1\wt{B}_1\gamma_1,B_2\mathfrak{f}_{21}-B_2\mathfrak{f}_{12} +B_2\wt{B}_1\gamma_2-B_2\wt{B}_1\gamma_1\bigr)
\shortintertext{and}
&\wt{S}(t_1,t_2)= (d_1t_1,d_2t_2,0,0,0,0).
\end{align*}
We next compute~\eqref{homo10}. Since $d_1\mid d_2$, $B_1=d_2 \wt{B}_1$ and $B_2=d_2 \wt{B}_2$, if $(\gamma_1,\gamma_2,\mathfrak{f}_{11},\mathfrak{f}_{12},\mathfrak{f}_{21}, \mathfrak{f}_{22})\in \ker \wt{T}_2$, then
\begin{align*}
& \wt{T}_1(\gamma_1,\gamma_2,\mathfrak{f}_{11},\mathfrak{f}_{12},\mathfrak{f}_{21},\mathfrak{f}_{22}) = \left(B_1\gamma_1,B_1\gamma_1,d_1\mathfrak{f}_{21} + B_2\gamma_2,B_2\gamma_2\right)
\shortintertext{and}
& \wt{T}_3(\gamma_1,\gamma_2,\mathfrak{f}_{11},\mathfrak{f}_{12},\mathfrak{f}_{21},\mathfrak{f}_{22}) = \bigl(B_1\wt{B}_1\gamma_2-B_1\wt{B}_1\gamma_1, B_2\wt{B}_1\gamma_2-B_2\wt{B}_1\gamma_1\bigr).
\end{align*}
Since
$$
B_1\wt{B}_1 = \wt{B}_1B_1 = \frac{d_2}{d_1}\wt{B}_1B_2\quad\text{and}\quad B_2\wt{B}_1 = \wt{B}_1B_2 = \wt{B}_2B_1,
$$
this implies that if $(\gamma_1,\gamma_2,\mathfrak{f}_{11},\mathfrak{f}_{12}, \mathfrak{f}_{21},\mathfrak{f}_{22})\in \ker \wt{T}_1\cap \ker \wt{T}_2$, then
$\wt{T}_3(\gamma_1,\gamma_2,\mathfrak{f}_{11},\mathfrak{f}_{12},\mathfrak{f}_{21},\mathfrak{f}_{22}) = 0$. Hence
$$
\ho_{\blackdiamond,\Yleft}^2(H^{d_1d_2}_0,I)\simeq \frac{\ker \wt{T}_1\cap\ker \wt{T}_2\cap\ker \wt{T}_3 }{\ima \wt{S}} = \frac{\wt{K}_{12}}{d_1 I} \oplus \frac{K_{12}}{d_2 I} \oplus I_{d_1} \oplus I_{d_1} \oplus (K_{12} \cap I_{d_1}) \oplus I_{d_2},
$$
as desired.
\end{proof}

\begin{remark}\label{remark3} Assume that we are under the hypothesis of Proposition~\ref{d2 divide a Bi}. The map $\phi\colon I^6\to I^6$, induces an~iso\-mor\-phism
$$
\frac{\wt{K}_{12}}{d_1 I} \oplus \frac{K_{12}}{d_2 I} \oplus I_{d_1} \oplus I_{d_1} \oplus (K_{12} \cap I_{d_1}) \oplus I_{d_2} \longrightarrow \frac{\ker T_1\cap \ker T_2\cap \ker T_3}{\ima S},
$$
where $T_1$, $T_2$, $T_3$ and $S$ are as in Proposition~\ref{caso particular 1}. Let $Z\subseteq \wt{K}_{12} \oplus K_{12} \oplus I_{d_1} \oplus I_{d_1} \oplus (K_{12} \cap I_{d_1}) \oplus I_{d_2}$ be such that the natural map~$Z\to \frac{\wt{K}_{12}}{d_1 I} \oplus \frac{K_{12}}{d_2 I} \oplus I_{d_1} \oplus I_{d_1} \oplus (K_{12} \cap I_{d_1}) \oplus I_{d_2}$ is a bi\-jec\-tion and let $T$ be as in~\eqref{formula0}. Notice that there is an obvious choice for $Z$ and that
$$
T(\gamma_1,\gamma_2,\mathfrak{f}_{11},\mathfrak{f}_{12},\mathfrak{f}_{21},\mathfrak{f}_{22}) = (\alpha_{\gamma_1,\gamma_2}, f_{\mathfrak{f}_{11}, \mathfrak{f}_{12},\mathfrak{f}_{21},\mathfrak{f}_{22},\gamma_1, \gamma_2}).
$$
Clearly $T(\phi(Z))$ is a complete family of representatives of~$2$-cocycles modulo coboundaries in $\wh{\mathcal{C}}_N^*(H^{d_1d_2}_{\lambda},I)$. A~sim\-ilar observation is pertinent below Proposition~\ref{B2 divide a d1}. We leave the task of formulating it to the reader. From the preceding discussion, it becomes clear that, in the cases considered in Propositions~\ref{d2 divide a Bi} and~\ref{B2 divide a d1}, computing $T$ is indispensable for constructing the extensions. Since $\alpha_{\gamma_1,\gamma_2}$ is already known from~\eqref{alpha_{gamma_1,gamma_2}}, the determination of $T$ reduces to establishing the following remark.
\end{remark}

\begin{remark} For $(\gamma_1,\gamma_2,\mathfrak{f}_{11}, \mathfrak{f}_{12},\mathfrak{f}_{21}, \mathfrak{f}_{22})\in \ker T_1\cap \ker T_2\cap \ker T_3$, let $f\coloneqq f_{\mathfrak{f}_{11}, \mathfrak{f}_{12},\mathfrak{f}_{21},\mathfrak{f}_{22},\gamma_1, \gamma_2}$. By \cite{GGV}*{Theorem~3.4},  we have
$$
f(k_1a_1+k_2a_2,\lambda_1e_1+\lambda_2e_2) = \sum_{j=1}^2 \sum_{i=1}^2 k_j\lambda_i \mathfrak{f}_{ji}+ k_1k_2(\lambda_1B_1+\lambda_2B_2)\mathfrak{f}_{12} + \sum_{j=1}^2 \binom{k_j}{2}(\lambda_1B_1+\lambda_2B_2)\mathfrak{f}_{jj},
$$
where $0\le k_1,\lambda_1 <d_1$ and $0\le k_2,\lambda_2 <d_2$,
\end{remark}

\begin{example} Let $d_1$ and $d_2$ be natural numbers, let $K$ be an abelian group and let $I\coloneqq K\oplus K\oplus K$. Assume that~$d_1\mid d_2$. If $d_2^2K=\binom{d_1}{2}d_2K =\binom{d_2}{2}d_2K=0$, then the maps $B_1\colon I\to I$ and $B_2\colon I\to I$, defined by the~ mul\-ti\-pli\-ca\-tion by the matrices
$$
\begin{pmatrix}
0 & \frac{d_2^2}{d_1} & 0\\
0 & 0 & \frac{d_2^2}{d_1}\\
0 & 0 & 0	
\end{pmatrix}
\quad\text{and}\quad
\begin{pmatrix}
0 & d_2 & 0\\
0 & 0 & d_2\\
0 & 0 & 0	
\end{pmatrix},
$$
respectively, satisfy the hypotheses of Proposition~\ref{d2 divide a Bi}.
\end{example}

\begin{proposition}\label{B2 divide a d1} Assume that $A_1=A_2=\Ide$ and $\binom{d_1}{2} B_2=\binom{d_2}{2} B_2=0$. Assume furthermore that there exist morphisms $J_1,J_2\colon I\to I$ such that $J_1 B_2=B_1$,  $J_2 B_2=B_2 J_2=d_1 \Ide$ and $B_2J_1(I_{d_2})=0$. Then
$$
\ho_{\blackdiamond,\Yleft}^2(H^{d_1d_2}_0,I)\simeq  \ker B_2\oplus \ker B_2 \oplus I_{d_1} \oplus  \frac{\ker B_2}{\ima B_2} \oplus I_{d_1}  \oplus \frac{I_{d_2}}{\ima B_2}.
$$
\end{proposition}

\begin{proof} Let $\phi\colon I^6\to I^6$ be the morphism given by
$$
\phi(\gamma_1,\gamma_2,\mathfrak{f}_{11},\mathfrak{f}_{12},\mathfrak{f}_{21},\mathfrak{f}_{22}) \coloneqq \left(\gamma_1+J_2 \mathfrak{f}_{12}+J_2 \mathfrak{f}_{21},\gamma_2+\frac{d_2}{d_1}J_2 \mathfrak{f}_{22}, \mathfrak{f}_{11}+J_1 \mathfrak{f}_{12}+J_1 \mathfrak{f}_{21}, \mathfrak{f}_{12}+\mathfrak{f}_{21},\mathfrak{f}_{21}+J_1\mathfrak{f}_{22},\mathfrak{f}_{22}\right).
$$
It is clear that $\phi$ is invertible and that
$$
\phi^{-1}(\gamma_1,\gamma_2,\mathfrak{f}_{11},\mathfrak{f}_{12},\mathfrak{f}_{21},\mathfrak{f}_{22}) \coloneqq \left(\gamma_1-J_2 \mathfrak{f}_{12}, \gamma_2-\frac{d_2}{d_1}J_2 \mathfrak{f}_{22},\mathfrak{f}_{11}-J_1 \mathfrak{f}_{12},\mathfrak{f}_{12}-\mathfrak{f}_{21} +J_1\mathfrak{f}_{22}, \mathfrak{f}_{21}-J_1\mathfrak{f}_{22},\mathfrak{f}_{22} \right).
$$
Set $\wt{T}_i=T_i \phi$ and $\wt{S}= \phi^{-1} S$, where the $T_i$'s and $S$ are as in Proposition~\ref{prop2}. By~\eqref{formula00}, we have
\begin{equation*}
\ho_{\blackdiamond}^2(H^{d_1d_2}_0,I) \simeq \frac{\ker T_1\cap\ker T_2\cap\ker T_3 }{\ima S} \simeq \frac{\ker \wt{T}_1\cap\ker \wt{T}_2\cap\ker \wt{T}_3 }{\ima \wt{S}},
\end{equation*}
where the last isomorphism is induced by $\phi^{-1}$. A direct computation using that
$$
d_1J_1 = J_1B_2J_2 = B_1J_2,\quad B_2J_2 = d_1\Ide,\quad J_2B_2= d_1\Ide\quad\text{and}\quad J_1B_2 = B_1,
$$
shows that
\begin{align*}
&\wt{T}_1(\gamma_1,\gamma_2,\mathfrak{f}_{11},\mathfrak{f}_{12},\mathfrak{f}_{21},\mathfrak{f}_{22}) =(d_1 \mathfrak{f}_{11}+d_1 J_1 \mathfrak{f}_{21}+d_1 J_1 \mathfrak{f}_{12},d_2 \mathfrak{f}_{12}+d_2 \mathfrak{f}_{21},d_1 \mathfrak{f}_{21}+d_1J_1\mathfrak{f}_{22},d_2 \mathfrak{f}_{22}),  \\
&\wt{T}_2(\gamma_1,\gamma_2,\mathfrak{f}_{11},\mathfrak{f}_{12},\mathfrak{f}_{21},\mathfrak{f}_{22}) =(d_1 \mathfrak{f}_{11}-B_1 \gamma_1,-B_2 \gamma_1, d_2 \mathfrak{f}_{21}-B_1 \gamma_2,-B_2 \gamma_2),  \\
&\wt{T}_3(\gamma_1,\gamma_2,\mathfrak{f}_{11},\mathfrak{f}_{12},\mathfrak{f}_{21},\mathfrak{f}_{22})=(B_1J_1\mathfrak{f}_{22}-B_1 \mathfrak{f}_{12},B_2J_1\mathfrak{f}_{22}-B_2 \mathfrak{f}_{12})
\shortintertext{and}
&\wt{S}(t_1,t_2)= (0,0,0,B_2 t_1,0,B_2 t_2).
\end{align*}
We claim that
\begin{equation}\label{ida2}
\ker \wt{T}_1\cap \ker \wt{T}_2\cap \ker \wt{T}_3 = \ker B_2\oplus \ker B_2 \oplus I_{d_1} \oplus \ker B_2 \oplus I_{d_1}  \oplus I_{d_2}.
\end{equation}
Indeed, since $B_1 = J_1B_2$ we know that $(\gamma_1,\gamma_2,\mathfrak{f}_{11},\mathfrak{f}_{12},\mathfrak{f}_{21},\mathfrak{f}_{22})\in \ker \wt{T}_2$ if and only if
$$
B_2 \gamma_2 = B_2 \gamma_1 = d_1 \mathfrak{f}_{11} = d_2\mathfrak{f}_{21} = 0.
$$
If, moreover $(\gamma_1,\gamma_2,\mathfrak{f}_{11},\mathfrak{f}_{12},\mathfrak{f}_{21},\mathfrak{f}_{22})\in \ker \wt{T}_1$, then we have
$$
d_2 \mathfrak{f}_{22} = d_2 \mathfrak{f}_{12} = 0,
$$
which implies that $B_2J_1\mathfrak{f}_{12} = B_2J_1\mathfrak{f}_{22} = 0$, because $B_2J_1(I_{d_2})=0$. Since, we also have $d_1\Ide = J_2B_2$ this implies that $d_1J_1\mathfrak{f}_{12} = d_1J_1\mathfrak{f}_{22} = 0$. Hence
$$
d_1 \mathfrak{f}_{21} = 0.
$$
Since $d_1\mid d_2$, from the previous discussion we obtain that $(\gamma_1,\gamma_2,\mathfrak{f}_{11},\mathfrak{f}_{12},\mathfrak{f}_{21},
\mathfrak{f}_{22}) \in \ker \wt{T}_1\cap \ker \wt{T}_2$ if and only if
$$
B_2 \gamma_2 = B_2 \gamma_1 = d_1 \mathfrak{f}_{11} = d_2 \mathfrak{f}_{22} = d_2 \mathfrak{f}_{12} = d_1 \mathfrak{f}_{21} = 0.
$$
Since this implies that $B_2J_1\mathfrak{f}_{22} = 0$, and moreover $B_1=J_2B_2$, we obtain that equality~\eqref{ida2} holds. By the formula of $\wt{S}$, this finishes the proof of the proposition.
\end{proof}

\begin{example} Let $d$, $d'$, $d''$ and $d'''$ be odd natural numbers, let $I$ be an abelian group, and let $d_1\coloneqq dd'$, $d_2\coloneqq d_1d'''$, $B_1\coloneqq dd''\Ide_I$ and $B_2\coloneqq d\Ide_I$. If $d'd'''\mid d''$ and $d^2d'I=0$, then $B_1$ and $B_2$ satisfy the hipotheses of Proposition~\ref{B2 divide a d1}.
\end{example}

\subsection[Some examples when \texorpdfstring{$d_1\nmid \lambda\binom{d_2}{2}$}{d1 no divides lambda d2 (d2-1)/2}]{Some examples when \texorpdfstring{$\pmb{d_1\nmid \lambda\binom{d_2}{2}}$}{d1 no divides lambda d2 (d2-1)/2}}

Let $H^{d_1d_2}_{\lambda}$ be as in Section~\ref{section 2} and let $I$ be an abelian group. In this section we assume that $d_1\nmid \lambda\binom{d_2}{2}$ and~we~com\-pute~$\ho_{\blackdiamond,\Yleft}^2(H^{d_1d_2}_{\lambda},I)$, under suitable hypotheses. By Remark~\ref{lambda divide a d1}, we can suppose that $\lambda \mid d_1$, and we do it. As in Section~\ref{Seccion 4},
$$
\ho_{\blackdiamond,\Yleft}^2(H^{d_1d_2}_{\lambda},I) \simeq \frac{\ker T_1\cap \ker T_2\cap \ker T_3}{\ima S},
$$
where $T_1,T_2,T_3$ and $S$ are the maps introduced in~\cite{GGV2}*{Subsubsection~2.2.3}.

\subsubsection[Case \texorpdfstring{$\lambda = 1$}{lambda=1} \texorpdfstring{$d_1=2$}{d1=2} and \texorpdfstring{$v_2(d_2) = 1$}{v2(d2) = 1}]{Case \texorpdfstring{$\pmb{\lambda = 1}$}{lambda=1}, \texorpdfstring{$\pmb{d_1=2}$}{d1=2} and \texorpdfstring{$\pmb{v_2(d_2) = 1}$}{v2(d2) = 1}}
Assume that $\lambda=1$, $d_1=2$ and $d_2=2m$ for some odd $m\in \mathds{N}$. Let $n$ and $s$ be as at the beginning of this section. By Remark~\ref{caso v_2=1}, we have $s=1$, $n=2$, $r_1=2d_2$, and we can take $e_1\coloneqq (1,0)$ and $a_1 = e_2\coloneqq (0,1)$. So, $T_3$ is the zero map and, by~\eqref{def de punto},
\begin{equation}\label{1a_j cdot e_i'}
a_1\cdot e_i = \begin{cases} e_1 & \text{if $i=1$,}\\ e_1+e_2 & \text{if $i=2$.}\end{cases}
\end{equation}
Hence, if we write $a_1 = \lambda_{11}e_1 + \lambda_{21}e_2$ and $a_1\cdot e_i=\Gamma^1_{i1}e_1 + \Gamma^2_{i1}e_2$ with $0\le \lambda_{i1}< d_i$ and $0\le \Gamma^k_{i1}<d_k$, then
$$
\lambda_{11} = 0,\quad i-1 \lambda_{21} = 1, \quad \Gamma_{11}^1 = \Gamma_{21}^1 = \Gamma_{21}^2 = 1 \quad\text{and}\quad \Gamma_{11}^2 = 0.
$$
A straightforward computation, using these facts, shows that
$$
T_1,T_2\colon I^4\to I^2\quad\text{and}\quad S\colon I^2\to I^4,
$$
are given by
\begin{align}
& T_1(\gamma_1,\gamma_2, \mathfrak{f}_{11}, \mathfrak{f}_{12})= \left(d_1 \mathfrak{f}_{11} + \mathcal{U}_{11}, d_2 \mathfrak{f}_{12} + \mathcal{U}_{12}\right)\label{T_1'}\\
& T_2(\gamma_1,\gamma_2, \mathfrak{f}_{11}, \mathfrak{f}_{12})= \left(\mathcal{V}_{11} + \mathcal{W}_{11}, \mathcal{V}_{12} + \mathcal{W}_{12}\right)\label{T_2'}
\shortintertext{and}
& S(t_1,t_2)=\left(d_1t_1,d_2t_2,\mathcal{X}_{11},\mathcal{X}_{12}\right)\label{S'},
\end{align}
in which
\begin{itemize}
		
\item[-] $\mathcal{U}_{1j}\coloneqq \sum_{k=1}^{d_j-1}\alpha_{\gamma_1,\gamma_2}(a_1\cdot e_j,k (a_1\cdot e_j))-a_1\blackdiamond \gamma_j$,
		
\item[-] $\mathcal{V}_{1j}\coloneqq \sum_{\ell=0}^{r_1-1} a_1^{\times (r_1-1-\ell)}\blackdiamond f_1(a_1,a_1^{\times \ell}\cdot e_j)$,
		
\item[-] $\mathcal{W}_{1j}\coloneqq \sum_{\ell=1}^{r_1-1} a_1^{\times \ell}\blackdiamond f_1(a_1,\bar{a}_{1\ell})\Yleft e_j -\sum_{\ell=1}^{r_1-1} a_1^{\times(\ell+1)}\blackdiamond \alpha_{\gamma_1,\gamma_2}(a_1,\bar{a}_{1\ell})\Yleft e_j$, where $\bar{a}_{1\ell}\coloneqq a_1^{\times(\ell+1)}-a_1$, 	

\item[-] $\mathcal{X}_{1j}\coloneqq  - t_1 -\delta_{j2} t_2 + a_1\blackdiamond t_j + (e_j\blackdiamond t_2) \Yleft a_1\cdot e_j$, where $\delta_{j2}$ is the delta of Kronecker,
		
\item[-] $f_1 \coloneqq f_1^{\mathfrak{f}_{11},\mathfrak{f}_{12}}$.		

\end{itemize}
By Proposition~\ref{caracterizacion segundo caso}, the actions $\blackdiamond$ and $\Yleft$ are determined by maps $A,B_1, B_2 \in \End(I)$ satisfying conditions~\eqref{primeras condiciones A, B1 y B2 segundo caso} and~\eqref{segundas condiciones A, B1 y B2 segundo caso}. Next, we assume that $A=\Ide$. Under this hypothesis,~\eqref{primeras condiciones A, B1 y B2 segundo caso} and \eqref{segundas condiciones A, B1 y B2 segundo caso} reduce to
$$
2B_1=2mB_2=0,\quad B_1=-B_2^2 \quad \text{and} \quad B_2^3=0.
$$

\begin{proposition}\label{maps Ti caso d1=2 d202m con m impar} Under the assumptions above, the maps $T_1$, $T_2$ and $S$ are given by
\begin{align*}
& T_1(\gamma_1,\gamma_2, \mathfrak{f}_{11}, \mathfrak{f}_{12})= \left(2 \mathfrak{f}_{11}, 2m \mathfrak{f}_{12}+m \gamma_1\right)\\
& T_2(\gamma_1,\gamma_2, \mathfrak{f}_{11}, \mathfrak{f}_{12})= \left(4m \mathfrak{f}_{11}+mB_1 \gamma_1, 2m \mathfrak{f}_{11}+4m\mathfrak{f}_{12}+2m \gamma_1+mB_2\gamma_1-2B_2 \gamma_2\right)
\shortintertext{and}
& S(t_1,t_2)=\left(2 t_1,2m t_2, B_1 t_2, -t_1+B_1 t_2+B_2 t_2\right).
\end{align*}
\end{proposition}

\begin{proof} We first compute $T_1(\gamma_1,\gamma_2, \mathfrak{f}_{11}, \mathfrak{f}_{12})$. By~\eqref{1a_j cdot e_i'} and~\eqref{alpha_{gamma_1,gamma_2}}, we have $\mathcal{U}_{11}=0$ and
$$
\mathcal{U}_{12} = \sum_{k=1}^{d_2-1}\alpha_1(\gamma_1)(1,k)+\sum_{k=1}^{d_2-1}\alpha_2(\gamma_2)(1,k) - \gamma_2= m\gamma_1.
$$
The formula for $T_1(\gamma_1,\gamma_2, \mathfrak{f}_{11}, \mathfrak{f}_{12})$ follows from~\eqref{T_1'} and these facts. We  next compute $T_2(\gamma_1,\gamma_2,\mathfrak{f}_{11},\mathfrak{f}_{12})$. For~this we will need to calculate a general formula for $f_1(a_1,x)$. Let $\lambda_1,\lambda_2\in \mathbb{Z}$. Since, by~\eqref{1a_j cdot e_i'} and~\eqref{alpha_{gamma_1,gamma_2}},
$$
\sum_{\ell =1}^{\mathrm{r}_2(\lambda_1)-1} \alpha_{\gamma_1,\gamma_2} (a_1\cdot e_1,\ell (a_1\cdot e_1)) = 0\quad\text{and}\quad \sum_{\ell=1}^{\mathrm{r}_{d_2}(\lambda_2)-1} \alpha_{\gamma_1,\gamma_2}(a_1\cdot e_2,\ell (a_1\cdot e_2)) = \left\lfloor \frac{\mathrm{r}_{d_2}(\lambda_2)}{2}\right\rfloor \gamma_1,
$$
by~\cite{GGV2}*{Identities~(2.17) and~(2.18)}, we have
\begin{equation}\label{formula general para f}
\begin{aligned}
f_1(a_1,(\lambda_1,\lambda_2))&= \mathrm{r}_2(\lambda_1)\mathfrak{f}_{11}+\mathrm{r}_{d_2}(\lambda_2)\mathfrak{f}_{12}+\left\lfloor \frac{\mathrm{r}_{d_2}(\lambda_2)}{2}\right\rfloor \gamma_1+ \alpha_{\gamma_1,\gamma_2}((\lambda_1,0),(\lambda_2,\lambda_2))\\
&=\mathrm{r}_2(\lambda_1) \mathfrak{f}_{11} +\mathrm{r}_{d_2}(\lambda_2)\mathfrak{f}_{12}+\left\lfloor \frac{\mathrm{r}_{d_2}(\lambda_2)}{2}\right\rfloor \gamma_1+\mathrm{r}_2(\lambda_1)\mathrm{r}_2(\lambda_2) \gamma_1
\end{aligned}
\end{equation}
Since, by~\eqref{potencia} and~\eqref{def de punto}, we have $a_1^{\times \ell}\cdot e_1 = e_1$ and $a_1^{\times \ell}\cdot e_2 = e_2-\ell e_1$, from this it follows immediately that
\begin{align}
&\mathcal{V}_{11}=\sum_{\ell=0}^{4m-1} f_1(a_1,a_1^{\times \ell}\cdot e_1) = 4m \mathfrak{f}_{11}\label{primer sumando j 1}
\shortintertext{and}
&\mathcal{V}_{12}= \sum_{\ell=0}^{4m-1} f_1(a_1,a_1^{\times \ell}\cdot e_2)=\sum_{\ell=0}^{4m-1} f_1(a_1,-\ell e_1+e_2) =2m \mathfrak{f}_{11}+4m \mathfrak{f}_{12}+2m \gamma_1.\label{primer sumando j 2}
\end{align}
Our next aim is to determine $\mathcal{W}_{1j}$. For this we first compute $\sum_{\ell=1}^{4m-1} f_1(a_1,\bar{a}_{1\ell})\Yleft e_j$. Since $\bar{a}_{1\ell} = \binom{\ell+1}{2}e_1+\ell e_2$, by equality~\eqref{formula general para f}, we have
\begin{equation*}
\sum_{\ell=1}^{4m-1} f_1(a_1,\bar{a}_{1\ell})\Yleft e_j = \sum_{\ell=1}^{4m-1}\left(\mathrm{r}_2\left(\binom{\ell+1}{2}\right) B_j \mathfrak{f}_{11}+\mathrm{r}_{d_2}(\ell) B_j \mathfrak{f}_{12}+\left\lfloor\frac{\mathrm{r}_{d_2}(\ell)}{2}\right\rfloor B_j \gamma_1+\mathrm{r}_2\left(\binom{\ell+1}{2}\right)\mathrm{r}_2(\ell)B_j \gamma_1\right).
\end{equation*}
We claim that the sum of the first three terms vanish. To check this we note that
$$
\mathrm{r}_2\left(\binom{\ell+1}{2}\right)=\mathrm{r}_2\left(\frac{\ell(\ell+1)}{2}\right)=
\begin{cases}
0& \text{if }\ell \equiv 0 \mod 4,\\
1& \text{if }\ell \equiv 1 \mod 4,\\
1& \text{if }\ell \equiv 2 \mod 4,\\
0& \text{if }\ell \equiv 3 \mod 4,
\end{cases}
$$
from which we obtain
$$
\sum_{\ell=1}^{4m-1} \mathrm{r}_2\left(\binom{\ell+1}{2}\right) B_j \mathfrak{f}_{11}=2m B_j \mathfrak{f}_{11}=0,
$$
since $2B_1=0$ and $2mB_2=0$. For the second term we also obtain
$$
\sum_{\ell=1}^{4m-1} \mathrm{r}_{d_2}(\ell)B_j  \mathfrak{f}_{12} = 2\sum_{\ell=1}^{2m-1} \ell B_j \mathfrak{f}_{12}=2\binom{2m}{2} B_j \mathfrak{f}_{12}=2m(2m-1) B_j \mathfrak{f}_{12}=0.
$$
The third term yields
$$
\sum_{\ell=1}^{4m-1} \left\lfloor\frac{\mathrm{r}_{d_2}(\ell)}{2}\right\rfloor B_j\gamma_1 = 2 \sum_{\ell=1}^{2m-1} \left\lfloor \frac{\ell}{2}\right\rfloor B_j \gamma_1=4\sum_{k=1}^{m-1} k B_j \gamma_1=4\binom{m}{2} B_j \gamma_1=2m(m-1)B_j \gamma_1=0,
$$
which finishes the proof of the claim. Since
$$
\mathrm{r}_2\left(\binom{\ell+1}{2}\right)\mathrm{r}_2(\ell)=\begin{cases}
0& \text{if }\ell \equiv 0 \mod 4,\\
1& \text{if }\ell \equiv 1 \mod 4,\\
0& \text{if }\ell \equiv 2 \mod 4,\\
0& \text{if }\ell \equiv 3 \mod 4,
\end{cases}
$$
we obtain
\begin{equation}\label{segundo sumando en T2}
\sum_{\ell=1}^{4m-1} f_1(a_1,\bar{a}_{1\ell})\Yleft e_j = \sum_{\ell=1}^{4m-1} \mathrm{r}_2\left(\binom{\ell+1}{2}\right)\mathrm{r}_2(\ell) B_j\gamma_1 = mB_j \gamma_1.
\end{equation}
Adding \eqref{segundo sumando en T2} to
$$
- \sum_{\ell=1}^{4m-1} \alpha_{\gamma_1,\gamma_2}(a_1,\bar{a}_{1\ell})\Yleft e_j= - B_j \sum_{\ell=1}^{4m-1} \alpha_{\gamma_1,\gamma_2} \left(e_2,\binom{\ell+1}{2}e_1+\ell e_2\right) = - B_j \sum_{\ell=1}^{4m-1} \alpha_1(\gamma_2)\left(1,\ell\right) = -2B_j \gamma_2,
$$
we obtain
\begin{equation}\label{W1j}
\mathcal{W}_{1j}=mB_j \gamma_1 - 2B_j \gamma_2.
\end{equation}
The formula for $T_2(\gamma_1,\gamma_2, \mathfrak{f}_{11}, \mathfrak{f}_{12})$ follows immediately from~\eqref{T_2'}, \eqref{primer sumando j 1}, \eqref{primer sumando j 2} and~\eqref{W1j} (recall that $2B_1=0$). Finally, the formula for $S(t_1,t_2)$ is obtained straightforwardly from~\eqref{S'}.
\end{proof}
	
\begin{proposition}\label{ultimo calculo} Under the hypotheses of Proposition~\ref{maps Ti caso d1=2 d202m con m impar},
\begin{equation}\label{cohomologia caso d1=2, d2=2m caso general}
\ho_{\blackdiamond,\Yleft}^2(H^{2,d_2}_1,I)\simeq \frac{I_m \oplus \ker(2B_2) \oplus I_2}{\ima \wt{S}_1},
\end{equation}
where $\wt{S}_1\colon I\to I^3$ is the map defined by $\wt{S}_1(t)\coloneqq (2B_2t,2mt,B_1t)$.
\end{proposition}

\begin{proof} Define the invertible map $\phi\colon I^4\to I^4$ by
$$
\phi(\gamma_1,\gamma_2,\mathfrak{f}_{11},\mathfrak{f}_{12}) \coloneqq (\gamma_1-2\mathfrak{f}_{12},\gamma_2,\mathfrak{f}_{11},\mathfrak{f}_{12}),
$$
and set $\wt{T}_i=T_i\xcirc \phi$ and $\wt{S}= \phi^{-1}\xcirc S$, where the $T_i$'s and $S$ are as in Proposition~\ref{maps Ti caso d1=2 d202m con m impar}. By~\eqref{formula00}, we have
\begin{equation*}
\ho_{\blackdiamond,\Yleft}^2(H^{2,d_2}_1,I)\simeq \frac{\ker T_1\cap\ker T_2}{\ima S} \simeq \frac{\ker \wt{T}_1\cap\ker \wt{T}_2}{\ima \wt{S}},
\end{equation*}
where the last isomorphism is induced by $\phi^{-1}$. A straightforward computation, using that $2B_1 = 2mB_2 =0$, shows that
\begin{align*}
&\wt{T}_1(\gamma_1,\gamma_2,\mathfrak{f}_{11},\mathfrak{f}_{12})  =(2 \mathfrak{f}_{11}, m\gamma_1),  \\
&\wt{T}_2((\gamma_1,\gamma_2,\mathfrak{f}_{11},\mathfrak{f}_{12}))  =(4m  \mathfrak{f}_{11}+mB_1\gamma_1, 2m\mathfrak{f}_{11}+2m \gamma_1+mB_2 \gamma_1- 2B_2 \gamma_2)
\shortintertext{and}
&\wt{S}(t_1,t_2)= (2B_2 t_2, 2m t_2, B_1 t_2,-t_1+ B_1 t_2+B_2 t_2).
\end{align*}	
Thus
$$
\ker \wt{T}_1\cap\ker \wt{T}_2 = I_m \oplus \ker(2B_2) \oplus I_2\oplus I.
$$
Since $\ima \wt{S}=\ima \wt{S}_1\oplus I$, this gives~\eqref{cohomologia caso d1=2, d2=2m caso general}.
\end{proof}

\begin{remark} If there exists $u$ odd such that $uI = 0$, then $B_1 = 0$, because $B_1 = \gcd(2,u) B_1 = 0$.
\end{remark}

\begin{example} Let $I\coloneqq \mathds{Z}_{2m^2}$ and let $B_2$ be the multiplication by $m$.
Then
$$
B_1= - B_2^2,\quad I_m=2m I, \quad \ker(2B_2)= mI, \quad I_2=m^2 I\quad\text{and}\quad \wt{S}_1(t)=(2mt,2mt,-m^2t).
$$
Thus,
\begin{equation}\label{cohomologia parcial}
\ho_{\blackdiamond,\Yleft}^2(H^{2,d_2}_1,I)\simeq \frac{I_m \oplus \ker(2B_2) \oplus I_2}{\ima \wt{S}_1} = \frac{2mI\oplus mI\oplus m^2I}{\ima \wt{S}_1}.
\end{equation}
We claim that the map
$$
\theta\colon \mathbb{Z}_{2m}\to \frac{2mI\oplus mI\oplus m^2I}{\ima \wt{S}_1},
$$
given by $\theta(x)\coloneqq \ov{(0,mx,0)}$, where $\ov{(0,mx,0)}$ denotes the class of $(0,mx,0)$ in $\frac{2mI\oplus mI\oplus m^2I}{\ima \wt{S}_1}$, is an isomorphism. Indeed, since $|\ker \wt{S}_1| = |2mI| = m$ and $|2mI\oplus mI\oplus m^2I| = 4m^2$, we have
$$
\left|\frac{2mI\oplus mI\oplus m^2I}{\ima \wt{S}_1}\right| = 2m.
$$
Hence to finish the proof of the claim we only must prove that $\theta$ is injective, which can be easily checked.
\end{example}

\begin{remark}\label{remark1u} Assume that we are under the hypothesis of Proposition~\ref{ultimo calculo}. The map $\psi\colon I^3\to I^4$, defined by
$$
\psi (\gamma_1,\gamma_2,\mathfrak{f}_{11})\coloneqq \phi (\gamma_1,\gamma_2,\mathfrak{f}_{11},0) =(\gamma_1-2\mathfrak{f}_{12},\gamma_2,\mathfrak{f}_{11},0),
$$
where $\phi$ is as in the proof of Proposition~\ref{ultimo calculo}, induces an isomorphism
$$
\ov{\psi}\colon \frac{I_m \oplus \ker(2B_2) \oplus  I_2}{\ima \wt{S}_1}\longrightarrow \frac{\ker T_1\cap \ker T_2}{\ima S}.
$$
Let $Z\subseteq I_m \oplus \ker(2B_2) \oplus I_2$ be such that the canonical map $Z\to \frac{I_m \oplus \ker(2B_2) \oplus I_2}{\ima \wt{S}_1}$ is a bijection and let~$T$ be as in equality~\eqref{formula0}, so that
$$
T(\gamma_1,\gamma_2,\mathfrak{f}_{11},\mathfrak{f}_{12}) = (\alpha_{\gamma_1,\gamma_2}, f_{\mathfrak{f}_{11},\mathfrak{f}_{12},\gamma_1,\gamma_2}).
$$
Clearly $T(\psi(Z))$ is a complete family of representatives of~$2$-cocycles modulo coboundaries in $\wh{\mathcal{C}}_N^*(H^{2,d_2}_1,I)$. Hence, computing $T$ becomes essential. Since $\alpha_{\gamma_1,\gamma_2}$ is already known from~\eqref{alpha_{gamma_1,gamma_2}}, the determination of $T$ reduces to establishing the following proposition.
\end{remark}

\begin{proposition} Assume that we are under the hypothesis of Proposition~\ref{ultimo calculo}, and let $T_1$ and $T_2$ be as in Pro\-position~\ref{maps Ti caso d1=2 d202m con m impar}. For $(\gamma_1,\gamma_2,\mathfrak{f}_{11},\mathfrak{f}_{12})\in \ker T_1\cap \ker T_2$, let $f\coloneqq f_{\mathfrak{f}_{11},\mathfrak{f}_{12},\gamma_1,\gamma_2}$. For $0\le\ell< r_1$ and $u,v\in \mathbb{Z}$, we have
\begin{align*}
f(a_1^{\times \ell},(u,v)) & = L_u^v(\ell) \mathfrak{f}_{11} + \ell\mathrm{r}_{d_2}(v) \mathfrak{f}_{12} + \ell \left\lfloor\frac{\mathrm{r}_{d_2}(v)}{2}\right\rfloor\gamma_1 + \mathrm{r}_2(v)L_u^v(\ell) \gamma_1 + C_u^v(\ell) \left(\left\lfloor\frac{\ell+1}{4}\right\rfloor+L(\ell)\right) \mathfrak{f}_{11}\\
& + C_u^v(\ell) R(\ell) \mathfrak{f}_{12} + C_u^v(\ell) S(\ell)\gamma_1 + C_u^v(\ell) L(\ell) \gamma_1 + C_u^v(\ell)[\ell\ge d_2]\gamma_2,
\end{align*}
where
\begin{alignat*}{2}
& C_u^v(\ell)\coloneqq (u-\ell v)B_1+vB_2,&&\qquad L_u^v(\ell) \coloneqq \begin{cases} \ell \mathrm{r}_2(u) & \text{if $v$ is even,}\\ \left\lfloor \frac{\ell+\mathrm{r}_2(u)}{2}\right\rfloor & \text{if $v$ is odd,} \end{cases}\\
& L(\ell)\coloneqq \left\lfloor\frac{\ell+1}{4} \right\rfloor +\left\lfloor\frac{\mathrm{r}_4(\ell+1)}{3}\right\rfloor,&&\qquad R(\ell)\coloneqq \begin{cases} \binom{\ell}{2} & \text{if $\ell\le d_2$,}\\ \binom{d_2}{2} + \binom{\ell-d_2}{2}& \text{if $\ell>d_2$,}\end{cases}\\
\end{alignat*}
and
$$
S(\ell) \coloneqq \begin{cases} \left\lfloor \frac{\ell-1}{2}\right\rfloor \left\lceil \frac{\ell-1}{2}\right\rceil & \text{if $\ell\le d_2$,}\\  \frac{(d_2-2)d_2}{4}+\left\lfloor \frac{\ell-d_2-1}{2}\right\rfloor \left\lceil\frac{\ell-d_2-1}{2}\right\rceil & \text{if $\ell>d_2$.}\end{cases}
$$
\end{proposition}

\begin{proof} By~\eqref{def de punto} and~\eqref{potencia}, we have
\begin{equation}\label{formula3''}
a_1^{\times j} = \left(\binom{j}{2},j\right)\qquad\text{and}\qquad a_1^{\times j}\cdot (u,v) = (u-j v,v).
\end{equation}
Hence, by~\cite{GGV2}*{Proposition~2.37} and identity~\eqref{formula general para f},
\begin{align*}
f(a_1^{\times \ell},(u,v)) & = \sum_{j=1}^{\ell} f(a_1,(u+v-jv,v))+ C_u^v(\ell)\sum_{j=1}^{\ell-1} (f+\alpha_{\gamma_1,\gamma_2})\left(a_1,\left(\binom{\ell+1-j}{2},\ell-j\right)\right)\\
& = \sum_{i=1}^{\ell} \left(\mathrm{r}_2(u-iv) \mathfrak{f}_{11} + \mathrm{r}_{d_2}(v) \mathfrak{f}_{12} + \left\lfloor \frac{\mathrm{r}_{d_2}(v)}{2}\right\rfloor\gamma_1 + \mathrm{r}_2(u-iv)\mathrm{r}_2(v)\gamma_1\right)\\
& + C_u^v(\ell)\sum_{i=1}^{\ell-1} \left(\mathrm{r}_2\left(\binom{i+1}{2}\right) \mathfrak{f}_{11} + \mathrm{r}_{d_2}(i) \mathfrak{f}_{12} + \left\lfloor \frac{\mathrm{r}_{d_2}(i)}{2}\right\rfloor\gamma_1 + \mathrm{r}_2\left(\binom{i+1}{2}\right)\mathrm{r}_2(i)\gamma_1\right)\\
& + C_u^v(\ell)\sum_{i=1}^{\ell-1} \alpha_2(\gamma_2)(1,\mathrm{r}_{d_2}(i)).
\end{align*}
The formula in the statement follows from this by a direct computation.
\end{proof}

\subsubsection[Case \texorpdfstring{$\lambda$}{lambda} odd, \texorpdfstring{$d_1>2$}{d1>2} and \texorpdfstring{$v_2(d_1)=v_2(d_2)=1$}{v2(d1)=v2(d2)=1}]{Case \texorpdfstring{$\pmb{\lambda}$}{lambda} odd, \texorpdfstring{$\pmb{d_1>2}$}{d1>2} and \texorpdfstring{$\pmb{v_2(d_1)=v_2(d_2)=1}$}{v2(d1)=v2(d2)=1}}
Assume that $\lambda$ is odd, $d_1=2m_1$ and $d_2=2m_2$ for some odd natural numbers $m_1,m_2$ with $1<m_1$ and $m_1\mid m_2$. Let $n$ and $s$ be as at the beginning of this section. By Re\-mark~\ref{caso v_2=1}, we have $s=n=2$, $r_1 = d_1/2$, $r_2 = 2d_2$, and we can take $a_1 = 2e_1\coloneqq (2,0)$ and $a_2 = e_2\coloneqq (0,1)$. Therefore, by~\eqref{def de punto},
\begin{equation}\label{a_j cdot e_i''}
a_i\cdot e_j = \begin{cases} e_j & \text{if $(i,j)\ne (2,2)$,}\\ e_2 -\lambda e_1 & \text{if $(i,j)=(2,2)$.}\end{cases}
\end{equation}
Hence, if we write $a_i=\lambda_{1i}e_1 + \lambda_{2i}e_2$ and  $a_i\cdot e_j=\Gamma^1_{ji}e_1 + \Gamma^2_{ji}e_2$ with $0\le \lambda_{ji}< d_i$ and $0\le \Gamma^k_{ji}< d_k$, then
$$
\lambda_{11} = 2,\quad \lambda_{21} =\lambda_{12} = 0,\quad \lambda_{22} = 1\quad\text{and}\quad \Gamma_{ji}^k = \begin{cases} d_1-\lambda & \text{if $(i,j,k) = (2,2,1)$,}\\ \delta_{jk} & \text{otherwise.}\end{cases}
$$
A straightforward computation, using~\cite{GGV2}*{Subsubsection~2.2.3} and these facts, shows that
$$
T_1,T_2\colon I^6\to I^4,\quad T_3\colon I^6\to I^2\quad\text{and}\quad S\colon I^2\to I^6,
$$
are given by
\begin{align}
& T_1(\gamma_1,\gamma_2, \mathfrak{f}_{11}, \mathfrak{f}_{12}, \mathfrak{f}_{21}, \mathfrak{f}_{22})= \left(d_1 \mathfrak{f}_{11} + \mathcal{U}_{11}, d_2 \mathfrak{f}_{12} + \mathcal{U}_{12}, d_1 \mathfrak{f}_{21} + \mathcal{U}_{21}, d_2 \mathfrak{f}_{22} + \mathcal{U}_{22}\right),\label{T1''}\\
& T_2(\gamma_1,\gamma_2, \mathfrak{f}_{11}, \mathfrak{f}_{12}, \mathfrak{f}_{21}, \mathfrak{f}_{22})= \left(\mathcal{V}_{11} + \mathcal{W}_{11}, \mathcal{V}_{12} + \mathcal{W}_{12}, \mathcal{V}_{21} + \mathcal{W}_{21}, \mathcal{V}_{22} + \mathcal{W}_{22}\right),\label{T2''}\\
& T_3(\gamma_1,\gamma_2, \mathfrak{f}_{11}, \mathfrak{f}_{12}, \mathfrak{f}_{21}, \mathfrak{f}_{22})= (F_1(a_2,a_1,e_1)-F_1(a_1,a_2,e_1),F_1(a_2,a_1,e_2)-F_1(a_1,a_2,e_2))\label{T3''}
\shortintertext{and}
& S(t_1,t_2)=\left(d_1t_1,d_2t_2,\mathcal{X}_{11},\mathcal{X}_{12},\mathcal{X}_{21},\mathcal{X}_{22} + \mathcal{Y}_{22}\right),\label{S''}
\end{align}
in which
\begin{itemize}
	
\item[-] $\mathcal{U}_{ij}\coloneqq \sum_{k=1}^{d_j-1}\alpha_{\gamma_1,\gamma_2}(a_i\cdot e_j,k (a_i\cdot e_j))-a_i\blackdiamond \gamma_j$,
	
\item[-] $\mathcal{V}_{ij}\coloneqq \sum_{\ell=0}^{r_i-1} a_i^{\times (r_i-1-\ell)}\blackdiamond f_1(a_i,a_i^{\times \ell}\cdot e_j)$,
	
\item[-] $\mathcal{W}_{ij}\coloneqq \sum_{\ell=1}^{r_i-1} a_i^{\times \ell}\blackdiamond f_1(a_i,\bar{a}_{i\ell})\Yleft e_j -\sum_{\ell=1}^{r_i-1} a_i^{\times (\ell+1)}\blackdiamond \alpha_{\gamma_1,\gamma_2}(a_i,\bar{a}_{i\ell})\Yleft e_j$, where $\bar{a}_{i\ell}\coloneqq a_i^{\times(\ell+1)}-a_i$,
	
\item[-] $\mathcal{X}_{ij}\coloneqq -t_j + a_i\blackdiamond t_j + \frac{2}{i}(e_j\blackdiamond t_i) \Yleft e_j$,
	
\item[-] $\mathcal{Y}_{22}\coloneqq (\lambda - d_1)t_1 + (d_1-\lambda)(e_2\blackdiamond t_2) \Yleft e_1$,

\item[-] $f_1\coloneqq f_1^{\mathfrak{f}_{11},\mathfrak{f}_{12},\mathfrak{f}_{21},\mathfrak{f}_{22}}$ and $F_1$ is as in~\cite{GGV2}*{Definition~2.25}.

\end{itemize}
By Proposition~\ref{caracterizacion tercer caso}, we know that the actions $\blackdiamond$ and $\Yleft$ are determined by maps~$A_1, A_2, B_1, B_2 \in \End(I)$ satisfying conditions~\eqref{primeras condiciones A_1, A_2, B1 y B2 tercer caso}--\eqref{terceras condiciones A_1, A_2, B1 y B2 tercer caso}. We assume furthermore that $A_1=A_2=\Ide$. Under this hypothesis, ~\eqref{primeras condiciones A_1, A_2, B1 y B2 tercer caso}--\eqref{terceras condiciones A_1, A_2, B1 y B2 tercer caso} reduce to
\begin{equation}\label{pep}
d_1 B_1=d_2 B_2=0,\quad B_2^2=-\lambda B_1, \quad B_1 B_2=0, \quad 4B_1^2=0\quad \text{and} \quad 2B_2 B_1=0.
\end{equation}
So $\lambda B_1^2 = - B_1B_2^2 = 0$ and $\lambda B_2B_1 = - B_2^3 = \lambda B_1B_2 = 0$. Since $\lambda$ is odd and $4B_1^2=2B_2B_1=0$, we conclude that
$$
B_1^2 = B_2B_1=0.
$$

\begin{proposition}\label{S y Ti cuando v2(d1)=2 y d1 ne 2} Under the previous hypothesis the maps $T_i$'s and $S$ become
\begin{align*}
& T_1(\gamma_1,\gamma_2, \mathfrak{f}_{11}, \mathfrak{f}_{12}, \mathfrak{f}_{21},\mathfrak{f}_{22})= \left(d_1 \mathfrak{f}_{11}, d_2 \mathfrak{f}_{12}, d_1 \mathfrak{f}_{21}, d_2 \mathfrak{f}_{22}+(d_2-d_2/d)\gamma_1\right)\\
& T_2(\gamma_1,\gamma_2, \mathfrak{f}_{11}, \mathfrak{f}_{12}, \mathfrak{f}_{21},\mathfrak{f}_{22})= \left(\mathcal{V}_{11} + \mathcal{W}_{11}, \mathcal{V}_{12} + \mathcal{W}_{12}, \mathcal{V}_{21} + \mathcal{W}_{21}, \mathcal{V}_{22} + \mathcal{W}_{22}\right),\\
& T_3(\gamma_1,\gamma_2,\mathfrak{f}_{11},\mathfrak{f}_{12},\mathfrak{f}_{21},\mathfrak{f}_{22})=\left(B_1  (2\mathfrak{f}_{21}-\mathfrak{f}_{12}), (d_1-\lambda)\mathfrak{f}_{11} + (B_2-\lambda B_1)(2\mathfrak{f}_{21}-\mathfrak{f}_{12})\right)
\shortintertext{and}
& S(t_1,t_2)=\left(d_1 t_1,d_2 t_2, 2B_1 t_1, 2B_2 t_1,B_1 t_2,(B_2-\lambda B_1)t_2-(d_1-\lambda)t_1\right),
\end{align*}
in which

\begin{itemize}
	
\item[-] $\mathcal{V}_{ij} = r_i\mathfrak{f}_{ij}$ if $(i,j)\ne (2,2)$ and $\mathcal{V}_{22}\coloneqq d_2(d-1)\lambda \mathfrak{f}_{21} + r_2\mathfrak{f}_{22}+ \left(r_2-\frac{r_2}{d}\right)\gamma_1$,
	
\smallskip

\item[-] $\mathcal{W}_{11} = - B_1\gamma_1$,

\item[-] $\mathcal{W}_{12} = r_1(r_1-1)B_2\mathfrak{f}_{11}- B_2\gamma_1$,

\item[-] $\mathcal{W}_{21} = \lambda \mathrm{r}_d\left(\binom{r_2+1}{3}\right) B_1 \mathfrak{f}_{21} + C_{r_2} B_1\gamma_1 - 2B_1\gamma_2$,

\item[-] $\mathcal{W}_{22} = \sum_{\ell=1}^{r_2-1} \mathrm{r}_{d_1}\left(\lambda \binom{\ell+1}{2} \right) B_2 \mathfrak{f}_{21} + C_{r_2} B_2\gamma_1 - 2B_2\gamma_2$,

\end{itemize}
where $d\coloneqq d_1/\lambda$ and $C_{r_2}\coloneqq \frac{d_2}{d}+\#\bigl\{\ell : 0 \le \ell \le r_2-1\text{ and } \mathrm{r}_d\bigl(\binom{\ell+1}{2}\bigr) + \mathrm{r}_d(-\ell)\ge d\bigr\}$.
\end{proposition}

\begin{proof} We first compute $T_1(\gamma_1,\gamma_2,\mathfrak{f}_{11},\mathfrak{f}_{12},\mathfrak{f}_{21},\mathfrak{f}_{22})$. By~\eqref{a_j cdot e_i''} and~\eqref{alpha_{gamma_1,gamma_2}}, we have $\mathcal{U}_{11}=\mathcal{U}_{12}=\mathcal{U}_{21}=0$ and
$$
\mathcal{U}_{22} = \sum_{k=1}^{d_2-1}\alpha_1(\gamma_1)(d_1-\lambda,k(d_1-\lambda))+\sum_{k=1}^{d_2-1}\alpha_2(\gamma_2)(1,k) - \gamma_2 = \left(d_2-\frac{d_2}{d}\right)\gamma_1,
$$	
where for the last we have used that
$$
\alpha_1(\gamma_1)(d_1-\lambda,k(d_1-\lambda)) = \begin{cases} \gamma_1 & \text{if $d\nmid k$,}\\ 0 & \text{otherwise.}\end{cases}
$$	
The formula for $T_1(\gamma_1,\gamma_2, \mathfrak{f}_{11},\mathfrak{f}_{12},\mathfrak{f}_{21},\mathfrak{f}_{22})$ follows from~\eqref{T1''} and these facts.

We next compute $T_2(\gamma_1,\gamma_2, \mathfrak{f}_{11},\mathfrak{f}_{12},\mathfrak{f}_{21},\mathfrak{f}_{22})$. We begin by establishing a formula for $f_1(a_i,x)$. Let $\lambda_1,\lambda_2\in \mathbb{Z}$. Since, by~\eqref{a_j cdot e_i''} and~\eqref{alpha_{gamma_1,gamma_2}},
\begin{align*}
& \sum_{\ell=1}^{\mathrm{r}_{d_1}(\lambda_1)-1} \alpha_{\gamma_1,\gamma_2} (a_i\cdot e_1,\ell (a_i\cdot e_1)) = 0,\\
&\sum_{\ell=1}^{\mathrm{r}_{d_2}(\lambda_2)-1} \alpha_{\gamma_1,\gamma_2}(a_i\cdot e_2,\ell (a_i\cdot e_2)) = \begin{cases} 0 & \text{if $i=1$,}\\ \left(\mathrm{r}_{d_2}(\lambda_2)-\left\lceil \frac{\mathrm{r}_{d_2}(\lambda_2)}{d}\right\rceil\right) \gamma_1 & \text{if $i=2$,}\end{cases}
\shortintertext{and}
& \alpha_{\gamma_1,\gamma_2}(\lambda_1 e_1,\lambda_2 a_i\cdot e_2) = \begin{cases} 0 & \text{if $i=1$,}\\ \bigl[\mathrm{r}_{d_1}(\lambda_1)+\mathrm{r}_{d_1}(-\lambda\lambda_2)\ge d_1\bigr]\gamma_1 & \text{if $i=2$,}\end{cases}
\end{align*}
by~\cite{GGV2}*{Identitien~(2.17) and~(2.18)}, we have
\begin{align}
& f_1(a_1,\lambda_1 e_1+\lambda_2 e_2) = \mathrm{r}_{d_1}(\lambda_1)\mathfrak{f}_{11}+\mathrm{r}_{d_2}(\lambda_2)\mathfrak{f}_{12}\label{formula general para f'1}
\shortintertext{and}
& f_1(a_2,\lambda_1 e_1+\lambda_2 e_2) = \mathrm{r}_{d_1}(\lambda_1)\mathfrak{f}_{21}+\mathrm{r}_{d_2}(\lambda_2)\mathfrak{f}_{22} +\left(\mathrm{r}_{d_2}(\lambda_2)-\left\lceil \frac{\mathrm{r}_{d_2}(\lambda_2)}{d}\right\rceil + \bigl[\mathrm{r}_{d_1}(\lambda_1)+ \mathrm{r}_{d_1}(-\lambda\lambda_2)\ge d_1\bigr]\right) \gamma_1.\label{formula general para f'2}
\end{align}
Since, by~\eqref{potencia} and~\eqref{def de punto},
$$
a_i^{\times \ell}\cdot e_j= \begin{cases} e_j & \text{ if $(i,j)\ne (2,2)$,}\\ e_2-\lambda\ell e_1 & \text{ if $(i,j)= (2,2)$,} \end{cases}
$$
from this it follows immediately that
\begin{align*}
& \mathcal{V}_{ij} = \sum_{\ell=0}^{r_i-1}f_1(a_i,e_j)=r_i \mathfrak{f}_{ij} &&\text{if $(i,j)\ne (2,2)$,}
\shortintertext{and}
&\mathcal{V}_{22} =\sum_{\ell=0}^{r_2-1} f_1(a_2,-\lambda\ell e_1+e_2)\\
&\phantom{\mathcal{V}_{22}}= \sum_{\ell=0}^{r_2-1} \mathrm{r}_{d_1}(-\lambda\ell)\mathfrak{f}_{21} + r_2\mathfrak{f}_{22} + \sum_{\ell=0}^{r_2-1} \bigl[\mathrm{r}_{d_1}(-\lambda\ell)\!+\!\mathrm{r}_{d_1}(-\lambda)\ge d_1\bigr]\gamma_1\\
&\phantom{\mathcal{V}_{22}}= d_2(d-1)\lambda \mathfrak{f}_{21} + r_2\mathfrak{f}_{22}+ \left(r_2-\frac{r_2}{d}\right)\gamma_1.
\end{align*}
A similar computation using~\eqref{alpha_{gamma_1,gamma_2}}, \eqref{formula general para f'1}, \eqref{formula general para f'2} and that, by~\eqref{potencia} we have $\bar{a}_{1\ell}=(2\ell,0)$ and $\bar{a}_{2\ell}=\bigl(\lambda \binom{\ell+1}{2},\ell\bigr)$, yields
\begin{align*}
\mathcal{W}_{1j} &= \sum_{\ell=1}^{r_1-1}  B_j f\left(a_1,2\ell e_1 \right) - \sum_{\ell=1}^{r_1-1} B_j \alpha_{\gamma_1,\gamma_2}\left(a_1,2\ell e_1\right) = r_1(r_1-1) B_j \mathfrak{f}_{11}-B_j \gamma_1
\shortintertext{and}
\mathcal{W}_{2j} & = \sum_{\ell=1}^{r_2-1}B_j f\left(a_2,\lambda \binom{\ell+1}{2} e_1 + \ell e_2\right) - B_j \sum_{\ell=1}^{r_2-1} \alpha_{\gamma_1,\gamma_2}\left(a_2,\lambda \binom{\ell+1}{2} e_1 + \ell e_2\right)\\
&= \sum_{\ell=1}^{r_2-1} \mathrm{r}_{d_1}\left(\lambda \binom{\ell+1}{2} \right) B_j \mathfrak{f}_{21} + \sum_{\ell=1}^{r_2-1} \mathrm{r}_{d_2}(\ell) B_j \mathfrak{f}_{22} + C_{r_2} B_j\gamma_1 - 2B_j\gamma_2,
\end{align*}
where in the last equality we have use that $d_2B_j = 0$ and that
\begin{align*}
& \sum_{\ell=1}^{r_2-1} \left(\mathrm{r}_{d_2}(\ell)-\left\lceil \frac{\mathrm{r}_{d_2}(\ell)}{d}\right\rceil\right) = \frac{d_2}{d} \left(d_2d-d_2-d+1\right),\\
& \sum_{\ell=1}^{r_2-1} \left[\mathrm{r}_{d_1}\left(\lambda \binom{\ell+1}{2}\right)\!+\!\mathrm{r}_{d_1}(-\lambda \ell)\ge d_1\right] = \sum_{\ell=1}^{r_2-1} \left[\mathrm{r}_d\left(\binom{\ell+1}{2}\right)+\mathrm{r}_d(-\ell)\ge d\right] = C_{r_2}-\frac{d_2}{d}
\shortintertext{and}
& \sum_{\ell=1}^{r_2-1} \alpha_{\gamma_1,\gamma_2}\left(a_2,\lambda \binom{\ell+1}{2} e_1 + \ell e_2\right) = \sum_{\ell=1}^{r_2-1} \alpha_2(\gamma_2)\left(1,\ell\right) =  2\gamma_2.
\end{align*}
Since, by~\eqref{pep},
\begin{align*}
& r_1(r_1-1)B_1 = 0,\quad \sum_{\ell=1}^{r_2-1} r_{d_2}(\ell) B_j = \sum_{\ell=1}^{r_2-1} \ell B_j = d_2(r_2-1)B_j = 0
\shortintertext{and}
& \sum_{\ell=1}^{r_2-1} \mathrm{r}_{d_1}\left(\lambda \binom{\ell+1}{2} \right)B_1 = \mathrm{r}_{d_1}\left(\lambda \sum_{\ell=1}^{r_2-1} \binom{\ell+1}{2} \right)B_1 = \lambda \mathrm{r}_d\left(\binom{r_2+1}{3}\right)B_1,
\end{align*}
we obtain that
\begin{align*}
& \mathcal{W}_{11} =  -B_1\gamma_1,\quad \mathcal{W}_{12} = r_1(r_1-1) B_2 \mathfrak{f}_{11}-B_2 \gamma_1,\quad\mathcal{W}_{21} = \lambda \mathrm{r}_d\left(\binom{r_2+1}{3} \right) B_1 \mathfrak{f}_{21} + C_{r_2} B_1\gamma_1 - 2B_1\gamma_2
\shortintertext{and}
& \mathcal{W}_{22} = \sum_{\ell=1}^{r_2-1} \mathrm{r}_{d_1}\left(\lambda \binom{\ell+1}{2} \right) B_2 \mathfrak{f}_{21} + C_{r_2} B_2\gamma_1 - 2B_2\gamma_2.
\end{align*}
By substituting into~\eqref{T2''} these equalities together with the formulas for $\mathcal{V}_{ij}$ obtained above, we arrive at the formula for $T_2$ as stated in the proposition.

Next we compute $T_3(\gamma_1,\gamma_2, \mathfrak{f}_{11}, \mathfrak{f}_{12}, \mathfrak{f}_{21},\mathfrak{f}_{22})$. Using~\cite{GGV2}*{Definition~2.25}, \eqref{producto}, \eqref{def de punto}, \eqref{formula general para f'1}, \eqref{formula general para f'2}, \eqref{alpha_{gamma_1,gamma_2}} and the fact that $d_1B_1=0$, we obtain
\begin{align*}
&F_1(a_2,a_1,e_1) = f_1(a_1,e_1) + f_1(a_2,e_1) + B_1f_1(a_2,2e_1) - B_1\alpha_{\gamma_1,\gamma_2}(a_2,2e_1)= \mathfrak{f}_{11} + \mathfrak{f}_{21} + 2B_1 \mathfrak{f}_{21},\\
& F_1(a_1,a_2,e_1) = f_1(a_2,e_1) + f_1(a_1,e_1) + B_1 f_1(a_1,e_2) - B_1\alpha_{\gamma_1,\gamma_2}(a_1,e_2)= \mathfrak{f}_{21} + \mathfrak{f}_{11} + B_1\mathfrak{f}_{12},\\
&F_1(a_2,a_1,e_2) = f_1(a_1,-\lambda e_1+e_2) + f_1(a_2,e_2) + (d_1-\lambda)B_1 f_1(a_2,2e_1)\\
&\phantom{F_1(a_2,a_1,e_2)} + B_2 f_1(a_2,2e_1) - (d_1-\lambda)B_1 \alpha_{\gamma_1,\gamma_2}(a_2,2e_1) - B_2 \alpha_{\gamma_1,\gamma_2}(a_2,2e_1)\\
&\phantom{F_1(a_2,a_1,e_2)}= (d_1-\lambda)\mathfrak{f}_{11} + \mathfrak{f}_{12} + \mathfrak{f}_{22} -2\lambda B_1 \mathfrak{f}_{21} + 2B_2 \mathfrak{f}_{21}
\shortintertext{and}
&F_1(a_1,a_2,e_2) = f_1(a_2,e_2) + f_1(a_1,e_2) + (d_1-\lambda)B_1f_1(a_1,e_2) + B_2f_1(a_1,e_2)\\
&\phantom{F_1(a_1,a_2,e_2)} + (d_1-\lambda)B_1\alpha_{\gamma_1,\gamma_2}(a_1,e_2) + B_2\alpha_{\gamma_1,\gamma_2}(a_1,e_2)\\
&\phantom{F_1(a_1,a_2,e_2)}= \mathfrak{f}_{22} + \mathfrak{f}_{12} - \lambda B_1 \mathfrak{f}_{12} + B_2\mathfrak{f}_{12}.
\end{align*}
The formula for $T_3(\gamma_1,\gamma_2, \mathfrak{f}_{11}, \mathfrak{f}_{12}, \mathfrak{f}_{21},\mathfrak{f}_{22})$ follows immediately from these facts.

Finally, the formula for $S(t_1,t_2)$ follows straightforwardly from~\eqref{S''} and the fact that $d_1B_1 = 0$.
\end{proof}

\begin{remark}\label{sobre C{r2}} Since $\binom{\ell+1+2d}{2}=\binom{\ell+1}{2} +\binom{2d}{2}+2d(\ell+1)$ and $d\mid \binom{2d}{2}+2d(\ell+1)$, we see that $C_{r_2} = \frac{d_2}{d} \wt{C}_{r_2}$, where
$$
\wt{C}_{r_2}\coloneqq 1+ \#\left\{\ell : 0 \le \ell \le 2d-1\text{ and } \mathrm{r}_d\left(\binom{\ell+1}{2}\right) + \mathrm{r}_d(-\ell)\ge d\right\}.
$$
\end{remark}

In the following proposition we compute the cohomology, adding some conditions.

\begin{proposition}\label{calculo homo} Assume that $r_1I=0$, $r_1d\mid d_2$, and $3\nmid m_2$ or $v_3(m_1)-v_3(\lambda)<v_3(m_2)$. Assume also that there exists $\wt{B}_2\in \End(I)$ such that $B_2 = \lambda \wt{B}_2$, $B_1 = -\wt{B}_2B_2$ and $\wt{B}_2B_1 = 0$ (note that this is trivially true when $\lambda = 1$). Then
\begin{equation}\label{cohomologia caso d1=2m_1, d2=2m_2 m_1,m_2 impares}
\ho_{\blackdiamond,\Yleft}^2(H^{d_1d_2}_{\lambda},I)\simeq \ker B_2\oplus \ker B_2\oplus I_{\lambda} \oplus \ker B_1\oplus \frac{I}{\ima B_1}\oplus \frac{I}{\lambda I}.
\end{equation}
\end{proposition}

\begin{proof} Since $r_1d\mid d_2$, we have $C_{r_2}I = \wt{C}_{r_2} \frac{d_2}{d} I = 0$. Also $\mathrm{r}_d\left(\binom{r_2+1}{3} \right) = 0$, because $3\nmid m_2$ or $v_3(m_1)-v_3(\lambda)<v_3(m_2)$. By these facts the formulas in Proposition~\ref{S y Ti cuando v2(d1)=2 y d1 ne 2} reduce to
\begin{align*}
& T_1(\gamma_1,\gamma_2, \mathfrak{f}_{11}, \mathfrak{f}_{12}, \mathfrak{f}_{21},\mathfrak{f}_{22})= (0, 0, 0, 0),\\
& T_2(\gamma_1,\gamma_2, \mathfrak{f}_{11}, \mathfrak{f}_{12},\mathfrak{f}_{21},\mathfrak{f}_{22})= -\bigl(B_1\gamma_1,B_2\gamma_1,2B_1\gamma_2,2B_2 \gamma_2\bigr),\\
& T_3(\gamma_1,\gamma_2,\mathfrak{f}_{11},\mathfrak{f}_{12},\mathfrak{f}_{21},\mathfrak{f}_{22})= \bigl(B_1 (2\mathfrak{f}_{21}-\mathfrak{f}_{12}),
-\lambda\mathfrak{f}_{11}+(B_2-\lambda B_1)(2\mathfrak{f}_{21}-\mathfrak{f}_{12})\bigr)
\shortintertext{and}
& S(t_1,t_2)=(0,0,2B_1 t_1,2B_2 t_1,B_1 t_2,(B_2-\lambda B_1)t_2+\lambda t_1).
\end{align*}
Consider the map $\phi \colon I^6\to I^6$ given by
$$
\phi(\gamma_1,\gamma_2,\mathfrak{f}_{11},\mathfrak{f}_{12},\mathfrak{f}_{21},\mathfrak{f}_{22}) \coloneqq \left(\gamma_1,\gamma_2,\mathfrak{f}_{11} -\wt{B}_2\mathfrak{f}_{12}-2\wt{B}_2^2\mathfrak{f}_{22}, \mathfrak{f}_{12}+2\wt{B}_2\mathfrak{f}_{22} +2\mathfrak{f}_{21}, \mathfrak{f}_{21},\mathfrak{f}_{22} -\lambda \mathfrak{f}_{21}\right).
$$
Note that $\phi$ invertible with inverse
$$
\phi^{-1}(\gamma_1,\gamma_2,\mathfrak{f}_{11},\mathfrak{f}_{12},\mathfrak{f}_{21},\mathfrak{f}_{22}) = \left(\gamma_1,\gamma_2,\mathfrak{f}_{11} +\wt{B}_2\mathfrak{f}_{12}-2\wt{B}_2\mathfrak{f}_{21},\mathfrak{f}_{12}-2\mathfrak{f}_{21}-2\wt{B}_2\mathfrak{f}_{22}-2\lambda\wt{B}_2 \mathfrak{f}_{21} , \mathfrak{f}_{21}, \mathfrak{f}_{22}+\lambda \mathfrak{f}_{21}\right).
$$
Set $\wt{T}_i = T_i \phi$ and $\wt{S} = \phi^{-1} S$, where the $T_i$'s and $S$ are as in Proposition~\ref{S y Ti cuando v2(d1)=2 y d1 ne 2}. By~\eqref{formula00}, we have
\begin{equation*}
\ho_{\blackdiamond,\Yleft}^2(H^{d_1d_2}_{\lambda},I)\simeq \frac{\ker T_1\cap\ker T_2\cap\ker T_3}{\ima S} \simeq \frac{\ker \wt{T}_1\cap\ker \wt{T}_2\cap\ker \wt{T}_2}{\ima \wt{S}},
\end{equation*}
where the last isomorphism is induced by $\phi^{-1}$. A straightforward computation using that
$$
B_2\wt{B}_2 =\lambda\wt{B}_2^2 = \wt{B}_2B_2 = -B_1\quad\text{and}\quad B_1\wt{B}_2 = - \lambda\wt{B}_2^3 = \wt{B}_2B_1 = 0,
$$
shows that
\begin{align*}
&\wt{T}_1(\gamma_1,\gamma_2,\mathfrak{f}_{11},\mathfrak{f}_{12},\mathfrak{f}_{21},\mathfrak{f}_{22})=\left(0,0,0,0\right),\\
&\wt{T}_2(\gamma_1,\gamma_2,\mathfrak{f}_{11},\mathfrak{f}_{12},\mathfrak{f}_{21},\mathfrak{f}_{22}) =-\bigl(B_1\gamma_1,B_2\gamma_1,2B_1\gamma_2,2B_2 \gamma_2\bigr),\\
& \wt{T}_3(\gamma_1,\gamma_2,\mathfrak{f}_{11},\mathfrak{f}_{12},\mathfrak{f}_{21},\mathfrak{f}_{22})= \bigl(-B_1\mathfrak{f}_{12}, -\lambda\mathfrak{f}_{11}+\lambda B_1\mathfrak{f}_{12}\bigr)
\shortintertext{and}
&\wt{S}(t_1,t_2)= (0,0,0,0,B_1 t_2,B_2 t_2+\lambda t_1).
\end{align*}
Since $\ker B_1\subseteq \ker B_2$ and the multiplication by $2$ is an isomorphism of $I$, we have
$$
\ker \wt{T}_1\cap\ker \wt{T}_2\cap\ker \wt{T}_3 = \ker B_2\oplus \ker B_2\oplus I_{\lambda} \oplus \ker B_1\oplus I\oplus I.
$$
Moreover, $\ima \wt{S}= 0\oplus 0\oplus 0\oplus 0\oplus\ima B_1\oplus \lambda I$, because $\ima B_2\subseteq \lambda I$. Thus~\eqref{cohomologia caso d1=2m_1, d2=2m_2 m_1,m_2 impares} holds.
\end{proof}

\begin{example} Let $t,t_1,t_2,\lambda\in \mathds{N}$ be odd natural numbers and let $d_1\coloneqq 2t^3\lambda t_1$, $d_2\coloneqq 2t^6\lambda t_1^2t_2$ and $I\coloneqq \mathds{Z}_{t^3\lambda t_1}$. Let $B_1,B_2\colon I\to I$ be the morphisms defined by $B_1(y)\coloneqq -t^2\lambda t_1^2y$ and $B_2(y)\coloneqq t\lambda t_1y$. Then
\begin{equation*}
\ho_{\blackdiamond,\Yleft}^2(H^{d_1d_2}_{\lambda},I)\simeq t^2I\oplus t^2I\oplus t^3t_1I \oplus \frac{t}{\gcd(t_1,t)}I\oplus \frac{I}{t^2\lambda t_1I}\oplus \frac{I}{\lambda I}.
\end{equation*}
\end{example}

\begin{remark}\label{remark1uu} Assume that we are under the hypothesis of Proposition~\ref{calculo homo}. The map $\phi\colon I^6\to I^6$ induces an~isomor\-phism
$$
\ov{\phi}\colon \ker B_2\oplus \ker B_2\oplus I_{\lambda} \oplus \ker B_1\oplus \frac{I}{\ima B_1}\oplus \frac{I}{\lambda I}\longrightarrow \frac{\ker T_1\cap \ker T_2\cap \ker T_3}{\ima S}.
$$
Let $Z\subseteq \ker B_2\oplus \ker B_2\oplus I_{\lambda} \oplus \ker B_1\oplus I^2$ be such that the canonical map $Z\to \ker B_2\oplus \ker B_2\oplus I_{\lambda} \oplus \ker B_1\oplus \frac{I}{\ima B_1}\oplus \frac{I}{\lambda I}$ is a bijection and let~$T$ be as in equality~\eqref{formula0}, so that
$$
T(\gamma_1,\gamma_2,\mathfrak{f}_{11},\mathfrak{f}_{12},\mathfrak{f}_{21},\mathfrak{f}_{22}) = (\alpha_{\gamma_1,\gamma_2}, f_{\mathfrak{f}_{11}, \mathfrak{f}_{12},\mathfrak{f}_{21},\mathfrak{f}_{22},\gamma_1,\gamma_2}).
$$
Clear\-ly~$T(\psi(Z))$ is a complete family of representatives of~$2$-cocycles modulo coboundaries in $\wh{\mathcal{C}}_N^*(H^{2,d_2}_1,I)$. Hence, computing $T$ becomes essential. Since $\alpha_{\gamma_1,\gamma_2}$ is already known from~\eqref{alpha_{gamma_1,gamma_2}}, the determination of $T$ reduces to establishing a formula for
\begin{equation}\label{etiqueta0}
f(a_1^{\times \ell_1}\times a_2^{\times \ell_2},(u,v)),
\end{equation}
where $f\coloneqq f_{\mathfrak{f}_{11}, \mathfrak{f}_{12},\mathfrak{f}_{21},\mathfrak{f}_{22},\gamma_1,\gamma_2}$, $0\le \ell_1<r_1$ and $0\le \ell_2<r_2$. This formula is too complicated to be written here. Next, we derive results that permit its calculation. To begin with, note that, by~\cite{GGV2}*{Proposition~2.37},
\begin{equation}\label{etiqueta1}
f(a_1^{\times \ell_1}\times a_2^{\times \ell_2},(u,v)) = f(a_1^{\times \ell_1},(u,v)) + f(a_2^{\times \ell_2},(u,v)) +  C_u^v(\ell_2) (f+\alpha_{\gamma_1,\gamma_2})(a_1^{\times \ell_1},a_2^{\times \ell_2}),
\end{equation}
where $C_u^v(\ell_2)\coloneqq uB_1-\ell_2\lambda v B_1+vB_2$. Note now that, by~\eqref{def de punto}, \eqref{producto} and~\eqref{potencia}, we have
\begin{equation}\label{formula3'''}
a_1^{\times \ell_1}\times a_2^{\times \ell_2} = \left(2\ell_1+\lambda \binom{\ell_2}{2},\ell_2\right)\qquad\text{and}\qquad \bigl(a_1^{\times \ell_1}\times a_2^{\times \ell_2}\bigr)\cdot (u,v) = (u-\ell_2\lambda v,v).
\end{equation}
Using this, \eqref{formula general para f'1} and again~\cite{GGV2}*{Proposition~2.37}, we obtain that,
\begin{align*}
f(a_1^{\times \ell_1},(x,y)) & = \ell_1f(a_1,(x,y)) + C_x^y(0) \sum_{j=1}^{\ell_1-1} f(a_1,a_1^{\times \ell_1-j}) + C_x^y(0)\sum_{j=1}^{\ell_1-1} \alpha_{\gamma_1,\gamma_2}(a_1,a_1^{\times \ell_1-j})\\
& = \ell_1 \mathrm{r}_{d_1}(x) \mathfrak{f}_{11} + \ell_1 \mathrm{r}_{d_2}(y)\mathfrak{f}_{12} + \ell_1(\ell_1-1)C_x^y(0)\mathfrak{f}_{11} + C_x^y(0) \sum_{j=1}^{\ell_1-1} \alpha_1(\gamma_1)(2,2\ell_1-2j)\\
& = \ell_1 \mathrm{r}_{d_1}(x) \mathfrak{f}_{11} + \ell_1\mathrm{r}_{d_2}(y)\mathfrak{f}_{12} + \ell_1(\ell_1-1)C_x^y(0)\mathfrak{f}_{11},\\[2pt]
\alpha_{\gamma_1,\gamma_2}(a_1^{\times \ell_1},a_2^{\times \ell_2}) & = \alpha_1(\gamma_1)\left(2\ell_1,\lambda\binom{\ell_2}{2}\right) = \left[2\ell_1+\mathrm{r}_{d_1}\left(\lambda\binom{\ell_2}{2}\right)\ge d_1\right]\gamma_1
\shortintertext{and}
f(a_2^{\times \ell_2},(u,v)) & = \sum_{j=0}^{\ell_2-1} f(a_2,a_2^{\times j}\cdot (u,v)) + C_u^v(\ell_2) \sum_{j=1}^{\ell_2-1} f(a_2,{}^{a_2} a_2^{\times \ell_2-j}) + C_u^v(\ell_2) \sum_{j=1}^{\ell_2-1} \alpha_{\gamma_1,\gamma_2}(a_2,{}^{a_2} a_2^{\times \ell_2-j})\\
& = \sum_{j=0}^{\ell_2-1} f(a_2,(u-j\lambda v,v)) + C_u^v(\ell_2) \sum_{j=1}^{\ell_2-1} (f+\alpha_{\gamma_1,\gamma_2}) \left(a_2,\left(\lambda\binom{\ell_2-j+1}{2},\ell_2-j\right)\right)\\
& = \sum_{j=0}^{\ell_2-1} f(a_2,(u-j\lambda v,v)) + C_u^v(\ell_2) \sum_{j=1}^{\ell_2-1} f \left(a_2,\left(\lambda\binom{\ell_2-j+1}{2},\ell_2-j\right)\right),
\end{align*}
where in the last equality we have used that
$$
\alpha_{\gamma_1,\gamma_2}\left(a_2,\left(\lambda\binom{\ell_2-j+1}{2},\ell_2-j\right)\right) = \alpha_2(\gamma_2)(1,\ell_2-j) = 0.
$$
The formula for $f(a_1^{\times \ell_1}\times a_2^{\times \ell_2},(u,v))$ can be obtained from~\eqref{etiqueta1}, these facts, and equality~\eqref{formula general para f'2}.
\end{remark}

\subsubsection[Case \texorpdfstring{$v_2(d_1)=v_2(d_2)\ge 2$}{v2(d1)=v2(d2) >=2}, and \texorpdfstring{$\lambda$}{lambda} odd]{Case \texorpdfstring{$\pmb{v_2(d_1)=v_2(d_2)\ge 2}$}{v2(d1)=v2(d2)>=2} and  \texorpdfstring{$\pmb{\lambda}$}{lambda} odd}

Assume that $\lambda$ is odd, $d_1=2^wm_1$ and $d_2=2^wm_2$ with $w\ge 2$ and $m_1,m_2$ odd natural numbers such that $m_1\mid m_2$. Let~$n$ and $s$ be as at the beginning of this section. By Re\-mark~\ref{caso final}, we have $s=n=2$, $r_1=d_1/2$, $r_2=2d_2$, and we can take $a_1 = (2t,2q)$ and $a_2 = e_2\coloneqq (0,1)$, where $q\coloneqq d_2/d_1$ and $0\le t < d_1/2$ is such that
\begin{equation*}
\lambda q(2q - 1) - 2t= 1 \text{ in } \mathds{Z}_{d_1} \quad\text{($t$ exists since $\lambda$ and $q$ are odd)}.
\end{equation*}
Therefore, by~\eqref{def de punto},
\begin{equation}\label{a_j cdot e_i'''}
a_i\cdot e_j = \begin{cases} e_1 & \text{if $j= 1$,}\\ e_2 - 2q\lambda e_1 & \text{if $(i,j)=(1,2)$,}\\ e_2 -\lambda e_1 & \text{if $(i,j)=(2,2)$.}\end{cases}
\end{equation}
Hence, if we write $a_i=\lambda_{1i}e_1 + \lambda_{2i}e_2$ and  $a_i\cdot e_j=\Gamma^1_{ji}e_1 + \Gamma^2_{ji}e_2$ with $0\le \lambda_{ji}< d_i$ and $0\le \Gamma^k_{ji}< d_k$, then
$$
\lambda_{11} = 2t,\quad \lambda_{21} = 2q,\quad \lambda_{12} = 0,\quad \lambda_{22} = 1\quad\text{and}\quad \Gamma_{ji}^k = \begin{cases} \delta_{1k} & \text{if $j=1$,}\\ r_{d_1}(-2q\lambda) & \text{if $(i,j,k) = (1,2,1)$,}\\ d_1-\lambda & \text{if $(i,j,k) = (2,2,1)$,}\\ 1 & \text{if $(j,k) = (2,2)$.} \end{cases}
$$
A straightforward computation, using~\cite{GGV2}*{Subsubsection~2.2.3} and these facts, shows that
$$
T_1,T_2\colon I^6\to I^4,\quad T_3\colon I^6\to I^4 \quad\text{and}\quad S\colon I^2\to I^6,
$$
are given by
\begin{align}
& T_1(\gamma_1,\gamma_2, \mathfrak{f}_{11}, \mathfrak{f}_{12}, \mathfrak{f}_{21}, \mathfrak{f}_{22})= \left(d_1 \mathfrak{f}_{11} + \mathcal{U}_{11}, d_2 \mathfrak{f}_{12} + \mathcal{U}_{12}, d_1 \mathfrak{f}_{21} + \mathcal{U}_{21}, d_2 \mathfrak{f}_{22} + \mathcal{U}_{22}\right),\label{T1'''}\\
& T_2(\gamma_1,\gamma_2, \mathfrak{f}_{11}, \mathfrak{f}_{12}, \mathfrak{f}_{21}, \mathfrak{f}_{22})= \left(\mathcal{V}_{11} + \mathcal{W}_{11}, \mathcal{V}_{12} + \mathcal{W}_{12}, \mathcal{V}_{21} + \mathcal{W}_{21}, \mathcal{V}_{22} + \mathcal{W}_{22}\right),\label{T2'''}\\
& T_3(\gamma_1,\gamma_2, \mathfrak{f}_{11}, \mathfrak{f}_{12}, \mathfrak{f}_{21}, \mathfrak{f}_{22})= (F_1(a_2,a_1,e_1)-F_1(a_1,a_2,e_1),F_1(a_2,a_1,e_2)-F_1(a_1,a_2,e_2))\label{T3'''}
\shortintertext{and}
& S(t_1,t_2)=\left(d_1t_1,d_2t_2,\mathcal{X}_{11},\mathcal{X}_{12},\mathcal{X}_{21},\mathcal{X}_{22}\right),\label{S'''}
\end{align}
in which
\begin{itemize}
	
\item[-] $\mathcal{U}_{ij}\coloneqq \sum_{k=1}^{d_j-1}\alpha_{\gamma_1,\gamma_2}(a_i\cdot e_j,k (a_i\cdot e_j))-a_i\blackdiamond \gamma_j$,
	
\item[-] $\mathcal{V}_{ij}\coloneqq \sum_{\ell=0}^{r_i-1} a_i^{\times (r_i-1-\ell)}\blackdiamond f_1(a_i,a_i^{\times \ell}\cdot e_j)$,
	
\item[-] $\mathcal{W}_{ij}\coloneqq \sum_{\ell=1}^{r_i-1} a_i^{\times \ell}\blackdiamond f_1(a_i,\bar{a}_{i\ell})\Yleft e_j -\sum_{\ell=1}^{r_i-1} a_i^{\times (\ell+1)}\blackdiamond \alpha_{\gamma_1,\gamma_2}(a_i,\bar{a}_{i\ell})\Yleft e_j$, where $\bar{a}_{i\ell}\coloneqq a_i^{\times(\ell+1)}-a_i$,
	
\item[-] $\mathcal{X}_{11}\coloneqq -t_1 + a_1\blackdiamond t_1 + 2t (e_1\blackdiamond t_1) \Yleft e_1+ 2q (e_1\blackdiamond t_2) \Yleft e_1$,

\item[-] $\mathcal{X}_{21}\coloneqq -t_1 + a_2\blackdiamond t_1 + (e_1\blackdiamond t_2) \Yleft e_1$,

\item[-] $\mathcal{X}_{12}\coloneqq - r_{d_1}(-2q\lambda) t_1 - t_2 + a_1\blackdiamond t_2 + 2t (e_2\blackdiamond t_1) \Yleft (-2q\lambda e_1 + e_2) + 2q (e_2\blackdiamond t_2) \Yleft (-2q\lambda e_1 + e_2)$,
	
\item[-] $\mathcal{X}_{22}\coloneqq (\lambda-d_1)t_1 - t_2 + a_2\blackdiamond t_2 + (e_2\blackdiamond t_2)\Yleft (-\lambda e_1+e_{22})$,

\item[-] $f_1\coloneqq f_1^{\mathfrak{f}_{11},\mathfrak{f}_{12},\mathfrak{f}_{21},\mathfrak{f}_{22}}$ and $F_1$ is as in~\cite{GGV2}*{Definition~2.25}.

\end{itemize}
By Proposition~\ref{caracterizacion cuarto caso}, we know that the actions $\blackdiamond$ and $\Yleft$ are determined by maps~$A_1, A_2, B_1, B_2 \in \End(I)$ satisfying conditions~\eqref{primeras condiciones A_1, A_2, B1 y B2 cuarto caso}--\eqref{cuartas condiciones A_1, A_2, B1 y B2 cuarto caso}. We assume furthermore that $A_1=A_2=\Ide$, $B_1=B_2=0$ and $d_2=d_1$. Therefore, $q = 1$ and $t = \frac{\lambda-1}{2}$, which implies $a_1 = (\lambda-1,2)$.

\begin{proposition}\label{prop fin} The maps $T_i$'s and $S$ are given by
\begin{align*}
& T_1(\gamma_1,\gamma_2,\mathfrak{f}_{11},\mathfrak{f}_{12},\mathfrak{f}_{21},\mathfrak{f}_{22}) = \left(d_1\mathfrak{f}_{11},d_1\mathfrak{f}_{12} + (d_1-2\lambda)\gamma_1,d_1\mathfrak{f}_{21}, d_1\mathfrak{f}_{22} + (d_1-1)\gamma_1\right),\\
& T_2(\gamma_1,\gamma_2,\mathfrak{f}_{11},\mathfrak{f}_{12},\mathfrak{f}_{21},\mathfrak{f}_{22})= \left(r_1\mathfrak{f}_{11},r_1(r_1-\lambda) \mathfrak{f}_{11} + r_1\mathfrak{f}_{12}+ \left(r_1-\lambda\right)\gamma_1, r_2\mathfrak{f}_{21},d_1(d_1-\lambda) \mathfrak{f}_{21} + r_2\mathfrak{f}_{22}+ 2(d_1-\lambda)\gamma_1 \right),\\
& T_3(\gamma_1,\gamma_2,\mathfrak{f}_{11},\mathfrak{f}_{12},\mathfrak{f}_{21},\mathfrak{f}_{22})=(0,(d_1-\lambda) \mathfrak{f}_{11} - (d_1-2\lambda) \mathfrak{f}_{21})
\shortintertext{and}
&S(t_1,t_2)=\left(d_1t_1,d_2t_2, 0, (2\lambda-d_1)t_1, 0, (\lambda-d_1) t_1\right).
\end{align*}
\end{proposition}

\begin{proof} We first compute $T_1(\gamma_1,\gamma_2,\mathfrak{f}_{11},\mathfrak{f}_{12},\mathfrak{f}_{21},\mathfrak{f}_{22})$. By~\eqref{a_j cdot e_i'''} and~\eqref{alpha_{gamma_1,gamma_2}}, we have
\begin{align*}
\mathcal{U}_{11} &= \sum_{k=1}^{d_1-1}\alpha_{\gamma_1,\gamma_2}(e_1,k e_1)-\gamma_1=0,\\
\mathcal{U}_{12} &= \sum_{k=1}^{d_1-1}\alpha_{\gamma_1,\gamma_2}(-2\lambda e_1+e_2,-2k\lambda e_1+ke_2)) - \gamma_2 = \sum_{k=1}^{d_1-1}\alpha_1(\gamma_1)(d_1-2\lambda,-2k\lambda) = (d_1-2\lambda)\gamma_1,\\
\mathcal{U}_{21} &= \sum_{k=1}^{d_1-1}\alpha_{\gamma_1,\gamma_2}(e_1,k e_1)-\gamma_1=0
\shortintertext{and}
\mathcal{U}_{22} & = \sum_{k=1}^{d_1-1}\alpha_{\gamma_1,\gamma_2}(-\lambda e_1+e_2,-k\lambda e_1+ke_2) - \gamma_2 = \sum_{k=1}^{d_1-1}\alpha_1(\gamma_1)(d_1-1,-k) = (d_1-1)\gamma_1.
\end{align*}
The formula for $T_1(\gamma_1,\gamma_2,\mathfrak{f}_{11},\mathfrak{f}_{12},\mathfrak{f}_{21},\mathfrak{f}_{22})$ follows immediately from these facts.

We next compute $T_2(\gamma_1,\gamma_2, \mathfrak{f}_{11},\mathfrak{f}_{12},\mathfrak{f}_{21},\mathfrak{f}_{22})$. We begin by establishing a general formula for $f_1(a_i,x)$. Let $u,v\in \mathbb{Z}$. Since, by~\eqref{a_j cdot e_i'''} and~\eqref{alpha_{gamma_1,gamma_2}},
\begin{align*}
& \sum_{\ell=1}^{\mathrm{r}_{d_1}(u)-1} \alpha_{\gamma_1,\gamma_2} (a_i\cdot e_1,\ell (a_i\cdot e_1)) = 0,\\
&\sum_{\ell=1}^{\mathrm{r}_{d_1}(v)-1} \alpha_{\gamma_1,\gamma_2}(a_i\cdot e_2,\ell (a_i\cdot e_2)) = \begin{cases} \left(\mathrm{r}_{d_1}(v)-\left\lceil \frac{2\mathrm{r}_{d_1}(v)}{d}\right\rceil\right) \gamma_1 & \text{if $i=1$,}\\[5pt] \left(\mathrm{r}_{d_1}(v)-\left\lceil \frac{\mathrm{r}_{d_1}(v)}{d}\right\rceil\right) \gamma_1 & \text{if $i=2$,}\end{cases}
\shortintertext{and}
& \alpha_{\gamma_1,\gamma_2}(u e_1,v a_i\cdot e_2) = \begin{cases} \bigl[\mathrm{r}_{d_1}(u)+\mathrm{r}_{d_1}(-2\lambda v)\ge d_1\bigr]\gamma_1 & \text{if $i=1$,}\\ \bigl[\mathrm{r}_{d_1}(u)+\mathrm{r}_{d_1}(-\lambda v)\ge d_1\bigr]\gamma_1 & \text{if $i=2$,}\end{cases}
\end{align*}
by~\cite{GGV2}*{Identities~(2.17) and~(2.18}, we have
\begin{align}
& f_1(a_1,(u,v)) = \mathrm{r}_{d_1}(u)\mathfrak{f}_{11}+\mathrm{r}_{d_1}(v)\mathfrak{f}_{12} +\left(\mathrm{r}_{d_1}(v)-\left\lceil \frac{2\mathrm{r}_{d_1}(v)}{d}\right\rceil\right) \gamma_1+ \bigl[\mathrm{r}_{d_1}(u)+\mathrm{r}_{d_1}(-2\lambda v)\ge d_1\bigr]\gamma_1\label{formula general para f''1}
\shortintertext{and}
& f_1(a_2,(u,v)) = \mathrm{r}_{d_1}(u)\mathfrak{f}_{21}+\mathrm{r}_{d_1}(v)\mathfrak{f}_{22}+ \left(\mathrm{r}_{d_1}(v)- \left\lceil \frac{\mathrm{r}_{d_1}(v)}{d}\right\rceil\right)\gamma_1 + \bigl[\mathrm{r}_{d_1}(u)+\mathrm{r}_{d_1}(-\lambda v)\ge d_1\bigr]\gamma_1.\label{formula general para f''2}
\end{align}
Since, by~\eqref{def de punto} and~\eqref{potencia},
$$
a_i^{\times \ell}\cdot e_j=\begin{cases} e_1 & \text{if $j=1$,}\\-2\ell\lambda e_1 + e_2 & \text{if $i=1$ and $j=2$,}\\-\ell\lambda e_1 + e_2, & \text{if $i=2$ and $j=2$,}\end{cases}
$$
from this it follows immediately that
\begin{align*}
& \mathcal{V}_{i1} = \sum_{\ell=0}^{r_i-1}f_1(a_i,e_1)=r_i \mathfrak{f}_{i1} \\
&\mathcal{V}_{12} =\sum_{\ell=0}^{r_1-1} f_1(a_1,-2\lambda\ell e_1+e_2)\\
&\phantom{\mathcal{V}_{22}}= \sum_{\ell=0}^{r_1-1} \mathrm{r}_{d_1}(-2\lambda\ell)\mathfrak{f}_{11} + r_1\mathfrak{f}_{12} + \sum_{\ell=0}^{r_1-1} \bigl[\mathrm{r}_{d_1}(-2\lambda\ell)+(d_1-2\lambda)\ge d_1\bigr]\gamma_1\\
&\phantom{\mathcal{V}_{22}}= r_1(r_1-\lambda) \mathfrak{f}_{11} + r_1\mathfrak{f}_{12}+ \left(r_1-\lambda\right)\gamma_1
\shortintertext{and}
&\mathcal{V}_{22} =\sum_{\ell=0}^{r_2-1} f_1(a_2,-\lambda\ell e_1+e_2)\\
&\phantom{\mathcal{V}_{22}}= \sum_{\ell=0}^{r_2-1} \mathrm{r}_{d_1}(-\lambda\ell)\mathfrak{f}_{21} + r_2\mathfrak{f}_{22} + \sum_{\ell=0}^{r_2-1} \bigl[\mathrm{r}_{d_1}(-\lambda\ell)\!+\!\mathrm{r}_{d_1}(-\lambda)\ge d_1\bigr]\gamma_1\\
&\phantom{\mathcal{V}_{22}} = d_1(d-1)\lambda \mathfrak{f}_{21} + r_2\mathfrak{f}_{22}+ \left(r_2-\frac{r_2}{d}\right)\gamma_1\\
&\phantom{\mathcal{V}_{22}} = d_1(d-1)\lambda \mathfrak{f}_{21} + r_2\mathfrak{f}_{22}+  2(d_1-\lambda)\gamma_1.
\end{align*}
Since, clearly $\mathcal{W}_{ij} = 0$ for all $i,j$, the formula for $T_2(\gamma_1,\gamma_2,\mathfrak{f}_{11},\mathfrak{f}_{12},\mathfrak{f}_{21}, \mathfrak{f}_{22})$ follows immediately from these facts.
	
Next we compute $T_3(\gamma_1,\gamma_2,\mathfrak{f}_{11},\mathfrak{f}_{12},\mathfrak{f}_{21},\mathfrak{f}_{22})$. Using~\cite{GGV2}*{Definition~2.25} and \eqref{def de punto} we obtain that for $i\ne j$ and $k$,
$$
F_1(a_i,a_j,e_k)=f_1(a_j,a_i\cdot e_k)+f_1(a_i,e_k) = \begin{cases} \mathfrak{f}_{j1}+\mathfrak{f}_{i1} &\text{if $k=1$,}\\ f_1(a_1,-\lambda e_1+e_2)+ \mathfrak{f}_{22} &\text{if $i=2$ and $k=2$,}\\  f_1(a_2,-2\lambda e_1+e_2)+ \mathfrak{f}_{12} &\text{if $i=1$ and $k=2$.}\end{cases}
$$
The formula for $T_3(\gamma_1,\gamma_2,\mathfrak{f}_{11},\mathfrak{f}_{12},\mathfrak{f}_{21},\mathfrak{f}_{22})$ follows from this fact, using that, by~\eqref{formula general para f''1} and~\eqref{formula general para f''2},
$$
f_1(a_1,-\lambda e_1+e_2) = (d_1-\lambda) \mathfrak{f}_{11} + \mathfrak{f}_{12} + \gamma_1\quad\text{and}\quad f_1(a_2,-2\lambda e_1+e_2) = (d_1-2\lambda) \mathfrak{f}_{21} + \mathfrak{f}_{22} + \gamma_1.
$$
Finally, the formula for $S(t_1,t_2)$ follows straightforwardly from~\eqref{S'''}.
\end{proof}

\begin{corollary}\label{pendorcho} Assume that we are under the hypotheses of Proposition~\ref{prop fin}, that $(1-\lambda)dI = 0$ and that there exists an endomorphism $b$ of $I$, such that $\lambda(d-1)b=\Ide_I$. Then
$$
\ho_{\blackdiamond}^2(H^{d_1d_1}_{\lambda},I)\simeq \frac{I}{d_1 I} \oplus I_{r_1} \oplus I_{d_1}.
$$
\end{corollary}

\begin{proof} Consider the invertible map $\phi \colon I^6\to I^6$ given by
$$
\phi(\gamma_1,\gamma_2,\mathfrak{f}_{11},\mathfrak{f}_{12},\mathfrak{f}_{21},\mathfrak{f}_{22}) \coloneqq \left(\gamma_1-d_1b\mathfrak{f}_{22},\gamma_2, \mathfrak{f}_{11},\mathfrak{f}_{12}+(1-\lambda b)\mathfrak{f}_{22},\mathfrak{f}_{21},\mathfrak{f}_{22}\right).
$$
Set $\wt{T}_i=T_i \phi$ and $\wt{S}= \phi^{-1} S$, where the $T_i$'s and $S$ are as in Proposition~\ref{prop fin}. By~\eqref{formula00}, we have
\begin{equation}\label{homo10''}
\ho_{\blackdiamond}^2(H^{d_1d_2}_0,I) \simeq \frac{\ker T_1\cap\ker T_2\cap\ker T_3 }{\ima S} \simeq \frac{\ker \wt{T}_1\cap\ker \wt{T}_2\cap\ker \wt{T}_3 }{\ima \wt{S}},
\end{equation}
where the last isomorphism is induced by $\phi^{-1}$. A direct computation shows that
\begin{align*}
&\wt{T}_1(\gamma_1,\gamma_2,\mathfrak{f}_{11},\mathfrak{f}_{12},\mathfrak{f}_{21},\mathfrak{f}_{22})=\left(d_1\mathfrak{f}_{11},d_1\mathfrak{f}_{12} + (d_1-2\lambda)\gamma_1,d_1\mathfrak{f}_{21},(d_1-1)\gamma_1 + (1-\lambda)d_1b\mathfrak{f}_{22}\right),\\
&\wt{T}_2(\gamma_1,\gamma_2,\mathfrak{f}_{11},\mathfrak{f}_{12},\mathfrak{f}_{21},\mathfrak{f}_{22}) =\left(r_1\mathfrak{f}_{11}, r_1(r_1-\lambda) \mathfrak{f}_{11} + r_1\mathfrak{f}_{12} + \left(r_1-\lambda\right)\gamma_1,r_2 \mathfrak{f}_{21},d_1(d_1-\lambda) \mathfrak{f}_{21} + 2(d_1-\lambda)\gamma_1 \right),\\
& \wt{T}_3(\gamma_1,\gamma_2,\mathfrak{f}_{11},\mathfrak{f}_{12},\mathfrak{f}_{21},\mathfrak{f}_{22}) = (0,(d_1-\lambda)\mathfrak{f}_{11} - (d_1-2\lambda)\mathfrak{f}_{21})
\shortintertext{and}
&\wt{S}(t_1,t_2)= (0,d_1 t_2,0, 0,0,(\lambda-d_1)t_1).
\end{align*}
Clearly $(\gamma_1,\gamma_2,\mathfrak{f}_{11},\mathfrak{f}_{12},\mathfrak{f}_{21},\mathfrak{f}_{22})\in \ker\wt{T}_1\cap\ker\wt{T}_2\cap\ker\wt{T}_3$ if and only if
$$
r_1\mathfrak{f}_{11} = d_1\mathfrak{f}_{21} = 2(d_1-\lambda)\gamma_1 = r_1\mathfrak{f}_{12} + \left(r_1-\lambda\right)\gamma_1 = (d_1-1)\gamma_1 + (1-\lambda)d_1b\mathfrak{f}_{22}= 2\lambda \mathfrak{f}_{21}-\lambda \mathfrak{f}_{11} = 0
$$
But since the multiplication in $I$ by $\lambda$ and $d-1$ are bijections, $r_1-\lambda$ and $d_1-1$ are odd this is equivalent to
$$
r_1\mathfrak{f}_{11} = d_1\mathfrak{f}_{21} = 2\gamma_1 = r_1\mathfrak{f}_{12} + \gamma_1 = \gamma_1 + (1-\lambda)d_1b\mathfrak{f}_{22}= 2 \mathfrak{f}_{21}-\mathfrak{f}_{11} = 0
$$
Como $(1-\lambda)dI = 0$, these conditions reduce to
$$
r_1\mathfrak{f}_{11} = d_1\mathfrak{f}_{21} = \gamma_1 = r_1\mathfrak{f}_{12} = 2 \mathfrak{f}_{21}-\mathfrak{f}_{11} = 0.
$$
Hence $(\gamma_1,\gamma_2,\mathfrak{f}_{11},\mathfrak{f}_{12},\mathfrak{f}_{21},\mathfrak{f}_{22})\in \ker\wt{T}_1\cap\ker\wt{T}_2\cap\ker\wt{T}_3$ if and only if
$$
\gamma_1 = 0,\quad \gamma_2\in I,\quad \mathfrak{f}_{12}\in I_{r_1},\quad \mathfrak{f}_{21}\in I_{d_1},\quad \mathfrak{f}_{22}\in I\quad\text{and}\quad \mathfrak{f}_{11} = 2\mathfrak{f}_{21}.
$$
Moreover $\ima\wt{S} = 0\oplus d_1I\oplus 0\oplus 0\oplus 0\oplus I$, because the multiplication in $I$ by $\lambda-d_1$ is a bijection. Hence
$$
\ho_{\blackdiamond,\Yleft}^2(H^{d_1d_1}_{\lambda},I)\simeq \frac{\ker T_1\cap\ker T_2\cap\ker T_3 }{\ima S}\simeq \frac{\ker \wt{T}_1\cap\ker \wt{T}_2\cap\ker \wt{T}_3 }{\ima \wt{S}} \simeq \frac{I}{d_1 I} \oplus I_{r_1} \oplus I_{d_1},
$$
as desired.
\end{proof}

\begin{remark}\label{pepi} Assume that we are under the hypothesis of Corollary~\ref{pendorcho}. The map $\psi\colon I^3\to I^6$, defined by
\begin{equation}\label{pepito}
\psi(\gamma_2,\mathfrak{f}_{12},\mathfrak{f}_{21})\coloneqq (0,\gamma_2,2\mathfrak{f}_{21},\mathfrak{f}_{12},\mathfrak{f}_{21},0),
\end{equation}
induces an~iso\-mor\-phism
$$
\frac{I}{d_1 I} \oplus I_{r_1} \oplus I_{d_1} \longrightarrow \frac{\ker T_1\cap \ker T_2\cap \ker T_3}{\ima S},
$$
where $T_1$, $T_2$, $T_3$ and $S$ are as in Proposition~\ref{prop fin}. Let $Z\subseteq I\oplus I_{r_1} \oplus I_{d_1}$ be such that the canonical epimorphism $Z\to \frac{I}{d_1 I} \oplus I_{r_1} \oplus I_{d_1}$ is a bijection, and let~$T$ be as in~\eqref{formula0}, so that
$$
T(\gamma_1,\gamma_2,\mathfrak{f}_{11},\mathfrak{f}_{12},\mathfrak{f}_{21},\mathfrak{f}_{22}) = (\alpha_{\gamma_1,\gamma_2}, f_{\mathfrak{f}_{11},\mathfrak{f}_{12},\mathfrak{f}_{21},\mathfrak{f}_{22},\gamma_1,\gamma_2}).
$$
Clear\-ly~$T(\psi(Z))$ is a complete family of representatives of~$2$-cocycles modulo coboundaries in $\wh{\mathcal{C}}_N^*(H^{2,d_2}_1,I)$. Hence, computing $T$ becomes essential. Since $\alpha_{\gamma_1,\gamma_2}$ is already known from~\eqref{alpha_{gamma_1,gamma_2}} and we have~\eqref{pepito}, for this it suffices to compute
$$
f_{2\mathfrak{f}_{21},\mathfrak{f}_{12},\mathfrak{f}_{21},0,0,\gamma_2}\quad\text{with $\gamma_2\in I$, $\mathfrak{f}_{12}\in I_{r_1}$ and $\mathfrak{f}_{21}\in I_{d_1}$.}
$$
\end{remark}

\begin{proposition}\label{Ultcalculo de f} Assume that we are under the hypothesis of the previous remark and let $f\coloneqq f_{2\mathfrak{f}_{21}, \mathfrak{f}_{12},\mathfrak{f}_{21},0,0,\gamma_2}$, where $\gamma_2\in I$, $\mathfrak{f}_{12}\in I_{r_1}$ and $\mathfrak{f}_{21}\in I_{d_1}$. We have:
$$
f(a_1^{\times \ell_1}\times a_2^{\times \ell_2},(u,v)) = \left(2\ell_1 u + \ell_2 u + 4\ell_1 \lambda v - 4\ell_1^2 \lambda v - \lambda \ell_1 \ell_2v - \lambda \binom{\ell_2}{2} v\right)\mathfrak{f}_{21} + \ell_1 v \mathfrak{f}_{12},
$$
where $0\le\ell_1,u<d_1$ and $0\le\ell_2,v<d_2$.
\end{proposition}

\begin{proof} By~\eqref{def de punto} and~\eqref{potencia}, we obtain
\begin{equation*}
a_i^{\times j}\cdot (u,v) = \begin{cases} (u- 2 \lambda jv,v) & \text{if $i=1$,}\\ (u-\lambda jv,v) & \text{if $i=2$.}\end{cases}
\end{equation*}
Hence, by~\cite{GGV2}*{Proposition~2.37},
$$
f(a_1^{\times \ell_1}\times a_2^{\times \ell_2},(u,v)) = f(a_1^{\times \ell_1},(u,v)) + f(a_2^{\times \ell_2},(u-\lambda\ell_1 v,v)).
$$
Consequently, by~\cite{GGV2}*{Proposition~2.12},
\begin{equation}\label{pepitito12}
f(a_1^{\times \ell_1}\times a_2^{\times \ell_2},(u,v)) = \sum_{j=0}^{\ell_1-1} f(a_1,(u-2\lambda j,v)) + \sum_{j=0}^{\ell_2-1} f(a_2,(u-\lambda (\ell_1+j)v,v)).
\end{equation}
Moreover, by~\eqref{formula general para f''1} and~\eqref{formula general para f''2},
\begin{equation*}
f(a_1,(u-2\lambda j v,v)) = 2(u-2\lambda j v)\mathfrak{f}_{21}+v\mathfrak{f}_{12}\quad\text{and}\quad f(a_2,(u-\lambda (\ell_1+j)v,v)) = (u-\lambda (\ell_1+j)v)\mathfrak{f}_{21}.
\end{equation*}
The formula in the statement follows easily from this and~\eqref{pepitito12}.
\end{proof}

\end{document}